\documentclass[11pt,a4paper]{article}
\usepackage[margin=1in,a4paper]{geometry}
\usepackage{commath}
\usepackage{amsfonts}
\usepackage{amsthm}
\usepackage{amsmath}
\usepackage{amssymb}
\usepackage{enumerate}
\usepackage[longnamesfirst]{natbib}
\usepackage[bookmarks,colorlinks = true,linkcolor = blue,urlcolor  = blue,citecolor =
blue,anchorcolor = blue]{hyperref}
\usepackage{booktabs}
\usepackage{xr}
\usepackage{graphicx}
\usepackage{multirow}
\usepackage{float}
\usepackage{color}

\newtheorem{lemma}{Lemma}[section]

\newtheorem{assumption}{Assumption}[section]

\newtheorem{thm}{Theorem}[section]
\newtheorem{corollary}{Corollary}[section]
\numberwithin{equation}{section}
\DeclareMathOperator*{\ve}{vec}
\DeclareMathOperator*{\vech}{vech}
\DeclareMathOperator*{\tr}{tr}
\DeclareMathOperator*{\var}{var}
\DeclareMathOperator*{\cov}{cov}

\begin{document}

\title{Estimation of the Kronecker Covariance Model by Quadratic Form\thanks{%
We thank participants at the Celebration of Peter Phillips' Forty Years at
Yale held at Yale University on October 19th-20th, 2018, and at the CUHK
Workshop on Econometrics 2019 for helpful comments. We thank Liang Jiang,
Chen Wang, Tengyao Wang for useful discussions. Detailed comments from the
Co-Editor and two anonymous referees also greatly improved the article. Any
remaining errors are our own.}}
\author{Oliver B. Linton\thanks{%
Faculty of Economics, Austin Robinson Building, Sidgwick Avenue, Cambridge,
CB3 9DD. Email: \texttt{obl20@cam.ac.uk.}Thanks to the Cambridge INET and
the Keynes Fund for financial support. } \\
University of Cambridge \and Haihan Tang\thanks{%
Corresponding author. Fanhai International School of Finance, Fudan
University. 220 Handan Road, Yangpu District, Shanghai, 200433, China.
Email: \texttt{hhtang@fudan.edu.cn}. Haihan Tang is sponsored by the
National Natural Science Foundation of China (grant number 71903034) and
Shanghai Pujiang Program (grant number 2019PJC015).} \\
Fudan University}
\date{\today}
\maketitle

\begin{abstract}
\noindent We propose a new estimator, the quadratic form estimator, of the
Kronecker product model for covariance matrices. We show that this estimator
has good properties in the large dimensional case (i.e., the cross-sectional
dimension $n$ is large relative to the sample size $T$). In particular, the
quadratic form estimator is consistent in a relative Frobenius norm sense
provided $\log^3n/T\to 0$. We obtain the limiting distributions of the
Lagrange multiplier (LM) and Wald tests under both the null and local
alternatives concerning the mean vector $\mu$. Testing linear restrictions
of $\mu$ is also investigated. Finally, our methodology is shown to perform
well in finite sample situations both when the Kronecker product model is
true, and when it is not true.
\end{abstract}

\textit{Some} \textit{key words: }Covariance matrix; Kronecker product;
Quadratic form; Lagrange multiplier test; Wald test

\section{Introduction}

Covariance matrices are of great importance in many fields. In finance, they
are a key element in portfolio choice and risk management (\cite%
{markowitz1952}). In psychology, scholars have long assumed that some
observed variables are related to certain latent traits through a factor
model, and then use the covariance matrix of the observed variables to
deduce properties of the latent traits. In econometrics, covariance matrices
often appear in test statistics representing the sampling variability of a
vector of parameter estimates. \cite{anderson1984} is a classic statistical
reference that studies estimation of and hypothesis testing about covariance
matrices in the low dimensional case (i.e., the dimension of the covariance
matrix, $n$, is small compared with the sample size $T$).

There are many new methodological approaches to covariance and precision
matrix estimation in the large dimensional case (i.e., $n$ is large compared
with $T$);\footnote{%
Some studies have made a distinction between the large dimensional case and
the high dimensional case (\cite{hafnerlintontang2018}). We no longer make
this distinction in this article. As long as $n$ is large relative to $T$,
regardless of $n$ exceeding $T$, we call it the large dimensional case.}
see, e.g., \cite{ledoitwolf2003}, \cite{bickellevina2008}, \cite%
{fanfanlv2008}, \cite{ledoitwolf2012}, \cite{fanliaomincheva2013}, and \cite%
{ledoitwolf2015}. \cite{fanliaoliu2016econometricsJ} gave an excellent
account of the recent developments in theory and practice of estimating
large dimensional covariance matrices. The usual approaches include: to
impose some sparsity on the covariance matrix, meaning that many elements of
the covariance matrix are assumed to be zero or small, thereby reducing the
number of parameters to be estimated; or at least to "shrink" towards a
sparse matrix, or to use a factor model which reduces the dimensionality of
the parameter space. Most of this literature assumes i.i.d. data.

We consider the problem of estimating a large covariance matrix $\Sigma .$
We impose a model structure that reduces the effective dimensionality. In
particular, we consider the Kronecker product model. Let $n=n_{1}\times
\cdots \times n_{v}$, where $n_{j}\in \mathbb{Z}$ and $n_{j}\geq 2$ for $%
j=1,\ldots ,v$. We suppose that 
\begin{equation}
\Sigma =\sigma ^{2}\times \Sigma _{1}\otimes \cdots \otimes \Sigma _{v},
\label{model1}
\end{equation}%
where $\Sigma _{j}$ is an $n_{j}\times n_{j}$ unknown covariance matrix
satisfying $\mathrm{tr}(\Sigma _{j})=n_{j}$ for $j=1,\ldots ,v,$ and $%
0<\sigma ^{2}<\infty $ is a scalar parameter.

Kronecker product models arise naturally from \textit{multiway data} (c.f. 
\cite{kroonenberg2008}). Multiway data are a generalization of two-way or
three-way data that are widely encountered in social science. For example,
the scores on 3 subjects (mathematics, English and music) of 50 students
observed over 10 years are three-way data, the "ways" being subjects,
students and years. Let $w_{i,j,t}$ denote the score of subject $i$ of
student $j$ in year $t$. To model $w_{i,j,t}$, one could use an interactive
effects model similar to \cite{bai2009}: 
\begin{equation*}
w_{i,j,t}=\mu _{i,j}+\gamma _{i,t}f_{j,t},\qquad i=1,2,3,\quad j=1,\ldots
,50,\quad t=1,\ldots ,10
\end{equation*}%
where $\mu _{i,j}$ is the subject-student specific mean, while $\gamma
_{i,t} $ and $f_{j,t}$ are the subject-time specific and student-time
specific effects, respectively. Stacking all the observations $\{w_{i,j,t}\}$
of year $t$ into a $150\times 1$ column vector $y_{t}$, we have $y_{t}=\mu
+\gamma _{t}\otimes f_{t}$, where $\mu $ is the $150\times 1$ mean vector
containing stacked $\{\mu _{i,j}\}$, $\gamma _{t}=(\gamma _{1,t},\gamma
_{2,t},\gamma _{3,t})^{\intercal }$, and $f_{t}=(f_{1,t},\ldots
,f_{50,t})^{\intercal }$. Suppose that $\gamma _{t}$ is a random vector
independent of $f_{t}$, and that both are mean-zero and stationary in time.
Then, 
\begin{equation*}
\mathbb{E}[(y_{t}-\mu )(y_{t}-\mu )^{\intercal }]=\mathbb{E}[\gamma
_{t}\gamma _{t}^{\intercal }]\otimes \mathbb{E}[f_{t}f_{t}^{\intercal }].
\end{equation*}%
In this case the covariance matrix of $y_{t}$ is a Kronecker product of two
sub-matrices, which describe the subject specific and individual specific
dependencies.

Extending the idea to multiway data, one might think of a typical equity
portfolio constructed by intersections of 5 size quintiles, 5 book-to-market
ratio quintiles, and 10 industries, in the spirit of \cite{famafrench1993},
over a number of years, as four-way data: sizes $\times$ B/P ratios $\times$
industries $\times$ years. Situations in which higher-way data are collected
are also on the increase. For example, electroencephalography (EEG), a
non-invasive way of detecting structural abnormalities such as brain tumors,
also provide multiway data, such as EEG bands $\times$ patients $\times$
leads $\times$ doses $\times$ time $\times$ task conditions (\cite%
{estienneetal2001}).

Consider $(v+1)$-way data $w_{i_{1},i_{2},\ldots ,i_{v},t}$, where $%
i_{j}=1,\ldots ,n_{j}$ for $j=1,\ldots ,v$ and $t=1,\ldots ,T$. We use
subscript $t$ to denote the $(v+1)$th way of the data in the hope to \textit{%
broadly} interpret the $(v+1)$th way as "time", $T$ as the sample size, all
other ways as the "cross-section", and $n:=n_{1}\times \cdots \times n_{v}$
as the cross-sectional dimension. In other words, the $(v+1)$th way of the
data need not correspond to the time dimension, should the multiway data
contain such a dimension. In the rest of the article, we shall no longer
stress this distinction. Suppose that $w_{i_{1},i_{2},\ldots ,i_{v},t}=\mu
_{i_{1},i_{2},\ldots ,i_{v}}+\varepsilon _{i_{1},t}^{1}\varepsilon
_{i_{2},t}^{2}\cdots \varepsilon _{i_{v},t}^{v},$ where $i_{j}=1,\ldots
,n_{j}$ for $j=1,\ldots ,v,$ and $t=1,\ldots ,T$. Equivalently, in the
stacked form 
\begin{equation*}
y_{t}:=(w_{1,1,\ldots ,1,t},\ldots ,w_{n_{1},n_{2},\ldots
,n_{v},t})^{\intercal }=\mu +\varepsilon _{t}^{1}\otimes \varepsilon
_{t}^{2}\otimes \cdots \otimes \varepsilon _{t}^{v},
\end{equation*}%
where $\mu $ is the stacked mean vector, $\varepsilon _{t}^{j}:=(\varepsilon
_{1,t}^{j},\ldots ,\varepsilon _{n_{j},t}^{j})^{\intercal }$ is an $%
n_{j}\times 1$ mean-zero random vector with covariance matrix $\mathbb{E}%
[\varepsilon _{t}^{j}\varepsilon _{t}^{j\intercal }]$ for all $t$ for $%
j=1,\ldots ,v$. If $\varepsilon _{t}^{1},\ldots ,\varepsilon _{t}^{v}$ are
mutually independent for all $t$, then 
\begin{equation*}
\mathbb{E}[(y_{t}-\mu )(y_{t}-\mu )^{\intercal }]=\mathbb{E}[\varepsilon
_{t}^{1}\varepsilon _{t}^{1\intercal }]\otimes \mathbb{E}[\varepsilon
_{t}^{2}\varepsilon _{t}^{2\intercal }]\otimes \cdots \otimes \mathbb{E}%
[\varepsilon _{t}^{v}\varepsilon _{t}^{v\intercal }].
\end{equation*}%
We hence see that the covariance matrix of $y_{t}$ is a Kronecker product of 
$v$ sub-matrices.

Recent work on Kronecker product models for multiway data include \cite%
{hoff2011}, \cite{hoff2015}, \cite{hoff2016} etc. Kronecker product models
have also been considered in the psychometric literature (\cite%
{campbelloconnell1967}, \cite{swain1975}, \cite{cudeck1988}, \cite%
{verheeswansbeek1990} etc). In the spatial literature, there are a number of
studies that consider Kronecker product models for the correlation matrix of
a random field (\cite{lohlam2000}). \cite{robinson1998} and \cite%
{hidalgoschafgans2017} exploited separable error covariance matrix
structures to develop inference methods without the need for smoothing.

These literatures have all focussed on the low dimensional case. \cite%
{hafnerlintontang2018} were the first to study Kronecker product models in
the large dimensional case. The proper framework for studying the large
dimensional case is the joint limit setting developed by \cite%
{phillipsmoon1999} in which $n$ and $T$ tend to infinity simultaneously.%
\footnote{%
Peter Phillips has made some fundamental contributions to large dimensional
analysis. \cite{phillipsmoon1999} provided three asymptotic frameworks for
analysing double-index ($n,T$) processes: sequential limit framework (e.g., $%
n\rightarrow \infty $ followed by $T\rightarrow \infty $), diagonal path
limit framework (i.e., both $n$ and $T$ pass to infinity along some specific
diagonal in the two dimensional array), and joint limit framework (i.e., $%
n,T\rightarrow \infty $ simultaneously). In particular, they provided a
central limit theorem in joint limit framework for double-index processes
(Theorem 2 of \cite{phillipsmoon1999}). However, the Lindeberg condition of
that theorem is perhaps difficult to verify in practice. In Section \ref{sec
AppendixB}, we provide a variant (Theorem \ref{thmdoubleindexCLT}), which
relies on a Lyapounov's condition. Moreover, the variant allows the central
limit theorem to kick in from either the cross-sectional or time dimension.}
Since $n$ tends to infinity, there are two main cases when considering (\ref%
{model1}): (a) $\{n_{j}\}_{j=1}^{v}$ are all fixed while $v\rightarrow
\infty $; (b) $n_{j}\rightarrow \infty $ for at least some $j$ while $v$ is
fixed. Case (a) corresponds to practical situations where the data have a
large number of ways but in each way the number of entities is small; case
(b) often corresponds to, say, three-way or four-way data in which at least
one way has a large number of entities. The methodologies developed in \cite%
{hafnerlintontang2018} and this article are perfectly geared for case (a) in
the sense that (\ref{model1}) is \textit{correctly} specified for the data.

We do not analyse case (b) theoretically, but our estimation and inference
procedures can in principle be applied to case (b) also, but the theory will
require more work and stronger restrictions on the relationship between $n$
and $T.$ For example, if $v=2$ and $n_{1}=n_{2}=\sqrt{n},$ then the
sub-matrices $\Sigma _{1},\Sigma _{2}$ each contain order $n$ unknown
quantities. If $n/T\rightarrow 0$ fast enough, then we may show some
consistency of our estimators of the sub-matrices $\Sigma _{1},\Sigma _{2}.$
On the other hand, if this rate condition is not satisfied, one could
combine the separable structure (i.e., the Kronecker product) with sparsity
restrictions on the sub-matrices. This has been investigated in the
literature. Other approaches have been considered in \cite{akdemirgupta2011}%
, \cite{hoff2011}, and \cite{hoff2015}. Henceforth, when we say the
Kronecker product model (\ref{model1}), we implicitly mean case (a).

The Kronecker product model leads to substantial dimension reduction even
though it need not be sparse in the sense of (2.1) of \cite%
{fanliaoliu2016econometricsJ}. \cite{hafnerlintontang2018} showed that the
matrix logarithm of a Kronecker product covariance or correlation matrix is
a sparse matrix (with $O(\log n)$ unknown quantities) and the logarithmic
operator converts the multiplicative Kronecker product structure into an
additive one. Therefore, the logarithm of a Kronecker product covariance or
correlation matrix is a linear function of a much "smaller" vector of
unknown quantities. They used this to develop a closed-form estimator, they
established its consistency and provided a central limit theorem (CLT).
However, their results require strong, albeit sufficient but not necessary,
conditions; in particular they obtained Frobenius norm consistency of the
estimator under a condition that at least $n/T\rightarrow 0,$ which is very
restrictive. On the contrary, other methodologies typically achieve \textit{%
average} Frobenius norm consistency provided $s\log n/T\rightarrow 0,$ where 
$s$ is some sparsity index (e.g., see \cite{bickellevina2008} Theorem 2 with 
$q=0$).\footnote{\textit{Average} Frobenius norm means dividing a Frobenius
norm by $\sqrt{n}$, while \textit{relative} Frobenius norm means dividing a
Frobenius norm by the Frobenius norm of a target matrix, say, the unknown
covariance matrix. These two concepts are similar, but not exactly the same.}

In this article, we relax the rate restriction on $n$ imposed by \cite%
{hafnerlintontang2018} and allow $n$ to be possibly larger than $T$. We
propose a new covariance matrix estimator called the \textit{quadratic form}
estimator based on the Kronecker product model. Our estimator averages
elements of the sample covariance matrix, so we obtain a rate improvement by
averaging. In particular, under a cross-sectional weak dependence condition,
the quadratic form estimator achieves \textit{relative} Frobenius norm
consistency provided $\log^3n/T\to 0$. Moreover, this method automatically
produces a symmetric and positive definite covariance matrix estimator,
unlike some of the sparsifying methods considered by \cite%
{fanliaoliu2016econometricsJ}.

We apply our methodology to a concrete testing problem; we consider the null
hypothesis $H_{0}:\mu =\mu _{0}$, where $\mu $ is the mean of the large
dimensional data $y_{t}$ and $\mu _{0}$ is some known vector. One practical
example would be that $y_{t}$ corresponds to differences between treated and
controlled groups and we want to test whether the mean cross-sectional
differences are different from zero. We define the Lagrange multiplier (LM)
and Wald test statistics based on our estimated precision matrix and
establish their asymptotic distributions under both null and local
alternatives of the form $H_{1}:\mu =\mu _{0}+\theta /\sqrt{T}$ for some
vector $\theta $. We also provide two results regarding testing linear
restrictions of $\mu $.

We compare our estimation and testing methods with \cite{ledoitwolf2004}'s
linear shrinkage estimator and \cite{ledoitwolf2017}'s direct nonlinear
shrinkage estimator in Monte Carlo simulations. Our methods perform very
well in moderate-sized samples. In fact, they work well even in situations
where a Kronecker product model is \textit{misspecified} for a covariance
matrix. 


The rest of the article is structured as follows. In Section \ref{sec model}
we discuss the model and identification while in Section \ref{sec estimation}
we propose the quadratic form estimator. Section \ref{sec asymptotic
properties} gives the rate of convergence for the quadratic form estimator.
In Section \ref{sec test statistics} we define the LM and Wald test
statistics and establish their asymptotic distributions under both null and
local alternatives. We also consider testing linear restrictions of $\mu$.
Section \ref{sec simulation} conducts Monte Carlo simulations comparing our
approach with Ledoit and Wolf estimators. Section \ref{sec conclusion}
concludes. All the major proofs are put in the Appendix while auxiliary
lemmas and theorems are in Section \ref{sec AppendixB}.

\subsection{Notation}

Let $A$ be an $m\times n$ matrix. Let $\ve A$ denote the vector obtained by
stacking the columns of $A$ one underneath the other. The \textit{%
commutation matrix} $K_{m,n}$ is an $mn\times mn$ \textit{orthogonal} matrix
which translates $\ve A$ to $\ve(A^{\intercal})$, i.e., $\ve%
(A^{\intercal})=K_{m,n}\ve(A)$. If $A$ is a symmetric $n\times n$ matrix,
its $n(n-1)/2$ supradiagonal elements are redundant in the sense that they
can be deduced from symmetry. If we eliminate these redundant elements from $%
\ve A$, we obtain a new $n(n+1)/2\times1$ vector, denoted $\vech A$. They
are related by the full-column-rank, $n^{2}\times n(n+1)/2$ \textit{%
duplication matrix} $D_{n}$: $\ve A=D_{n}\vech A$. Conversely, $\vech %
A=D_{n}^{+}\ve A$, where $D_{n}^{+}$ is $n(n+1)/2\times n^{2}$ and the
Moore-Penrose generalized inverse of $D_{n}$. In particular, $%
D_{n}^{+}=(D_{n}^{\intercal}D_{n})^{-1}D_{n}^{\intercal}$ because $D_{n}$ is
full-column-rank.

For $x\in\mathbb{R}^{n}$, let $\Vert x\Vert_{2}:=\sqrt{%
\sum_{i=1}^{n}x_{i}^{2}}$ and $\|x\|_{\infty}:=\max_{1\leq i\leq n}|x_{i}|$
denote the Euclidean ($\ell_{2}$) norm and the element-wise maximum ($%
\ell_{\infty}$) norm, respectively. 
Let $\lambda_{\max}(\cdot)$ and $\lambda_{\min}(\cdot) $ denote the maximum
and minimum eigenvalues of some real symmetric matrix, respectively. For any
real $m\times n$ matrix $A=(a_{i,j})_{1\leq i\leq m, 1\leq j\leq n}$, let $%
\Vert A\Vert_{F}:=[\text{tr}(A^{\intercal}A)]^{1/2}\equiv[\text{tr}%
(AA^{\intercal })]^{ 1/2}\equiv\Vert\ve A\Vert_{2}$, $\|A\|_{1}:=%
\sum_{i=1}^{m}\sum_{j=1}^{n}|a_{i,j}|$, $\Vert
A\Vert_{\ell_{2}}:=\max_{\Vert x\Vert_{2}=1}\Vert Ax\Vert_{2}\equiv\sqrt{%
\lambda_{\max}(A^{\intercal}A)}$, $\|A\|_{\ell_1}:=\max_{1\leq j\leq
n}\sum_{i=1}^{m}|a_{i,j}|$, and $\|A\|_{\ell_{\infty}}:=\max_{1\leq i\leq
m}\sum_{j=1}^{n}|a_{i,j}|$ denote the Frobenius ($\ell_{2}$) norm, $\ell_{1}$
norm, and spectral norm ($\ell_{2}$-operator norm), maximum column sum
matrix norm ($\ell_1$-operator norm), and maximum row sum matrix norm ($%
\ell_{\infty}$-operator norm) of $A$, respectively. Note that $%
\|\cdot\|_{\infty}$ can also be applied to matrix $A$, i.e., $%
\|A\|_{\infty}=\max_{1\leq i\leq m,1\leq j\leq n}|a_{i,j}|$; however $%
\|\cdot\|_{\infty}$ is not a matrix norm so it does not have the
submultiplicative property of a matrix norm.

Landau (order) notation in this article, unless otherwise stated, should be
interpreted in the sense that $n,T\rightarrow \infty $ simultaneously. An 
\textit{absolute} positive constant refers to a constant independent of
anything which is a function of $n$ and/or $T$. 
We write $a\asymp b$ if there exist absolute constants $0< c_1\leq c_2$ such
that $c_1b\leq a\leq c_2b$. For real numbers $a,b$ let $a \vee b$ denote $%
\max(a,b)$.

\section{The Model and Identification}
\label{sec model}

We now directly work with the high-level $n$-dimensional random vector $%
y_{t} $ with $\mu :=\mathbb{E}y_{t}$ and $\Sigma :=\mathbb{E}[(y_{t}-\mu
)(y_{t}-\mu )^{\intercal }]$ for every $t$. In particular, $\Sigma $ takes
the form of (\ref{model1}). For each $j$, $\Sigma _{j}$ contains $%
n_{j}(n_{j}+1)/2-1$ (unrestricted) parameters. In total, model (\ref{model1}%
) contains $\sum_{j=1}^{v}n_{j}(n_{j}+1)/2-(v-1)$ unknown parameters. This
model is the same as considered in \cite{hafnerlintontang2018} except that
we make a different identifying restriction. The implied form for $\Sigma
^{-1}$ is also Kronecker, i.e., $\Sigma ^{-1}=\sigma ^{-2}\times \Sigma
_{1}^{-1}\otimes \cdots \otimes \Sigma _{v}^{-1}.$

We show that model (\ref{model1}) is indeed identified. First, the parameter 
$\sigma $ is identified because 
\begin{equation*}
\mathrm{tr}(\Sigma )=\sigma ^{2}\times \mathrm{tr}(\Sigma _{1}\otimes \cdots
\otimes \Sigma _{v})=\sigma ^{2}\times \mathrm{tr}(\Sigma _{1})\times \cdots
\times \mathrm{tr}(\Sigma _{v})=\sigma ^{2}n,
\end{equation*}%
whence we have $\sigma ^{2}=\mathrm{tr}(\Sigma )/n$. We next consider
identification of the remaining parameters based on the \textit{partial
trace operator} (\cite{filipiakkleinvojtkova2018}). Suppose that an $n\times
n$ matrix $A$ can be written in terms of $n_{1}\times n_1$ blocks of $%
n_{-1}\times n_{-1}$ dimensional matrices $A _{-1;i,j}$, where $%
n_{-1}:=n/n_{1}$; that is, 
\begin{equation}  \label{eqn A}
A =\left( 
\begin{array}{ccc}
A_{-1;1,1} & \cdots & A_{-1;1,n_{1}} \\ 
& \ddots & \vdots \\ 
&  & A_{-1;n_{1},n_{1}}%
\end{array}%
\right) .
\end{equation}%
Then the partial trace operator $\mathrm{PTR}_{n_{1}}:\mathbb{R}^{n\times n}%
\mathbb{\rightarrow R}^{n_{1}\times n_{1}}$ is defined as follows:%
\begin{equation*}
\mathrm{PTR}_{n_{1}}(A )=\left( 
\begin{array}{ccc}
\mathrm{tr}(A_{-1;1,1}) & \cdots & \mathrm{tr}(A_{-1;1,n_{1}}) \\ 
& \ddots & \vdots \\ 
&  & \mathrm{tr}(A_{-1;n_{1},n_{1}})%
\end{array}%
\right) .
\end{equation*}

Consider model (\ref{model1}), and let $\Sigma _{-1}:=\Sigma _{2}\otimes
\cdots \otimes \Sigma _{v}$. Define the $n_{1}\times n_{1}$ matrix $%
d^{(1)}:= \mathrm{PTR}_{n_{1}}(\Sigma )=\sigma ^{2}\mathrm{tr}(\Sigma
_{-1})\times \Sigma _{1}.$ Then $\Sigma _{1}=d^{(1)}/(\mathrm{tr}%
(d^{(1)})/n_{1}).$ According to Definition 1.1(ii) of \cite%
{filipiakkleinvojtkova2018}, $\mathrm{PTR}_{n_{1}}(\Sigma )=\sum_{\ell
=1}^{n_{-1}}(I_{n_{1}}\otimes e_{\ell ,n_{-1}}^{\intercal })\Sigma
(I_{n_{1}}\otimes e_{\ell ,n_{-1}}),$ where $e_{\ell ,n_{-1}}$ is the $%
n_{-1}\times 1$ elementary vector with one in position $\ell $ and zero
elsewhere. In this sense, $d^{(1)}$ is a quadratic form of $\Sigma $.

We next consider the remaining components $\Sigma _{h},$ $h=2,\ldots ,v$.
Write 
\begin{equation*}
\Sigma _{-h}:=\left\{ 
\begin{array}{ll}
\Sigma _{h+1}\otimes \cdots \otimes \Sigma _{v}\otimes \Sigma _{1}\otimes
\cdots \otimes \Sigma _{h-1} & \text{ for }h=2,\ldots ,v-1 \\ 
\Sigma _{1}\otimes \cdots \otimes \Sigma _{v-1} & \text{ for }h=v.%
\end{array}%
\right. .
\end{equation*}%
Note that $\Sigma _{-h}$ is $n_{-h}\times n_{-h}$ dimensional, where $%
n_{-h}:=n/n_{h}$. Recalling the identity $B\otimes A=K_{p,m}(A\otimes
B)K_{m,p}$ for $A$ ($m\times m$) and $B$ ($p\times p$) (\cite%
{magnusneudecker1986} Lemma 4), we write 
\begin{align}
\Sigma ^{(h)}& :=K_{n_{h}\times \cdots \times n_{v},n_{1}\times \cdots
\times n_{h-1}}\Sigma K_{n_{1}\times \cdots \times n_{h-1},n_{h}\times
\cdots \times n_{v}}  \notag \\
& =K_{n_{h}\times \cdots \times n_{v},n_{1}\times \cdots \times
n_{h-1}}(\sigma ^{2}\times \Sigma _{1}\otimes \cdots \otimes \Sigma
_{v})K_{n_{1}\times \cdots \times n_{h-1},n_{h}\times \cdots \times n_{v}} 
\notag \\
& =\sigma ^{2}\times \Sigma _{h}\otimes \Sigma _{h+1}\otimes \cdots \otimes
\Sigma _{v}\otimes \Sigma _{1}\otimes \cdots \otimes \Sigma _{h-1}=\sigma
^{2}\times \Sigma _{h}\otimes \Sigma _{-h}.  \label{align Sigma h}
\end{align}%
Define the $n_{h}\times n_{h}$ matrix $d^{(h)}:=\mathrm{PTR}_{n_{h}}(\Sigma
^{(h)})=\sigma ^{2}\mathrm{tr}(\Sigma _{-h})\times \Sigma _{h}$. Then 
\begin{equation*}
\Sigma _{h}=\frac{d^{(h)}}{\tr(d^{(h)})/n_{h}}.
\end{equation*}

\section{Estimation}

\label{sec estimation}

We observe an $n$-dimensional weakly stationary time series vector $%
\{y_{t}\}_{t=1}^{T}$ with mean $\mu $ and covariance matrix $\Sigma $.
Define the sample covariance matrix 
\begin{equation*}
M_{T}:=\frac{1}{T}\sum_{t=1}^{T}(y_{t}-\bar{y})(y_{t}-\bar{y})^{\intercal },
\end{equation*}%
where $\bar{y}:=\frac{1}{T}\sum_{t=1}^{T}y_{t}$. Define $\hat{d}^{(1)}:=%
\mathrm{PTR}_{n_{1}}(M_{T})$. Then let $\tilde{\Sigma}_{1}:=\hat{d}^{(1)}/(%
\mathrm{tr}(\hat{d}^{(1)})/n_{1}).$ Likewise, define the "permuted" sample
covariance matrix 
\begin{equation}
M_{T}^{(h)}:=K_{n_{h}\times \cdots \times n_{v},n_{1}\times \cdots \times
n_{h-1}}M_{T}K_{n_{1}\times \cdots \times n_{h-1},n_{h}\times \cdots \times
n_{v}},  \label{eqn random5}
\end{equation}%
for $h=2,\ldots ,v$. Define $\hat{d}^{(h)}:=\mathrm{PTR}_{n_{h}}(M_{T}^{(h)})
$ for $h=2,\ldots ,v$. Then 
\begin{equation}
\tilde{\Sigma}_{h}:=\frac{\hat{d}^{(h)}}{\mathrm{tr}(\hat{d}^{(h)})/n_{h}},
\label{eqn tilde Sigma}
\end{equation}%
for $h=1,\ldots ,v$.

The \textit{quadratic form} estimator $\tilde{\Sigma}$ for $\Sigma $ is 
\begin{equation*}
\tilde{\Sigma}=\hat{\sigma}^{2}\times \tilde{\Sigma}_{1}\otimes \cdots
\otimes \tilde{\Sigma}_{v},
\end{equation*}%
\begin{equation}
\hat{\sigma}^{2}:=\frac{\mathrm{tr}(M_{T})}{n}.  \label{eqn sigma2}
\end{equation}%
By Lemma 2.4 of \cite{filipiakkleinvojtkova2018}, if $M_{T}$ is symmetric
and positive semidefinite, then so are $\{\tilde{\Sigma}_{j}\}_{j=1}^{v}$
and hence $\tilde{\Sigma}.$ Moreover, simulations show that even for
positive semidefinite $M_{T}$, $\{\tilde{\Sigma}_{j}\}_{j=1}^{v},$ and hence 
$\tilde{\Sigma},$ are positive definite. As a result, the quadratic form
estimator $\tilde{\Sigma}^{-1}$ for $\Sigma ^{-1}$ is $\tilde{\Sigma}^{-1}=%
\hat{\sigma}^{-2}\times \tilde{\Sigma}_{1}^{-1}\otimes \cdots \otimes \tilde{%
\Sigma}_{v}^{-1}.$ We stress that $\tilde{\Sigma}^{-1}$ exists even if $n>T$%
. The quadratic form estimator is closely related to the quasi-maximum
likelihood estimation (QMLE), but has the particular advantage in large
dimensions in the sense that it is in closed form.\footnote{%
In the previous version of this article, we introduced a variant of the
quadratic form estimator, which was derived by replacing the partial trace
operator with a partial sum operator. Because of inferiority of that
variant, we no longer include it in the current version.}

In general we expect each element of $M_{T}$ to be $\sqrt{T}$-consistent,
but here we are averaging over a large number of such elements. Under a
cross-sectional weak dependence condition, like Assumption \ref{assu
summability}, we should have a rate improvement for the quadratic form
estimator. We formally establish this in Section \ref{sec asymptotic
properties}.

\section{The Rate of Convergence}

\label{sec asymptotic properties}

In this section, we shall derive the rate of convergence for the quadratic
form estimator. We make the following assumptions:

\begin{assumption}
\label{assu normality}

\item 
\begin{enumerate}[(i)]

\item The sample $\{y_{t}\}_{t=1}^{T}$ are independent over $t$.

\item 
\begin{equation*}
\max_{1\leq i\leq n}\frac{1}{T}\sum_{t=1}^{T}\mathbb{E}|y_{t,i}|^{m} \leq
A^m,\qquad m=2,3,\ldots,
\end{equation*}
for some absolute positive constant $A$.

\item Consider a \textit{normal} random vector $z_t$ which has the same mean
vector and covariance matrix as those of $y_t$. The $n^2\times n^2$ kurtosis
matrix of $y_t$ satisfies 
\begin{equation*}
\var\del[1]{(y_t-\mu)\otimes (y_t-\mu)} \leq C \var\del[1]{ (z_t-\mu)\otimes
(z_t-\mu)},
\end{equation*}
for some absolute positive constant $C$ for every $t$, where $\leq$ is to be
interpreted componentwise.
\end{enumerate}
\end{assumption}

Assumption \ref{assu normality}(i) facilitates our technical analysis, but
is perhaps not necessary. Assumption \ref{assu normality}(ii) assumes the
existence of an infinite number of moments of $y_{t}$, which allows one to
invoke a concentration inequality such as the Bernstein's inequality. Normal
random vectors or random vectors that exhibit some exponential-type tail
probability (e.g., subgaussianity, subexponentiality, semiexponentiality
etc) satisfy this condition. Assumption \ref{assu normality}(iii) supposes
that the kurtosis matrix of $y_{t}$ is of the same order of magnitude as if
it were a normal random vector. We impose this restriction on the kurtosis
matrix of $y_{t}$ because not much research has touched on unrestricted
kurtosis matrices in the large dimensional case.

\begin{assumption}
\label{assuBasic}

\item 
\begin{enumerate}[(i)]


\item $\{n_{j}\}_{j=1}^{v}$ are all fixed for any fixed $v$ and $%
v\rightarrow \infty $ (i.e., $n,T\to \infty$).

\item $\min_{1\leq j\leq v}\lambda_{\min}(\Sigma_{j})$ is bounded away from
zero by an absolute positive constant as $n,T\to \infty$.
\end{enumerate}
\end{assumption}

Assumption \ref{assuBasic}(i) says that the dimensions of the sub-matrices
are fixed when the number of sub-matrices tends to infinity. Note that
Assumption \ref{assuBasic}(ii) does not necessarily imply that $\lambda
_{\min }(\Sigma )$ is bounded away from zero by an absolute positive
constant. This is because $\lambda _{\min }(\Sigma )=\sigma ^{2}\times
\prod_{j=1}^{v}\lambda _{\min }(\Sigma _{j})$ and $v\rightarrow \infty $.

\begin{lemma}
\label{lemmaOmega_j} Suppose Assumption \ref{assuBasic}(i) hold. We have

\begin{enumerate}[(i)]

\item $v=O(\log n)$.

\item $\max_{1\leq j\leq v}\lambda_{\max}(\Sigma_{j})$ is bounded from above
by an absolute positive constant as $n,T\to \infty$.
\end{enumerate}
\end{lemma}

Note that Lemma \ref{lemmaOmega_j}(ii) does not necessarily imply that $%
\lambda_{\max}(\Sigma)$ is bounded from above by an absolute positive
constant. This is because $\lambda_{\max}(\Sigma)=\sigma^{2}\times\prod
_{j=1}^{v}\lambda_{\max}(\Sigma_{j})$ and $v\to\infty$.

\begin{assumption}
\label{assu summability} Let $0\leq \beta_1\leq 2$. 
\begin{equation*}
\lim_{n\to\infty}\frac{1}{n^{\beta_1}}\|\Sigma\|_F^2=\lim _{n\to\infty}\frac{%
\sigma^4}{n^{\beta_1}}\del[3]{\prod_{j=1}^{v}\|\Sigma_j\|_F^2}=\omega
<\infty.
\end{equation*}
\end{assumption}

Assumption \ref{assu summability} characterises the cross-sectional
dependence of $\{y_{t}\}_{t=1}^{T}$. According to Proposition 1 of \cite%
{chudikpesaran2013}, $\{y_{t}\}_{t=1}^{T}$ is said to be \textit{%
cross-sectionally weakly dependent}. The smaller $\beta _{1}$ is, the less
cross-sectional dependence of $\{y_{t}\}_{t=1}^{T}$ is allowed and the
stronger Assumption \ref{assu summability} is. When $\beta _{1}=2$,
Assumption \ref{assu summability} is slack as we are not restricting
cross-sectional dependence of $\{y_{t}\}_{t=1}^{T}$ at all ($\Vert \Sigma
\Vert _{F}^{2}=O(n^{2})$ in general). On the one hand, we would like to
assume $\beta _{1}$ as close to 2 as possible to make Assumption \ref{assu
summability} as weak as possible. On the other hand, the smaller $\beta _{1}$
is, the weaker cross-sectional dependence $\{y_{t}\}_{t=1}^{T}$ exhibits,
and the faster rate of convergence the quadratic form estimator will be able
to achieve. There is a trade off.

One important case is $\beta_1=1$. In this case one \textit{sufficient}
condition for Assumption \ref{assu summability} is that $\Sigma$ has bounded
maximum column sum matrix norm (i.e., $\|\Sigma\|_{\ell_1}=O(1)$) or bounded
maximum row sum matrix norm (i.e., $\|\Sigma\|_{\ell_{\infty}}=O(1)$). To
see this 
\begin{align*}
\frac{1}{n}\|\Sigma\|_F^2&\leq \frac{1}{n}n\|\Sigma\|_{\ell_1}^2=\frac{1}{n}%
n\|\Sigma\|_{\ell_{\infty}}^2=O(1).
\end{align*}
Note that for symmetric $\Sigma$, bounded maximum column sum matrix norm or
bounded maximum row sum matrix norm implies that the maximum eigenvalue of $%
\Sigma$ is bounded from above by an absolute positive constant and the
minimum eigenvalue of $\Sigma^{-1}$ is bounded away from zero by an absolute
positive constant: $1/(\lambda_{\min}(\Sigma^{-1}))=\lambda_{\max}(\Sigma)=%
\|\Sigma\|_{\ell_2}\leq \|\Sigma\|_{\ell_1}=\|\Sigma\|_{\ell_{\infty}}=O(1)$%
. The assumption of bounded maximum column/row sum matrix norm has been used
by \cite{fanliaoyao2015} (their Assumption 4.1(i)) and \cite%
{pesaranyamagata2012} (their Assumption 3). 

\begin{thm}
\label{thm quadratic form final} Suppose Assumptions \ref{assu normality}, %
\ref{assuBasic} and \ref{assu summability} hold. If $\log^3n/T\to 0$ as $%
n,T\to \infty$, then we have

\begin{enumerate}[(i)]

\item 
\begin{equation*}
\frac{\enVert[1]{\tilde{\Sigma}-\Sigma}_{F}}{\|\Sigma\|_{F}}=O_{p}%
\del[3]{\sqrt{\frac{\log^3 n}{n^{2-\beta_1}T}}}+O_p\del[3]{\frac{\log^2
n}{T}}.
\end{equation*}

\item 
\begin{equation*}
\frac{\enVert[1]{\tilde{\Sigma}^{-1}-\Sigma^{-1}}_{F}}{\|\Sigma^{-1}\|_{F}}%
=O_{p}\del[3]{\sqrt{\frac{\log^3 n}{n^{2-\beta_1}T}}}+O_p\del[3]{\frac{%
\log^2 n}{T}}.
\end{equation*}

\item 
\begin{equation*}
\frac{\enVert[1]{\tilde{\Sigma}-\Sigma}_{1}}{\|\Sigma\|_{1}}=O_{p}%
\del[3]{\sqrt{\frac{\log^3 n}{n^{2-\beta_1}T}}}+O_p\del[3]{\frac{\log^2
n}{T}}.
\end{equation*}

\item 
\begin{equation*}
\frac{\enVert[1]{\tilde{\Sigma}^{-1}-\Sigma^{-1}}_{1}}{\|\Sigma^{-1}\|_{1}}%
=O_{p}\del[3]{\sqrt{\frac{\log^3 n}{n^{2-\beta_1}T}}}+O_p\del[3]{\frac{%
\log^2 n}{T}}.
\end{equation*}

\item 
\begin{equation*}
\frac{\enVert[1]{\tilde{\Sigma}-\Sigma}_{\ell_2}}{\|\Sigma\|_{\ell_2}}=O_{p}%
\del[3]{\sqrt{\frac{\log^3 n}{n^{2-\beta_1}T}}}+O_p\del[3]{\frac{\log^2
n}{T}}.
\end{equation*}

\item 
\begin{equation*}
\frac{\enVert[1]{\tilde{\Sigma}^{-1}-\Sigma^{-1}}_{\ell_2}}{%
\|\Sigma^{-1}\|_{\ell_2}}=O_{p}\del[3]{\sqrt{\frac{\log^3
n}{n^{2-\beta_1}T}}}+O_p\del[3]{\frac{\log^2 n}{T}}.
\end{equation*}
\end{enumerate}
\end{thm}

\bigskip

The reason that we divide the Frobenius norm of the estimation error, say, $%
\Vert \tilde{\Sigma}-\Sigma \Vert _{F}$, by the Frobenius norm of the
target, i.e., $\Vert \Sigma \Vert _{F}$, is to define a proper notion of
"consistency". This is necessary because the cross-sectional dimension $n$
is growing to infinity. In this case, even if every element of a
matrix-valued estimator is converging in probability to the corresponding
element of its target matrix, there is no guarantee that its overall
estimation error will converge to zero in probability when $n,T\rightarrow
\infty $. The rescaling of the Frobenius norm of the estimation error is
standard in the large dimensional case, but in the literature scholars tend
to divide the Frobenius norm of the estimation error by $\sqrt{n}$ (e.g.,
see \cite{bickellevina2008} Theorem 2, \cite{fanliaomincheva2011} p3330, 
\cite{ledoitwolf2004} Definition 1 etc). The same reasoning applies to the $%
\ell _{1}$ and the spectral norm of the estimation error.

Note that there are two terms on the right side. The term $O_{p}%
\del[1]{\log^2n/T}$ exists because we need to estimate the unknown $\mu $.
If we knew $\mu $, this term would not be present.\footnote{%
If we knew $\mu $, the estimation procedure in Section \ref{sec estimation}
applies to $M_{T}^{0}:=T^{-1}\sum_{t=1}^{T}(y_{t}-\mu )(y_{t}-\mu
)^{\intercal }$ instead of $M_{T}$.} 
The rate of convergence, $\del[1]{\log^3 n/(n^{2-\beta_1}T)}^{1/2}$,
contains an additional, non-standard item $\sqrt{n^{2-\beta _{1}}}$ in the
denominator. This non-standard item exists because of the cross-sectional
weak dependence condition (Assumption \ref{assu summability}). If $\beta
_{1}=2$ (i.e., we are not restricting cross-sectional dependence of $%
\{y_{t}\}_{t=1}^{T}$ at all), this term vanishes. The rate of convergence of
the quadratic form estimator then becomes $(\log ^{3}n/T)^{1/2}$, which is
comparable to the convergence rates of other existent estimators in the
large dimensional case.

Take part (i) of the theorem as an illustration. If $\beta_1=2$ and we knew $%
\mu$, we have $\enVert[1]{\tilde{\Sigma}-\Sigma}_{F}=O_{p}%
\del[1]{\|\Sigma\|_{F}(\log^3n/T)^{1/2}}$. A typical threshold estimator $%
\hat{\Sigma}_{\text{thres}}$ has $\|\hat{\Sigma}_{\text{thres}%
}-\Sigma^{\dagger}\|_{F}=O_{p}\del[1]{(sn\log
n/T)^{1/2}}$, where $\Sigma^{\dagger}$ is some sparse truth and $s$ is its
sparsity index (see \cite{bickellevina2008} Theorem 2 with $q=0$). According
to \cite{bickellevina2008}, $s$ is the upper bound of non-zero elements for
every row, so $\|\Sigma^{\dagger}\|_F=O(\sqrt{sn})$ under the sparsity
model. If one assumes $\|\Sigma^{\dagger}\|_F\asymp \sqrt{sn}$, one can
write $\|\hat{\Sigma}_{\text{thres}}-\Sigma^{\dagger}\|_{F}=O_{p}%
\del[1]{\|\Sigma^{\dagger}\|_F(\log n/T)^{1/2}}$. Then the two rates of
convergence only differ by a logarithmic factor.

Because of the cross-sectional weak dependence condition (Assumption \ref%
{assu summability}), the quadratic form estimator is able to achieve a
faster rate of convergence than a typical estimator does.

\section{Test Statistics}

\label{sec test statistics}

We apply our methodology to the testing issue. We consider the problem of
testing the null hypothesis $H_{0}:\mu =\mu _{0}$ against the alternative $%
H_{1}:\mu \neq \mu _{0}$.

The classical Wald test statistic (based on the sample covariance matrix $%
M_{T}$) is not defined when $n\geq T$; there is a large literature that
proposes alternative test statistics. \cite{baisaranadasa1996} proposed a
statistic based on $\Vert \bar{y}\Vert _{2}^{2}$, thereby avoiding the
inversion of the large sample covariance matrix, and established its
asymptotic normality. \cite{pesaranyamagata2012} extended this approach to
the Capital Asset Pricing Model (CAPM) regression setting and proposed
several test statistics. One of the test statistics is based on $\Vert
t\Vert _{2}^{2},$ where $t$ is a vector of individual $t$-statistics; \cite%
{pesaranyamagata2012} derived the limiting normal distribution of the
centred and scaled version of this under cross-sectional weak dependence
conditions. \cite{fanliaoyao2015} considered a Wald test statistic for
testing the CAPM restrictions inside a linear regression in the large
dimensional case. They regularized the estimated error covariance matrix by
imposing a sparsity assumption, and used that to form a quadratic form. They
established the null limiting distribution of their test statistic (they
also proposed a novel power enhancement procedure, which we do not study
here).

We now define the Lagrange multiplier (LM) test statistic 
\begin{equation}
LM_{n,T}=T(\bar{y}-\mu _{0})^{\intercal }\tilde{\Sigma}_{\mu _{0}}^{-1}(\bar{%
y}-\mu _{0}),  \label{lmnt}
\end{equation}%
where $\tilde{\Sigma}_{\mu _{0}}$ is the quadratic form estimator assuming
that we know $\mu =\mu _{0}$. The Wald test statistic is 
\begin{equation}
W_{n,T}=T(\bar{y}-\mu _{0})^{\intercal }\tilde{\Sigma}^{-1}(\bar{y}-\mu
_{0}),  \label{wnt}
\end{equation}%
which is the Hotelling $T^{2}$-statistic based on the quadratic form
estimator. 
We next present the large sample properties of the binity $LM_{n,T}$ and $%
W_{n,T}$. We make one more cross-sectional dependence assumption.

\begin{assumption}
\label{assu summability 2} Let $0\leq \beta_2\leq 2$. 
\begin{equation*}
\lim_{n\to\infty}\frac{1}{n^{\beta_2}} \|\Sigma^{-1}\|_{1}= \lim_{n\to\infty
}\frac{1}{n^{\beta_2}\sigma^{2}}\del[3]{\prod_{j=1}^{v}\|\Sigma_j^{-1}\|_1}%
=\omega^{\prime}<\infty.
\end{equation*}
\end{assumption}

The bigger $\beta _{2}$ is, the weaker Assumption \ref{assu summability 2}
is. This is because it is putting less restriction on the cross-sectional
dependence of $\Sigma ^{-1}$. When $\beta _{2}=2$, Assumption \ref{assu
summability 2} is slack, as in essence we are not restricting anything. On
the one hand, we wish to assume $\beta _{2}$ as close to 2 as possible to
make Assumption \ref{assu summability 2} as weak as possible. On the other
hand, we wish to assume that $\beta _{2}$ is as small as possible so that
our methodology could accommodate an $n$ as large as possible.

One important case is $\beta_2=1$. In this case, a \textit{sufficient}
condition for Assumption \ref{assu summability 2} is that $\Sigma^{-1}$ has
bounded maximum column sum matrix norm (i.e., $\|\Sigma^{-1}\|_{\ell_1}=O(1)$%
) or bounded maximum row sum matrix norm (i.e., $\|\Sigma^{-1}\|_{\ell_{%
\infty}}=O(1)$). To see this 
\begin{align*}
\frac{1}{n} \|\Sigma^{-1}\|_{1}&=\frac{1}{n} \sum_{i=1}^{n}\sum_{j=1}^{n}%
\envert[1]{(\Sigma^{-1})_{i,j}}\leq \max_{1\leq i\leq n}\sum_{j=1}^{n}%
\envert[1]{(\Sigma^{-1})_{i,j}}=
\|\Sigma^{-1}\|_{\ell_{\infty}}=\|\Sigma^{-1}\|_{\ell_{1}}=O(1).
\end{align*}
Note that for symmetric $\Sigma^{-1}$, bounded maximum column sum matrix
norm or bounded maximum row sum matrix norm implies that the maximum
eigenvalue of $\Sigma^{-1}$ is bounded from above by an absolute positive
constant and the minimum eigenvalue of $\Sigma$ is bounded away from zero by
an absolute positive constant: $1/(\lambda_{\min}(\Sigma))=\lambda_{\max}(%
\Sigma^{-1})=\|\Sigma^{-1}\|_{\ell_2}\leq
\|\Sigma^{-1}\|_{\ell_1}=\|\Sigma^{-1}\|_{\ell_{\infty}}=O(1)$. The
assumption of bounded maximum column/row sum matrix norm has been used by 
\cite{fanliaoyao2015} (their Assumption 4.1(i)) and \cite%
{pesaranyamagata2012} (their Assumption 3).

\bigskip

\begin{thm}
\label{thm Wald statistics} Suppose Assumptions \ref{assu normality}, \ref%
{assuBasic}, \ref{assu summability}, and \ref{assu summability 2} hold. We
make the following assumptions:

\begin{enumerate}[(a)]

\item 
\begin{equation*}
\frac{n^{2\beta_2+\beta_1-3}\log^{5} n}{T}=o(1).
\end{equation*}

\item Consider the Cholesky decomposition of $\Sigma$, i.e., $%
\Sigma=LL^{\intercal}$, where $L$ is a nonsingular lower triangular matrix L
with positive diagonal elements. Assume that $x_t:=L^{-1}(y_t-\mu)$ is
cross-sectionally independent for any $t$, and for some $\delta >0$ 
\begin{equation*}
\limsup_{n,T\to \infty}\max_{1\leq i\leq n}\max_{1\leq t\leq T}\mathbb{E}%
\envert[1]{x_{t,i}}^{4+2\delta}<\infty.
\end{equation*}
\end{enumerate}

Then under $H_{0}:\mu =\mu _{0}$, as $n,T\rightarrow \infty $, 
\begin{equation*}
\frac{LM_{n,T}-n}{\sqrt{2n}}\xrightarrow{d}N(0,1).
\end{equation*}%
If one additionally assumes 
\begin{equation}
\frac{n^{\beta _{2}-\frac{1}{2}}\cdot \log ^{3}n}{T}=o(1),
\label{eqn wald test additional assumption}
\end{equation}%
then under $H_{0}:\mu =\mu _{0}$, as $n,T\rightarrow \infty $, 
\begin{equation}
\frac{W_{n,T}-n}{\sqrt{2n}}\xrightarrow{d}N(0,1).
\label{eqn wald weak convergence to standard normal}
\end{equation}
\end{thm}

\bigskip

For the LM test, if we want to allow the interesting case of $n/T\to \infty$%
, then assumption (a) necessarily implies that $2\beta_2+\beta_1<4$, which
restricts both $\beta_2$ and $\beta_1$. In the special case of $%
\beta_1=\beta_2=1$, assumption (a) is reduced to $\log^5n/T=o(1)$, which is
a weak condition.

Assumption (b) is standard in the literature. \cite{fanliaoyao2015}
maintained normality (their Assumption 4.1(i)), which is a special case of
assumption (b). \cite{pesaranyamagata2012} also maintained assumption (b)
(their Assumption 2a). Assumption (b) implicitly assumes that $\lambda
_{\min }(\Sigma )$ is bounded away from zero by an absolute positive
constant, which strengthens Assumption \ref{assuBasic}(ii). Also note that $%
\var(x_{t})=I_{n}$, so strengthening from cross-sectional uncorrelatedness
to cross-sectional independence in assumption (b) is rather innocuous. In
addition, we assume that the $(4+2\delta )$th moment of $x_{t,i}$ is
(uniformly in $i$ and $t$) finite for $n,T$ sufficiently large, which is
also a weak assumption. Under the more restricted sequential limit ($%
T\rightarrow \infty $ and then $n\rightarrow \infty ),$ $\sqrt{T}(\bar{y}%
-\mu _{0})$ is approximately normal so the limiting properties could be
calculated for the non-normal case as if normality held. However, in our
framework of joint limits, such procedures break down, so we make assumption
(b).

In the low-dimensional case ($n$ fixed, $T\rightarrow\infty$), the LM test
statistic $LM_{n,T}$ and the Wald test statistic $W_{n,T}$ are
asymptotically equivalent in the sense that they all converge in
distribution to $\chi_{n}^{2}$.\footnote{%
The finite sample performance of these statistics is known to vary. \cite%
{parkphillips1988} established higher order approximations for a Wald test
of nonlinear restrictions in the finite dimensional case, and showed how to
improve performance of the test statistic. It may be possible to apply their
methodology to the large dimensional case.} In the large dimensional case ($%
n,T\rightarrow\infty$), Theorem \ref{thm Wald statistics} shows that $%
LM_{n,T}$ and $W_{n,T}$ are, again, asymptotically equivalent. The Wald test
requires an additional rate restriction (\ref{eqn wald test additional
assumption}), which is the price we pay for estimating $\Sigma^{-1}$ under
the alternative $H_1:\mu\neq \mu_0$.

Recall that a typical threshold estimator $\hat{\Sigma}_{\text{thres}}$ has $%
\|\hat{\Sigma}_{\text{thres}}^{-1}-(\Sigma^{\dagger})^{-1}\|_{\ell_2}=O_{p}%
\del[1]{s(\log n/T)^{1/2}}$, where $\Sigma^{\dagger}$ is some sparse truth
and $s$ is its sparsity index (see \cite{bickellevina2008} Theorem 1 with $%
q=0$). For this rate of convergence, a result like (\ref{eqn wald weak
convergence to standard normal}) requires, as both \cite{pesaranyamagata2012}
and \cite{fanliaoyao2015} have pointed out, $n\log n/T=o(1)$, which is
essentially a low-dimensional scenario. \cite{pesaranyamagata2012} and \cite%
{fanliaoyao2015} have hence come up with their own ingenious ways to relax
the condition $n\log n/T=o(1)$ and established results similar to (\ref{eqn
wald weak convergence to standard normal}) for their Wald test statistics in
the CAPM context.

In the case of our Wald test, if we also want to allow the interesting large
dimension case of $n/T\rightarrow \infty $, then assumption (a) and (\ref%
{eqn wald test additional assumption}) necessarily imply that $2\beta
_{2}+\beta _{1}<4$ and $\beta _{2}<3/2$, respectively. For example, we can
choose the special case $\beta _{1}=\beta _{2}=1$, so assumption (a) and (%
\ref{eqn wald test additional assumption}) reduce to 
\begin{equation*}
\frac{\log ^{5}n}{T}=o(1),\qquad \frac{n^{\frac{1}{2}}\cdot \log ^{3}n}{T}%
=o(1),
\end{equation*}%
the latter of which is the binding rate condition and the same as the rate
condition in Assumption 4.2 of \cite{fanliaoyao2015}.

In the simulation study below we compare our tests with test statistics that
use Ledoit and Wolf procedures to regularize the sample covariance matrix
estimator.

\subsection{Power Investigation}

In this section, we analyse the asymptotic distributions of the proposed
test statistics under the alternative hypothesis $H_{1}:\mu \neq \mu _{0}$.
In particular, we shall focus on a sequence of local alternatives $H_{1}:\mu
=\mu _{T}:=\mu _{0}+\theta /\sqrt{T}$, where $\max_{1\leq i\leq n}|\theta
_{i}|=O(\sqrt{\log n})$. We focus on the Wald test without loss of
generality.

\begin{thm}
\label{thm Wald local alternative} Suppose Assumptions \ref{assu normality}, %
\ref{assuBasic}, \ref{assu summability}, and \ref{assu summability 2} hold.
We make the following additional assumptions:

\begin{enumerate}[(a)]

\item 
\begin{enumerate}[(i)]

\item 
\begin{equation*}
\frac{n^{2\beta_2+\beta_1-3}\cdot\log^{5} n}{T}=o(1),
\end{equation*}

\item 
\begin{equation*}
\frac{n^{\beta_2-\frac{1}{2}}\cdot\log^3 n}{T}=o(1).
\end{equation*}
\end{enumerate}

\item Consider the Cholesky decomposition of $\Sigma$, i.e., $%
\Sigma=LL^{\intercal}$, where $L$ is an $n\times n$ nonsingular lower
triangular matrix with positive diagonal elements. Assume that $%
x_t:=L^{-1}(y_t-\mu)$ is cross-sectionally independent for any $t$, and for
some $\delta >0$ 
\begin{align*}
\limsup_{n,T\to \infty}\max_{1\leq i\leq n}\max_{1\leq t\leq T}\mathbb{E}%
\envert[1]{x_{t,i}}^{4+2\delta}&<\infty \\
\limsup_{n\to \infty}\frac{1}{n}\sum_{i=1}^{n}\envert[1]{(L^{-1}\theta)_i}%
^{2+\delta}&<\infty.
\end{align*}
\end{enumerate}

Then under $H_{1}:\mu =\mu _{0}+\theta /\sqrt{T}$, 
\begin{equation*}
\frac{W_{n,T}-n}{\sqrt{2n\del[1]{1+\frac{2}{n}\theta^{\intercal}\Sigma^{-1}%
\theta}}}-\frac{\theta ^{\intercal }\Sigma ^{-1}\theta}{\sqrt{2n%
\del[1]{1+\frac{2}{n}\theta^{\intercal}\Sigma^{-1}\theta}}}\xrightarrow{d}%
N(0 ,1).
\end{equation*}
\end{thm}

\bigskip

The preceding theorem shows that the asymptotic distribution of $(W_{n,T}-n)/%
\sqrt{2n+4\theta^{\intercal}\Sigma^{-1}\theta}$ under $H_{1}$ has a center $%
\theta^{\intercal }\Sigma ^{-1}\theta /\sqrt{2n+4\theta^{\intercal}%
\Sigma^{-1}\theta}$. Note that 
\begin{equation*}
\frac{\theta ^{\intercal }\Sigma ^{-1}\theta }{\sqrt{2n+4\theta ^{\intercal
}\Sigma ^{-1}\theta }}\geq \frac{\theta ^{\intercal }\theta \lambda _{\min
}(\Sigma ^{-1})}{\sqrt{2n+4\theta ^{\intercal }\theta \lambda _{\max
}(\Sigma ^{-1})}}=\frac{\theta ^{\intercal }\theta /\lambda _{\max }(\Sigma )%
}{\sqrt{2n+4\theta ^{\intercal }\theta /\lambda _{\min }(\Sigma )}}.
\end{equation*}%
In the special case of $0<\lambda _{\min }(\Sigma )<\lambda _{\max }(\Sigma
)<\infty $, we see that the test has power against local alternatives that
satisfy $\max_{1\leq i\leq n}|\theta _{i}|=O(\sqrt{\log n})$ and $\theta
^{\intercal }\theta =O(n^{\delta _{a}}),$ where $\delta _{a}\geq 1/2,$ and
power tending to one in the case where $\delta _{a}>1/2.$ This specification
requires that $\theta $ has a sufficiently large number of non-zero
elements. It does not require that all the elements of $\theta $ are
non-zero.

\subsection{Testing Linear Restrictions of $\protect\mu$}

In this section, we consider testing linear restrictions of $\mu $ using two
approaches. We first consider $H_{0}:R\mu =r$, where $R$ is a $q\times n$
matrix of rank $q$. We assume that $q$ is a fixed number; this case covers
applications where a finite number of linear restrictions are coming from
economic theory. 

\begin{thm}
\label{thm linear restrictions} Suppose Assumptions \ref{assu normality}, %
\ref{assuBasic} and \ref{assu summability} hold. We also make the following
assumptions:

\begin{enumerate}[(a)]

\item $\lambda_{\min}(\Sigma)$ is bounded away from zero by an absolute
positive constant.

\item Consider $H_{0}:R\mu =r$, where $R$ is a $q\times n $ matrix of rank $%
q $ for any fixed $n$ and $n\rightarrow \infty $ ($q$ is a fixed number).
Moreover, $R$ and $r$ are rescaled in such a way that $\lambda _{\min
}(RR^{\intercal })$ is bounded away from zero by an absolute constant, and 
\begin{equation}
\lambda _{\max }(RR^{\intercal })\Vert \Sigma \Vert _{\ell _{2}}\del[3]{
\sqrt{\frac{\log ^{3}n}{n^{2-\beta _{1}}T}}+\frac{\log ^{2}n}{T}} =o(1).
\label{eqn random2}
\end{equation}
\end{enumerate}

Then under $H_0: R\mu=r$, if $\log^3n/T\to 0$ as $n,T\to \infty$, 
\begin{equation*}
W_{n,T}^*:=T(R\bar{y}-r)^{\intercal}(R\tilde{\Sigma} R^{\intercal})^{-1}(R%
\bar{y}-r)\xrightarrow{d}\chi^2_q.
\end{equation*}
\end{thm}

\bigskip

Assumption (a) strengthens Assumption \ref{assuBasic}(ii) slightly, which is
a mild condition. A sufficient condition for (\ref{eqn random2}) in
assumption (b) is $\lambda_{\max}(RR^{\intercal})$ is bounded from above by
an absolute positive constant and $\|\Sigma\|_{\ell_2}<\infty$. The
requirement of $\lambda_{\min}(R R^{\intercal})$ and $\lambda_{\max}(RR^{%
\intercal})$ being bounded away from zero and from above by absolute
positive constants, respectively, could be achieved by normalising each row
of $R$ to have $\ell_2$ norm of 1.

We next take another approach to derive simultaneous confidence intervals
for \textit{all} linear combinations of $\mu$.

\begin{lemma}
\label{lemma simu confi} Suppose Assumptions \ref{assu normality}, \ref%
{assuBasic}, \ref{assu summability}, and \ref{assu summability 2} hold.
Simultaneously for all $\phi \in \mathbb{R}^{n}$, the unknown $\mu $
satisfies the following inequalities with confidence $1-\alpha $: 
\begin{equation*}
\frac{T\sbr[1]{\phi^{\intercal}(\bar{y}-\mu)}^{2}/\phi ^{\intercal }\tilde{%
\Sigma}\phi -n}{\sqrt{2n}}<z_{\alpha },
\end{equation*}%
as $n,T\rightarrow \infty $, where $z_{\alpha }$ is the upper $\alpha $
percentile of a standard normal.
\end{lemma}

One disadvantage of this approach is that the confidence region for $\mu$
could be conservative.

\section{Simulation Study}

\label{sec simulation}

In this section, we provide some Monte Carlo simulations that evaluate
performance of our procedures.

\subsection{The Correctly Specified Case}

We suppose that $y_{t}\sim N(\mu,\Sigma)$ with $\Sigma=\Sigma_{1}\otimes
\cdots\otimes\Sigma_{v}$, where 
\begin{equation*}
\Sigma_{j}=\left( 
\begin{array}{cc}
1 & \rho_{j} \\ 
\rho_{j} & 1%
\end{array}
\right) \qquad|\rho_{j}|<1, \quad j=1,\ldots,v,
\end{equation*}
so in this case $\Sigma$ is also the correlation matrix. Given $\|\Sigma
_{j}\|_{F}^{2}=2(1+\rho_{j}^{2})$, we have 
\begin{align*}
\frac{1}{n^{\beta_1}}\|\Sigma\|_{F}^{2}=\frac{1}{n^{\beta_1}}%
\prod_{j=1}^{v}2(1+\rho_{j}^{2})=n^{1-\beta_1}\prod_{j=1}^{v}(1+%
\rho_{j}^{2}).
\end{align*}
Since $\prod_{j=1}^{v}(1+\rho_{j}^{2})\geq 1$, Assumption \ref{assu
summability} necessarily implies $\beta_1\geq 1$. When $\beta_1=1$, $\frac{1%
}{n^{\beta_1}}\|\Sigma\|_{F}^{2}=\prod_{j=1}^{v}(1+\rho_{j}^{2})$, which
converges to a finite, non-zero limit as $v\rightarrow\infty$ if and only if 
$\sum_{j=1}^{v}\rho_{j}^{2}$ converges (\cite{knopp1947} Theorem 28.3). 
When $\beta_1>1$, Assumption \ref{assu summability} is satisfied if $%
\prod_{j=1}^{v}(1+\rho_{j}^{2})=O(n^{\beta_1-1})$.

Likewise 
\begin{equation*}
\Vert \Sigma _{j}^{-1}\Vert _{1}=\frac{2(1+|\rho _{j}|)}{1-\rho _{j}^{2}}=%
\frac{2}{1-|\rho _{j}|}\qquad j=1,\ldots ,v,
\end{equation*}%
so that via Lemma \ref{lemma F norm Kronecker} in Section \ref{sec AppendixB}
\begin{equation*}
\frac{1}{n^{\beta_2}}\Vert \Sigma ^{-1}\Vert _{1}=\frac{1}{n^{\beta_2}}%
\prod_{j=1}^{v}\Vert \Sigma _{j}^{-1}\Vert _{1}=n^{1-\beta_2}\prod_{j=1}^{v}%
\frac{1}{1-|\rho _{j}|}=n^{1-\beta_2}\frac{1}{\prod_{j=1}^{v}%
\del[1]{1-|\rho_j|}}.
\end{equation*}%
Since $\prod_{j=1}^{v}\del[1]{1-|\rho_j|}\leq 1$, Assumption \ref{assu
summability 2} necessarily implies $\beta_2\geq 1$. When $\beta_2=1$, $\frac{%
1}{n^{\beta_2}}\Vert \Sigma ^{-1}\Vert _{1}=1/\prod_{j=1}^{v}%
\del[1]{1-|\rho_j|}$, the denominator of which converges to a finite,
non-zero limit as $v\rightarrow \infty $ if and only if $\sum_{j=1}^{v}|\rho
_{j}|$ converges (\cite{knopp1947} Theorem 28.4). When $\beta_2>1$,
Assumption \ref{assu summability 2} is satisfied if $\sbr[1]{\prod_{j=1}^{v}%
\del[1]{1-|\rho_j|}}^{-1}=O(n^{\beta_2-1})$.\footnote{%
Furthermore, the largest eigenvalue of $\Sigma $ is $\prod_{j=1}^{v}(1+|\rho
_{j}|)$, which converges as $v\rightarrow \infty $ if and only if $%
\sum_{j=1}^{v}|\rho _{j}|$ converges (\cite{knopp1947} Theorem 28.3).}

We consider $\mu=0$, $n=2^{v}$, and $\rho_{j}=\rho^{j}$ for $j=1,\ldots,v$.
The number of Monte Carlo simulations is 1000. We compare the quadratic form
estimator with \cite{ledoitwolf2004}'s linear shrinkage estimator (the LW04
estimator hereafter) and \cite{ledoitwolf2017}'s direct nonlinear shrinkage
estimator (the LW17 estimator hereafter).\footnote{%
The Matlab code for the LW04 and LW17 estimators is downloaded from the
website of Professor Michael Wolf from the Department of Economics at the
University of Zurich. We are grateful for this.}

The first evaluation criterion is the \textit{relative mean square error}
(MSE) in terms of $\Sigma $. Given a generic estimator $\hat{\Sigma}_{G}$ of
the covariance matrix $\Sigma $, we compute 
\begin{equation*}
\frac{\mathbb{E}\Vert \hat{\Sigma}_{G}-\Sigma \Vert _{F}^{2}}{\Vert \Sigma
\Vert _{F}^{2}}
\end{equation*}%
%
where the expectation operator is taken with respect to all the simulations.
Often the precision matrix $\Sigma ^{-1}$ is of more interest than $\Sigma $%
, so we also compute the MSE of the estimator of $\Sigma ^{-1}$: 
\begin{equation*}
\frac{\mathbb{E}\Vert \hat{\Sigma}_{G}^{-1}-\Sigma ^{-1}\Vert _{F}^{2}}{%
\Vert \Sigma ^{-1}\Vert _{F}^{2}}
\end{equation*}%
%
where the expectation operator is taken with respect to all the simulations.
Note that this requires invertibility of the generic estimator $\hat{\Sigma}%
_{G}$ and therefore cannot be calculated for the sample covariance matrix $%
M_{T}$ when $n>T$.

We next calculate 
\begin{equation*}
1-\frac{\mathbb{E}\Vert\hat{\Sigma}_{G}-\Sigma\Vert_{F}^{2}}{\mathbb{E}\Vert
M_{T}-\Sigma\Vert_{F}^{2}},
\end{equation*}
where the expectation operator is taken with respect to all the simulations.
The preceding display is called the simulated \textit{percentage relative
improvement in average loss} (PRIAL) criterion in terms of $\Sigma$ by \cite%
{ledoitwolf2004}. The PRIAL measures the performance of the generic
estimator $\hat{\Sigma}_{G}$ with respect to the sample covariance estimator 
$M_{T}$. Note that PRIAL$\in(-\infty,1]$: A negative value means $\hat{%
\Sigma }_{G}$ performs worse than $M_{T}$ while a positive value means
otherwise. Likewise we also compute 
\begin{equation*}
1-\frac{\mathbb{E}\Vert\hat{\Sigma}_{G}^{-1}-\Sigma^{-1}\Vert_{F}^{2}}{%
\mathbb{E}\Vert M_{T}^{-1}-\Sigma^{-1}\Vert_{F}^{2}}.
\end{equation*}
Note that this requires invertibility of the sample covariance matrix $M_{T}$
and therefore can only be calculated for $n<T$.

Finally, we consider testing $H_0:\mu=0$ against $H_1:\mu\neq 0$. We compute
sizes of the LM and Wald tests (Theorem \ref{thm Wald statistics}). The
significance level is 5\%. To investigate power, we generate $\mu_i$ as $%
\mu_i\overset{i.i.d.}{\sim}N(0,1)/\sqrt{T}$ for $i=1,2,\ldots,\lfloor
n^{0.7}\rfloor$, where $\lfloor x \rfloor$ is the largest integer less than
or equal to $x$; $\mu_i=0$ for $i=\lfloor n^{0.7}\rfloor+1, \lfloor
n^{0.7}\rfloor+2,\ldots, n$. These also require invertibility of $\hat{\Sigma%
}_{G}$.

The results are reported in Tables \ref{table correct spe}-\ref{table
correct spe 2}. In Table \ref{table correct spe}, we set $T=252$ and $v=10$
so that $n=2^v=1024$; we set $\rho=0.5,0.7,0.85$. First, consider the top
panel ($\rho=0.5$). For the MSE in terms of $\Sigma$ (i.e., MSE-1), all the
estimators beat the sample covariance matrix $M_{T}$ by a large margin. The
quadratic form estimator $\tilde{\Sigma}$ also outperformed the LW04 and
LW17 estimators considerably. For the MSE in terms of $\Sigma^{-1}$ (i.e.,
MSE-2), a similar pattern exists. Note that the MSE-2 cannot be computed for 
$M_{T}$ because $M_{T}$ is not invertible when $n>T$. For the PRIAL in terms
of $\Sigma$ (i.e., PRIAL-1), again $\tilde{\Sigma}$ is better than the LW04
and LW17 estimators. The sample covariance matrix $M_{T}$ has zero PRIAL-1
by definition. The superiority of $\tilde{\Sigma}$ in this experiment is
expected because the true covariance matrix is indeed a Kronecker product.

Considering the size of the Wald test, we realize that the quadratic form
estimator $\tilde{\Sigma}$ has the correct size while the LW04 and LW17
estimators are over-sized. Note that the Wald test is not defined for $M_{T}$
because $M_{T}$ is not invertible. Size of the LM test is similar to that of
the Wald test for $\tilde{\Sigma}$, but the LM test seems to perform poorly
for both the LW04 and LW17 estimators. Undoubtedly, the quadratic form
estimator $\tilde{\Sigma}$ is the best performing estimator.

As we increase the "mother" correlation parameter $\rho $ from 0.5 to 0.85,
performance of $\tilde{\Sigma}$ remains unchanged across all five criteria.
In terms of MSE-1, performance of $M_{T}$ improves while performances of
LW04 and LW17 estimators initially worsen and then improve. In terms of
MSE-2, PRIAL-1, the size of the LM test, and the size of the Wald test, the
performances of both the LW04 and LW17 estimators worsen. Again the
quadratic form estimator $\tilde{\Sigma}$ is the best performing estimator.

\begin{table}[ptb]
\centering
\begin{tabular}{lcccc}
\toprule & $M_{T}$ & $\tilde{\Sigma}$ & LW04 & LW17 \\ 
\midrule & \multicolumn{4}{c}{$\rho=0.5$} \\ 
MSE-1 & 2.989 & 0.000 & 0.242 & 0.243 \\ 
MSE-2 & NA & 0.000 & 0.311 & 0.308 \\ 
PRIAL-1 & 0 & 1.000 & 0.919 & 0.919 \\ 
size of LM & NA & 0.051 & 1.000 & 1.000 \\ 
size of Wald & NA & 0.050 & 0.085 & 0.093 \\ 
\cmidrule(lr){2-5} & \multicolumn{4}{c}{$\rho=0.7$} \\ 
MSE-1 & 1.760 & 0.000 & 0.429 & 0.430 \\ 
MSE-2 & NA & 0.000 & 0.722 & 0.715 \\ 
PRIAL-1 & 0 & 1.000 & 0.756 & 0.756 \\ 
size of LM & NA & 0.050 & 1.000 & 1.000 \\ 
size of Wald & NA & 0.051 & 0.158 & 0.164 \\ 
\cmidrule(lr){2-5} & \multicolumn{4}{c}{$\rho=0.85$} \\ 
MSE-1 & 0.501 & 0.001 & 0.320 & 0.316 \\ 
MSE-2 & NA & 0.002 & 0.980 & 0.980 \\ 
PRIAL-1 & 0 & 0.998 & 0.360 & 0.370 \\ 
size of LM & NA & 0.051 & 1.000 & 1.000 \\ 
size of Wald & NA & 0.060 & 0.334 & 0.329 \\ 
\bottomrule &  &  &  & 
\end{tabular}%
\caption{{\protect\small $M_{T}$, $\tilde{\Sigma}$, LW04 and LW17 stand for
the sample covariance matrix, quadratic form estimator, \protect\cite%
{ledoitwolf2004}'s linear shrinkage estimator, and \protect\cite%
{ledoitwolf2017}'s direct nonlinear shrinkage estimator, respectively. MSE-1
and MSE-2 are the MSE in terms of $\Sigma$ and $\Sigma^{-1}$, respectively.
PRIAL-1 is the PRIAL in terms of $\Sigma$. $T=252$ and $n=2^{10}=1024$.
0.000 means less than 0.001.}}
\label{table correct spe}
\end{table}

Next, we fix $\rho$ at 0.7 and examine effects of $n$ and $T$; the results
are reported in Table \ref{table correct spe 1}. If we fix $T$ at 252 and
increase $v$ (and hence $n$), in terms of MSE-1, all the estimators except
the quadratic form estimator $\tilde{\Sigma}$ worsen. The same pattern is
observed when we use the MSE-2 criterion instead (the sample covariance
matrix $M_T$ dropped out in this case). In terms of PRIAL-1, we see that all
the candidate estimators are becoming increasingly superior to $M_T$. As $n$
increases, size of the Wald test worsens for all the estimators except $%
\tilde{\Sigma}$; a similar pattern is observed for the LM test. If we
increase $T$ from 252 to 504, all the estimators improve in terms of both
the MSE-1 and MSE-2 criteria. Also sizes of the Wald and LM tests in general
improve for all the estimators.

\begin{table}[ptb]
\centering
\begin{tabular}{lccccccccc}
\toprule & $M_{T}$ & $\tilde{\Sigma}$ & LW04 & LW17 &  & $M_{T}$ & $\tilde{%
\Sigma}$ & LW04 & LW17 \\ 
\midrule & \multicolumn{4}{c}{$n=2^{9}$, $T=252$} &  & \multicolumn{4}{c}{$%
n=2^{9}$, $T=504$} \\ 
MSE-1 & 0.882 & 0.001 & 0.346 & 0.345 &  & 0.442 & 0.000 & 0.249 & 0.246 \\ 
MSE-2 & NA & 0.001 & 0.676 & 0.656 &  & NA & 0.000 & 0.601 & 0.531 \\ 
PRIAL-1 & 0 & 0.999 & 0.608 & 0.609 &  & 0 & 0.999 & 0.438 & 0.443 \\ 
size of LM & NA & 0.039 & 1.000 & 1.000 &  & NA & 0.053 & 1.000 & 1.000 \\ 
size of Wald & NA & 0.041 & 0.153 & 0.149 &  & NA & 0.058 & 0.148 & 0.151 \\ 
\cmidrule(lr){2-5} \cmidrule(lr){7-10} & \multicolumn{4}{c}{$n=2^{10}$, $%
T=252$} &  & \multicolumn{4}{c}{$n=2^{10}$, $T=504$} \\ 
MSE-1 & 1.760 & 0.000 & 0.429 & 0.430 &  & 0.882 & 0.000 & 0.345 & 0.344 \\ 
MSE-2 & NA & 0.000 & 0.722 & 0.715 &  & NA & 0.000 & 0.677 & 0.659 \\ 
PRIAL-1 & 0 & 1.000 & 0.756 & 0.756 &  & 0 & 1.000 & 0.608 & 0.610 \\ 
size of LM & NA & 0.050 & 1.000 & 1.000 &  & NA & 0.059 & 1.000 & 1.000 \\ 
size of Wald & NA & 0.051 & 0.158 & 0.164 &  & NA & 0.062 & 0.168 & 0.168 \\ 
\cmidrule(lr){2-5} \cmidrule(lr){7-10} & \multicolumn{4}{c}{$n=2^{11}$, $%
T=252$} &  & \multicolumn{4}{c}{$n=2^{11}$, $T=504$} \\ 
MSE-1 & 3.514 & 0.000 & 0.489 & 0.490 &  & 1.760 & 0.000 & 0.429 & 0.429 \\ 
MSE-2 & NA & 0.000 & 0.747 & 0.744 &  & NA & 0.000 & 0.723 & 0.717 \\ 
PRIAL-1 & 0 & 1.000 & 0.861 & 0.861 &  & 0 & 1.000 & 0.756 & 0.757 \\ 
size of LM & NA & 0.057 & 1.000 & 1.000 &  & NA & 0.057 & 1.000 & 1.000 \\ 
size of Wald & NA & 0.067 & 0.202 & 0.221 &  & NA & 0.060 & 0.169 & 0.181 \\ 
\bottomrule &  &  &  &  &  &  &  &  & 
\end{tabular}%
\caption{{\protect\small $M_{T}$, $\tilde{\Sigma}$, LW04 and LW17 stand for
the sample covariance matrix, quadratic form estimator, \protect\cite%
{ledoitwolf2004}'s linear shrinkage estimator, and \protect\cite%
{ledoitwolf2017}'s direct nonlinear shrinkage estimator, respectively. MSE-1
and MSE-2 are the MSE in terms of $\Sigma$ and $\Sigma^{-1}$, respectively.
PRIAL-1 is the PRIAL in terms of $\Sigma$. $\protect\rho=0.7$. 0.000 means
less than 0.001.}}
\label{table correct spe 1}
\end{table}

\begin{table}[ptb]
\centering
\begin{tabular}{lccccccccc}
\toprule & $M_{T}$ & $\tilde{\Sigma}$ & LW04 & LW17 &  & $M_{T}$ & $\tilde{%
\Sigma}$ & LW04 & LW17 \\ 
\midrule & \multicolumn{4}{c}{$\rho=0.5$, $n=2^{9}$} &  & \multicolumn{4}{c}{%
$\rho=0.5$, $n=2^{10}$} \\ 
power of LM & NA & 0.890 & 0.730 & 0.866 &  & NA & 0.925 & 0.999 & 1.000 \\ 
power of Wald & NA & 0.905 & 0.689 & 0.734 &  & NA & 0.942 & 0.750 & 0.780
\\ 
\cmidrule(lr){2-5} \cmidrule(lr){7-10} & \multicolumn{4}{c}{$\rho=0.7$, $%
n=2^{9}$} &  & \multicolumn{4}{c}{$\rho=0.7$, $n=2^{10}$} \\ 
power of LM & NA & 1.000 & 1.000 & 1.000 &  & NA & 1.000 & 1.000 & 1.000 \\ 
power of Wald & NA & 1.000 & 0.742 & 0.833 &  & NA & 1.000 & 0.746 & 0.806
\\ 
\cmidrule(lr){2-5} \cmidrule(lr){7-10} & \multicolumn{4}{c}{$\rho=0.85$, $%
n=2^{9}$} &  & \multicolumn{4}{c}{$\rho=0.85$, $n=2^{10}$} \\ 
power of LM & NA & 1.000 & 1.000 & 1.000 &  & NA & 1.000 & 1.000 & 1.000 \\ 
power of Wald & NA & 1.000 & 0.990 & 1.000 &  & NA & 1.000 & 0.981 & 0.977
\\ 
\bottomrule &  &  &  &  &  &  &  &  & 
\end{tabular}%
\caption{{\protect\small $M_{T}$, $\tilde{\Sigma}$, LW04 and LW17 stand for
the sample covariance matrix, the quadratic form estimator, the \protect\cite%
{ledoitwolf2004}'s linear shrinkage estimator, and the \protect\cite%
{ledoitwolf2017}'s direct nonlinear shrinkage estimator, respectively. $%
T=252 $. Powers are not size-adjusted.}}
\label{table correct spe 2}
\end{table}

The results of power investigation are reported in Table \ref{table correct
spe 2}. We see that power of the quadratic form estimator $\tilde{\Sigma}$
is very good for the specified local alternative. Powers of the LW04 and
LW17 estimators in terms of the Wald test are less good. Powers of the LW04
and LW17 estimators in terms of the LM test come at a price of their sizes.

\subsection{The Misspecified Case}

To gauge how well the Kronecker product model performs when the true
covariance matrix does not have a Kronecker product form, we consider the
Monte Carlo setting used by \cite{ledoitwolf2004}. We still assume that $%
y_{t}\sim N(\mu,\Sigma)$. The true covariance matrix $\Sigma$ is \textit{%
diagonal} without loss of generality. The diagonal entries $\Sigma_{ii}$
(i.e., the eigenvalues of $\Sigma$) are log normally distributed: $%
\log\Sigma_{ii}\sim N(\mu_{\text{LW}},\sigma_{\text{LW}}^{2})$. \cite%
{ledoitwolf2004} defined the \textit{grand mean} $\mu_{\text{g}}$ and 
\textit{cross-sectional dispersion} $\alpha^{2}$ of the eigenvalues of $%
\Sigma$ as, respectively, 
\begin{equation*}
\mu_{\text{g}}:=\frac{1}{n}\sum_{i=1}^{n}\Sigma_{ii}\qquad\alpha^{2}:=\frac{1%
}{n}\sum_{i=1}^{n}(\Sigma_{ii}-\mu_{\text{g}})^{2}.
\end{equation*}
In the Monte Carlo simulations, we re-define $\mu_{\text{g}}$ and $%
\alpha^{2} $ as the corresponding population counterparts: 
\begin{equation*}
\mu_{\text{g}}=\mathbb{E}\Sigma_{ii}=e^{\mu_{\text{LW}}+\sigma_{\text{LW}%
}^{2}/2}\qquad\alpha^{2}=\var\Sigma_{ii}=e^{2(\mu_{\text{LW}}+\sigma_{\text{%
LW}}^{2})}-e^{2\mu_{\text{LW}}+\sigma_{\text{LW}}^{2}}.
\end{equation*}
\cite{ledoitwolf2004} set $\mu_{\text{g}}=1$, so we can solve $\mu_{\text{LW}%
}=-\log(1+\alpha ^{2})/2$ and $\sigma_{\text{LW}}^{2}=\log(1+\alpha^{2}),$
whence we have 
\begin{equation*}
\log\Sigma_{ii}\sim N\del[3]{-\frac{\log (1+\alpha^2)}{2}, \log (1+\alpha^2)}%
.
\end{equation*}
Note that in this data generating process, there are two sources of
randomness: one from the normal distribution of $y_{t}$ and the other from
the log normal distribution of $\Sigma_{ii}$. Also note that a diagonal
covariance matrix need not have a Kronecker product structure unless, say,
the diagonal elements are all equal. The number of Monte Carlo simulations
is again set at 1000. In the baseline setting of \cite{ledoitwolf2004}, $%
\mu=0, n=20 $, $T=40$, and $\alpha^{2}=0.5$.

There are a number of different Kronecker products that we can consider to
approximate $\Sigma $ (see \cite{hafnerlintontang2018} for more discussions
of model selection). The possible Kronecker factorizations are $5\times
2\times 2$, $4\times 5$, $2\times 10$. Within each Kronecker factorization,
we can further permute the Kronecker sub-matrices to obtain different
Kronecker models. We study \textit{all} the Kronecker products and compare
with the LW04 and LW17 estimators. All estimators do not use knowledge of $%
\mu =0$ and have to estimate it, except in the case of the LM test.

The results are reported in Table \ref{table LW base}. The first observation
is that the performance of the quadratic form estimator $\tilde{\Sigma}$ is
relatively robust to the Kronecker product factorization; the best
performing one is $2\times5\times2$. All the candidate estimators beat the
sample covariance matrix $M_{T}$. In terms of MSE-1 and MSE-2, the LW04 and
LW17 estimators are only slightly better than $\tilde{\Sigma}$ $%
(2\times5\times2)$. In terms of PRIAL-1 and PRIAL-2, $\tilde{\Sigma}$ $%
(2\times5\times2)$ is almost as good as the LW04 and LW17 estimators. In
terms of size of the LM test, $\tilde{\Sigma}$ $(2\times5\times2)$ has the
correct size while the LW04 and LW17 estimators are under-sized. In terms of
size of the Wald test, all candidate estimators are slightly over-sized.

\begin{table}[ptb]
\centering
\begin{tabular}{lccccc}
\toprule & \multirow{2}{*}{$M_T$} & $\tilde{\Sigma}$ & $\tilde{\Sigma}$ & $%
\tilde{\Sigma}$ & $\tilde{\Sigma}$ \\ 
&  & $(5\times2 \times2)$ & $(2\times5 \times2)$ & $(2\times2 \times5)$ & $%
(4 \times5)$ \\ 
\cmidrule(lr){2-6} MSE-1 & 0.446 & 0.137 & 0.136 & 0.137 & 0.140 \\ 
MSE-2 & 6.876 & 0.154 & 0.153 & 0.154 & 0.163 \\ 
PRIAL-1 & 0 & 0.684 & 0.685 & 0.682 & 0.675 \\ 
PRIAL-2 & 0 & 0.977 & 0.977 & 0.977 & 0.976 \\ 
size of LM & 0.004 & 0.043 & 0.050 & 0.038 & 0.038 \\ 
size of Wald & 0.690 & 0.092 & 0.087 & 0.081 & 0.094 \\ 
\cmidrule(lr){2-6} & $\tilde{\Sigma}$ & $\tilde{\Sigma}$ & $\tilde{\Sigma}$
& \multirow{2}{*}{LW04} & \multirow{2}{*}{LW17} \\ 
& $(5\times4)$ & $(2 \times10)$ & $(10 \times2)$ &  &  \\ 
\cmidrule(lr){2-6} MSE-1 & 0.139 & 0.189 & 0.188 & 0.113 & 0.129 \\ 
MSE-2 & 0.163 & 0.293 & 0.288 & 0.122 & 0.148 \\ 
PRIAL-1 & 0.679 & 0.570 & 0.571 & 0.738 & 0.702 \\ 
PRIAL-2 & 0.976 & 0.957 & 0.958 & 0.982 & 0.978 \\ 
size of LM & 0.041 & 0.035 & 0.028 & 0.022 & 0.015 \\ 
size of Wald & 0.100 & 0.163 & 0.167 & 0.074 & 0.083 \\ 
\bottomrule &  &  &  &  & 
\end{tabular}%
\caption{{\protect\small $M_{T}$, $\tilde{\Sigma}$, LW04 and LW17 stand for
the sample covariance matrix, quadratic form estimator (factorisations given
in parentheses), \protect\cite{ledoitwolf2004}'s linear shrinkage estimator,
and \protect\cite{ledoitwolf2017}'s direct nonlinear shrinkage estimator,
respectively. MSE-1 and MSE-2 are the MSE in terms of $\Sigma$ and $%
\Sigma^{-1}$, respectively. PRIAL-1 and PRIAL-2 are the PRIAL in terms of $%
\Sigma$ and $\Sigma^{-1}$, respectively. $n=20,T=40,\protect\alpha^2=0.5$.}}
\label{table LW base}
\end{table}

We next vary $\alpha^{2}$. We base the comparisons on the $2\times5\times2$
Kronecker product factorization. The results are reported in Table \ref%
{table LW alpha2 vary}. As $\alpha^{2}$ increases, performance of $M_{T}$
actually improves in terms of MSE-1 and MSE-2. On the other hand,
performances of $\tilde{\Sigma}$, the LW04 and LW17 estimators worsen in
terms of MSE-1, MSE-2, PRIAL-1 and PRIAL-2. The worsening performance of $%
\tilde{\Sigma}$ is not surprising because $\alpha^{2}$ can be interpreted as
the distance of $\Sigma$ from a Kronecker product model. The worsening
performance of the LW04 estimator has also been documented by \cite%
{ledoitwolf2004}. As $\alpha^2$ increases, $\tilde{\Sigma}$ has roughly
correct size for the LM test while both the LW04 and LW17 estimators are
under-sized. In terms of the Wald test, all the candidate estimators are
slightly over-sized.

\begin{table}[ptb]
\centering
\begin{tabular}{lccccccccc}
\toprule & \multirow{2}{*}{$M_T$} & $\tilde{\Sigma}$ & \multirow{2}{*}{LW04}
& \multirow{2}{*}{LW17} &  & \multirow{2}{*}{$M_T$} & $\tilde{\Sigma}$ & %
\multirow{2}{*}{LW04} & \multirow{2}{*}{LW17} \\ 
&  & $(2\times5 \times2)$ &  &  &  &  & $(2\times5 \times2)$ &  &  \\ 
\cmidrule(lr){2-5} \cmidrule(lr){6-10} & \multicolumn{4}{c}{$\alpha^{2}=0.25$%
} &  & \multicolumn{4}{c}{$\alpha^{2}=0.50$} \\ 
MSE-1 & 0.492 & 0.077 & 0.050 & 0.070 &  & 0.446 & 0.136 & 0.113 & 0.129 \\ 
MSE-2 & 7.405 & 0.089 & 0.048 & 0.086 &  & 6.876 & 0.153 & 0.122 & 0.148 \\ 
PRIAL-1 & 0 & 0.843 & 0.898 & 0.856 &  & 0 & 0.685 & 0.738 & 0.702 \\ 
PRIAL-2 & 0 & 0.988 & 0.993 & 0.988 &  & 0 & 0.977 & 0.982 & 0.978 \\ 
size of LM & 0.004 & 0.042 & 0.035 & 0.020 &  & 0.004 & 0.050 & 0.022 & 0.015
\\ 
size of Wald & 0.690 & 0.083 & 0.064 & 0.066 &  & 0.690 & 0.087 & 0.074 & 
0.083 \\ 
\cmidrule(lr){2-10} & \multirow{2}{*}{$M_T$} & $\tilde{\Sigma}$ & %
\multirow{2}{*}{LW04} & \multirow{2}{*}{LW17} &  & \multirow{2}{*}{$M_T$} & $%
\tilde{\Sigma}$ & \multirow{2}{*}{LW04} & \multirow{2}{*}{LW17} \\ 
&  & $(2\times5 \times2)$ &  &  &  &  & $(2\times5 \times2)$ &  &  \\ 
\cmidrule(lr){2-5} \cmidrule(lr){6-10} & \multicolumn{4}{c}{$\alpha^{2}=0.75$%
} &  & \multicolumn{4}{c}{$\alpha^{2}=1$} \\ 
MSE-1 & 0.396 & 0.195 & 0.154 & 0.167 &  & 0.353 & 0.243 & 0.173 & 0.184 \\ 
MSE-2 & 6.311 & 0.241 & 0.194 & 0.204 &  & 5.807 & 0.335 & 0.259 & 0.246 \\ 
PRIAL-1 & 0 & 0.469 & 0.589 & 0.557 &  & 0 & 0.218 & 0.469 & 0.440 \\ 
PRIAL-2 & 0 & 0.959 & 0.966 & 0.966 &  & 0 & 0.934 & 0.948 & 0.953 \\ 
size of LM & 0.004 & 0.058 & 0.017 & 0.013 &  & 0.004 & 0.067 & 0.017 & 0.013
\\ 
size of Wald & 0.690 & 0.091 & 0.087 & 0.093 &  & 0.690 & 0.106 & 0.090 & 
0.091 \\ 
\bottomrule &  &  &  &  &  &  &  &  & 
\end{tabular}%
\caption{{\protect\small $M_{T}$, $\tilde{\Sigma}$, LW04 and LW17 stand for
the sample covariance matrix, quadratic form estimator (factorisations given
in parentheses), \protect\cite{ledoitwolf2004}'s linear shrinkage estimator,
and \protect\cite{ledoitwolf2017}'s direct nonlinear shrinkage estimator,
respectively. MSE-1 and MSE-2 are the MSE in terms of $\Sigma$ and $%
\Sigma^{-1}$, respectively. PRIAL-1 and PRIAL-2 are the PRIAL in terms of $%
\Sigma$ and $\Sigma^{-1}$, respectively. $n=20,T=40$.}}
\label{table LW alpha2 vary}
\end{table}

Finally, we vary the ratio $n/T$. In the baseline setting we have $n/T=0.5$.
Here we consider two variations. The first variation is $n=16,T=50$ with a
ratio of $n/T=0.32$. The second variation is $n=40,T=20$ with a ratio of $%
n/T=2$. For the first variation, we identify the Kronecker product
factorizations: $2\times2\times2\times2$, $4\times4$, $4\times2\times2$ and $%
2\times8$. For the second variation, we use the Kronecker product
factorizations: $5\times2\times2\times2$, $5\times2\times4$, $5\times8$ and $%
10\times2\times 2$. We also considered permutations of sub-matrices for each
factorization, but the performances remained relatively unchanged, so we do
not report them in the interest of space. The results are reported in Table %
\ref{table LW n over T ratio vary}.

Consider the top panel of Table \ref{table LW n over T ratio vary} first.
All the candidate estimators beat the sample covariance matrix $M_{T}$.
Performance of the quadratic form estimator $\tilde{\Sigma}$ is relatively
robust to the Kronecker product factorizations ($2\times2\times2\times2$, $%
4\times4$ and $4\times2\times2$); the best performing one is $4\times
2\times2$. In terms of MSE-1, MSE-2, PRIAL-1 and PRIAL-2, the quadratic form
estimator $\tilde{\Sigma}$ $(4\times2\times2)$ is only slightly worse than
the LW04 and LW17 estimators. In terms of size of the LM test, $\tilde{%
\Sigma }$ $(4\times2\times2)$ has the correct size while both the LW04 and
LW17 estimators are under-sized. In terms of size of the Wald test, all the
candidate estimators are slightly over-sized.

Next consider the bottom panel of Table \ref{table LW n over T ratio vary}.
All the candidate estimators beat the sample covariance matrix $M_{T}$
again. The best performing quadratic form estimator has a factorization ($%
5\times2\times2\times2$). In terms of MSE-1, MSE-2 and PRIAL-1, $\tilde {%
\Sigma}$ $(5\times2\times2\times2)$ is comparable to the LW04 and LW17
estimators. In terms of size of the LM test, $\tilde{\Sigma}$ $%
(5\times2\times2\times2)$ and the LW04 estimator have correct size while the
LW17 estimator is slightly over-sized. In terms of size of the Wald test,
all the candidate estimators are slightly over-sized.

By looking at Tables \ref{table LW base} and \ref{table LW n over T ratio
vary} together, we observe that as $n/T$ increases, PRIAL-1 increases
monotonically for the best performing quadratic form estimator as well as
the LW04 and LW17 estimators. Such a pattern is consistent with \cite%
{ledoitwolf2004}. In terms of MSE-1 and MSE-2, performances of the best
performing quadratic form estimator as well as the LW04 and LW17 estimators
worsen as $n/T$ increases. In terms of size of the LM test, the best
performing quadratic form estimator always has the correct size, while sizes
of the Wald tests increase monotonically with $n/T$.

\begin{table}[ptb]
\centering
\begin{tabular}{lccccccc}
\toprule \multirow{2}{*}{$n/T=0.32$} & \multirow{2}{*}{$M_T$} & $\tilde {%
\Sigma}$ & $\tilde{\Sigma}$ & $\tilde{\Sigma}$ & $\tilde{\Sigma}$ & %
\multirow{2}{*}{LW04} & \multirow{2}{*}{LW17} \\ 
&  & $(2\times2\times2 \times2)$ & $( 4 \times4)$ & $(4\times2 \times2)$ & $%
(2\times8)$ &  &  \\ 
\cmidrule(lr){2-8} MSE-1 & 0.292 & 0.118 & 0.122 & 0.120 & 0.145 & 0.098 & 
0.109 \\ 
MSE-2 & 1.491 & 0.134 & 0.142 & 0.137 & 0.190 & 0.110 & 0.118 \\ 
PRIAL-1 & 0 & 0.580 & 0.571 & 0.576 & 0.492 & 0.655 & 0.618 \\ 
PRIAL-2 & 0 & 0.907 & 0.902 & 0.905 & 0.870 & 0.924 & 0.919 \\ 
size of LM & 0.013 & 0.057 & 0.050 & 0.050 & 0.041 & 0.023 & 0.019 \\ 
size of Wald & 0.373 & 0.081 & 0.090 & 0.080 & 0.133 & 0.072 & 0.074 \\ 
\midrule \multirow{2}{*}{$n/T=2$} & \multirow{2}{*}{$M_T$} & $\tilde{\Sigma}$
& $\tilde{\Sigma}$ & $\tilde{\Sigma}$ & $\tilde{\Sigma}$ & %
\multirow{2}{*}{LW04} & \multirow{2}{*}{LW17} \\ 
&  & $(5\times2\times2\times2)$ & $(5\times2\times4)$ & $(5\times8)$ & $%
(10\times2 \times2)$ &  &  \\ 
\cmidrule(lr){2-8} MSE-1 & 1.684 & 0.168 & 0.175 & 0.216 & 0.234 & 0.159 & 
0.196 \\ 
MSE-2 & NA & 0.182 & 0.194 & 0.286 & 0.337 & 0.151 & 0.164 \\ 
PRIAL-1 & 0 & 0.898 & 0.894 & 0.870 & 0.860 & 0.904 & 0.882 \\ 
PRIAL-2 & NA & NA & NA & NA & NA & NA & NA \\ 
size of LM & NA & 0.051 & 0.054 & 0.050 & 0.049 & 0.051 & 0.070 \\ 
size of Wald & NA & 0.155 & 0.159 & 0.224 & 0.260 & 0.129 & 0.140 \\ 
\bottomrule &  &  &  &  &  &  & 
\end{tabular}%
\caption{{\protect\small $M_{T}$, $\tilde{\Sigma}$, LW04 and LW17 stand for
the sample covariance matrix, quadratic form estimator (factorisations given
in parentheses), \protect\cite{ledoitwolf2004}'s linear shrinkage estimator,
and \protect\cite{ledoitwolf2017}'s direct nonlinear shrinkage estimator,
respectively. MSE-1 and MSE-2 are the MSE in terms of $\Sigma$ and $%
\Sigma^{-1}$, respectively. PRIAL-1 and PRIAL-2 are the PRIAL in terms of $%
\Sigma$ and $\Sigma^{-1}$, respectively. $\protect\alpha^2=0.5$.}}
\label{table LW n over T ratio vary}
\end{table}

\section{Concluding Remarks}

We have proposed a new estimator of the Kronecker product model for
covariance matrices - the quadratic form estimator. We establish the rate of
convergence and use the estimated precision matrix to form the LM and Wald
test statistics. The asymptotic distributions of these test statistics are
established under both null and local alternative hypotheses. Testing linear
restrictions of the unknown mean vector is also investigated. In Monte Carlo
simulations, the quadratic form estimator performs well both when the
Kronecker product model is correctly specified and when it is misspecified.

We remark on a number of possible extensions. One can generalize to allow
weakly time series dependent data (see \cite{hafnerlintontang2018} for some
work in this direction), and perhaps to where the spectral density matrix is
Kronecker product factored. We may also consider the two-sample case where $%
\Sigma _{1}:=\mathbb{E}[(y_{1,t}-\mu _{1})(y_{1,t}-\mu _{1})^{\intercal }]$ $%
(n\times n)$, $\Sigma _{2}:=\mathbb{E}[(y_{2,t}-\mu _{2})(y_{2,t}-\mu
_{2})^{\intercal }]$ $(n\times n)$, $\mu _{1}:=\mathbb{E}(y_{1,t})$, and $%
\mu _{2}:=\mathbb{E}(y_{2,t}).$ \cite{chophillips2018} showed that the
hypothesis of $\Sigma _{1}=\Sigma _{2}$ can be tested based on $\mathrm{tr}%
(\Sigma _{1}\Sigma _{2}^{-1})=n$; if both the covariance matrices have a
conformable Kronecker product structure, this simplifies to $\mathrm{tr}%
(\Sigma _{1,1}\Sigma _{2,1}^{-1})\times \cdots \times \mathrm{tr}(\Sigma
_{1,v}\Sigma _{2,v}^{-1})=n.$

\label{sec conclusion}

\appendix

\section{Appendix}

\subsection{Proof of Lemma \protect\ref{lemmaOmega_j}}

\begin{proof}
For part (i), since $\prod_{j=1}^{v}n_j=n$, we have $\del[1]{\min_{1\leq j\leq v}n_j}^v\leq n$. Thus \[v\leq \log n/\log\del[1]{\min_{1\leq j\leq v}n_j}=O(\log n).\]
For part (ii):
\[\max_{1\leq j\leq v}\lambda_{\max}(\Sigma_j)\leq \max_{1\leq j\leq v}\tr(\Sigma_j)=\max_{1\leq j\leq v}n_j<\infty.\]
\end{proof}

\subsection{Proof of Theorem \protect\ref{thm quadratic form final}}

We first give an auxiliary lemma and an auxiliary theorem leading to the
proof of Theorem \ref{thm quadratic form final}.

\subsubsection{Lemma \protect\ref{lemma F norm submatrices}}

\begin{lemma}
\label{lemma F norm submatrices} Suppose Assumptions \ref{assu normality}
and \ref{assuBasic} hold. Then we have

\begin{enumerate}[(i)]

\item Both $\max_{1\leq j\leq v}\|\Sigma_{j}\|_{F}$ and $\max_{1\leq j\leq v
}\|\Sigma_{j}^{-1}\|_{F}$ are bounded from above by absolute positive
constants. Moreover both $\min_{1\leq j\leq v}\|\Sigma_{j}\|_{F}$ and $%
\min_{1\leq j\leq v}\|\Sigma_{j}^{-1}\|_{F}$ are bounded away from zero by
absolute positive constants.

\item Both $\max_{1\leq j\leq v}\|\Sigma_{j}\|_{1}$ and $\max_{1\leq j\leq v
}\|\Sigma_{j}^{-1}\|_{1}$ are bounded from above by absolute positive
constants. Moreover both $\min_{1\leq j\leq v}\|\Sigma_{j}\|_{1}$ and $%
\min_{1\leq j\leq v}\|\Sigma_{j}^{-1}\|_{1}$ are bounded away from zero by
absolute positive constants.

\item Both $\max_{1\leq j\leq v}\|\Sigma_{j}\|_{\ell_2}$ and $\max_{1\leq
j\leq v }\|\Sigma_{j}^{-1}\|_{\ell_2}$ are bounded from above by absolute
positive constants. Moreover both $\min_{1\leq j\leq
v}\|\Sigma_{j}\|_{\ell_2}$ and $\min_{1\leq j\leq
v}\|\Sigma_{j}^{-1}\|_{\ell_2}$ are bounded away from zero by absolute
positive constants.
\end{enumerate}
\end{lemma}

\begin{proof}[Proof of Lemma \ref{lemma F norm submatrices}]
For part (i), note that
\[\lambda_{\min}(\Sigma_j)\leq \lambda_{\max}(\Sigma_j)\leq \|\Sigma_j\|_F\leq \sqrt{n_j}\lambda_{\max}(\Sigma_j)\]
whence we deduce that $\max_{1\leq j\leq v}\|\Sigma_j\|_F$ is bounded from above by an absolute positive constant and $\min_{1\leq j\leq v}\|\Sigma_j\|_F$ is bounded away from zero by an absolute positive constant via Assumption \ref{assuBasic} and Lemma \ref{lemmaOmega_j}. Similarly, we have
\[\frac{1}{\lambda_{\max}(\Sigma_j)}=\lambda_{\min}(\Sigma_j^{-1})\leq \lambda_{\max}(\Sigma_j^{-1})\leq \|\Sigma_j^{-1}\|_F\leq \sqrt{n_j}\lambda_{\max}(\Sigma_j^{-1})=\sqrt{n_j}\frac{1}{\lambda_{\min}(\Sigma_j)},\]
whence we deduce that $\max_{1\leq j\leq v}\|\Sigma_j^{-1}\|_F$ is bounded from above by an absolute positive constant and $\min_{1\leq j\leq v}\|\Sigma_j^{-1}\|_F$ is bounded away from zero by an absolute positive constant via Assumption \ref{assuBasic} and Lemma \ref{lemmaOmega_j}. 

For part (ii), note that
\begin{align*}
\|\Sigma_j\|_F&\leq \|\Sigma_j\|_1\leq n_j \|\Sigma_j\|_F\\
\|\Sigma_j^{-1}\|_F&\leq \|\Sigma_j^{-1}\|_1\leq n_j \|\Sigma_j^{-1}\|_F
\end{align*}
whence we deduce that part (ii) holds via part (i).

For part (iii), we have
\begin{align*}
\max_{1\leq j\leq v}\|\Sigma_{j}\|_{\ell_2}&\leq \max_{1\leq j\leq v}\|\Sigma_{j}\|_{F}\\
\max_{1\leq j\leq v}\|\Sigma_{j}^{-1}\|_{\ell_2}&\leq \max_{1\leq j\leq v}\|\Sigma_{j}^{-1}\|_{F}
\end{align*}
whence we could deduce the first half of the statement via part (i). Next,
\begin{align*}
\min_{1\leq j\leq v}\|\Sigma_j\|_{\ell_2}=\min_{1\leq j\leq v}\lambda_{\max}(\Sigma_j)\geq \min_{1\leq j\leq v}\lambda_{\min}(\Sigma_j)
\end{align*}
which is bounded away from zero by an absolute positive constant via Assumption \ref{assuBasic}(ii). Finally,
\begin{align*}
\min_{1\leq j\leq v}\|\Sigma_{j}^{-1}\|_{\ell_2}=\min_{1\leq j\leq v}\lambda_{\max}(\Sigma_{j}^{-1})\geq \min_{1\leq j\leq v}\lambda_{\min}(\Sigma_{j}^{-1})=\min_{1\leq j\leq v}\frac{1}{\lambda_{\max}(\Sigma_j)}=\frac{1}{\max_{1\leq j\leq v}\lambda_{\max}(\Sigma_j)}
\end{align*}
which is bounded away from zero by an absolute positive constant via Lemma \ref{lemmaOmega_j}(ii).
\end{proof}

\subsubsection{Theorem \protect\ref{thm quadratic form}}

\begin{thm}
\label{thm quadratic form} Suppose Assumptions \ref{assu normality}, \ref%
{assuBasic} and \ref{assu summability} hold. Then

\begin{enumerate}[(i)]

\item 
\begin{equation*}
\max_{1\leq h\leq v}\max_{1\leq i,j\leq n_{h}}\frac{1}{n_{-h}}%
\envert[1]{\hat{d}_{i,j}^{(h)}-d_{i,j}^{(h)}}=O_{p}\del[3]{\sqrt{\frac{\log
n}{n^{2-\beta_1}T}}}+O_p\del[3]{\frac{\log n}{T}}.
\end{equation*}

\item We have $\tr(d^{(h)})/(n_{h}n_{-h})=\sigma^{2}>0$ for $h=1,\ldots,v$.
Also, 
\begin{equation*}
\max_{1\leq h\leq v}\frac{1}{n_{h}n_{-h}}\envert[1]{\tr(\hat{d}^{(h)})-%
\tr(d^{(h)})}=O_{p}\del[3]{\sqrt{\frac{\log n}{n^{2-\beta_1}T}}}+O_p%
\del[3]{\frac{\log n}{T}}.
\end{equation*}
As a result, $\min_{1\leq h\leq v}\tr(\hat{d}^{(h)})/(n_{h}n_{-h})$ is
bounded away from zero by an absolute positive constant in probability.

\item 
\begin{equation*}
\max_{1\leq h\leq v}\|\tilde{\Sigma}_{h}-\Sigma_{h}\|_{\infty}=\max_{1\leq
h\leq v}\max_{1\leq i,j\leq n_{h}}\envert[2]{[\tilde{\Sigma}_h]_{i,j}-[%
\Sigma_h]_{i,j}}=O_{p}\del[3]{\sqrt{\frac{\log n}{n^{2-\beta_1}T}}}+O_p%
\del[3]{\frac{\log n}{T}},
\end{equation*}
where $[\tilde{\Sigma}_h]_{i,j}$ and $[\Sigma_h]_{i,j}$ are the $(i,j)$th
entry of $\tilde{\Sigma}_h$ and $\Sigma_h$, respectively.

\item 
\begin{equation*}
\envert[1]{\hat{\sigma}^2-\sigma^2}=O_{p}\del[3]{\sqrt{\frac{1}{n^{2-%
\beta_1}T}}}+O_p\del[3]{\frac{\log n}{T}}.
\end{equation*}

\item 
\begin{equation*}
\max_{1\leq h\leq v}\|\tilde{\Sigma}_{h}-\Sigma_{h}\|_{F}=O_{p}%
\del[3]{\sqrt{\frac{\log n}{n^{2-\beta_1}T}}}+O_p\del[3]{\frac{\log n}{T}}.
\end{equation*}

\item 
\begin{equation*}
\max_{1\leq h\leq v}\|\tilde{\Sigma}_{h}^{-1}-\Sigma_{h}^{-1}\|_{F}=O_{p}%
\del[3]{\sqrt{\frac{\log n}{n^{2-\beta_1}T}}}+O_p\del[3]{\frac{\log n}{T}}.
\end{equation*}

\item 
\begin{equation*}
\max_{1\leq h\leq v}\|\tilde{\Sigma}_{h}-\Sigma_{h}\|_{1}=O_{p}%
\del[3]{\sqrt{\frac{\log n}{n^{2-\beta_1}T}}}+O_p\del[3]{\frac{\log n}{T}}.
\end{equation*}

\item 
\begin{equation*}
\max_{1\leq h\leq v}\|\tilde{\Sigma}_{h}^{-1}-\Sigma_{h}^{-1}\|_{1}=O_{p}%
\del[3]{\sqrt{\frac{\log n}{n^{2-\beta_1}T}}}+O_p\del[3]{\frac{\log n}{T}}.
\end{equation*}

\item 
\begin{equation*}
\max_{1\leq h\leq v}\|\tilde{\Sigma}_{h}-\Sigma_{h}\|_{\ell_2}=O_{p}%
\del[3]{\sqrt{\frac{\log n}{n^{2-\beta_1}T}}}+O_p\del[3]{\frac{\log n}{T}}.
\end{equation*}

\item 
\begin{equation*}
\max_{1\leq h\leq v}\|\tilde{\Sigma}_{h}^{-1}-\Sigma_{h}^{-1}\|_{%
\ell_2}=O_{p}\del[3]{\sqrt{\frac{\log n}{n^{2-\beta_1}T}}}+O_p%
\del[3]{\frac{\log n}{T}}.
\end{equation*}
\end{enumerate}
\end{thm}

\begin{proof}
For part (i), note that $d_{i,j}^{(h)}=\tr\Sigma^{(h)}_{\{[i,j]\}}$, where $\Sigma^{(h)}_{\{[i,j]\}}$ is the $[i,j]$th block of $\Sigma^{(h)}$ (each block is $n_{-h}\times n_{-h}$ dimensional) for $i,j=1,\ldots,n_h$. Similarly, $\hat{d}_{i,j}^{(h)}=\tr M^{(h)}_{T,\{[i,j]\}}$, where $M^{(h)}_{T,\{[i,j]\}}$ is the $[i,j]$th block of $M_T^{(h)}$ (each block is $n_{-h}\times n_{-h}$ dimensional). Write
\begin{align*}
\hat{d}_{i,j}^{(h)}=\tr M^{(h)}_{T,\{[i,j]\}}&=\tr M^{0,(h)}_{T,\{[i,j]\}}-\tr \del[2]{\sbr[1]{(\bar{y}-\mu)(\bar{y}-\mu)^{\intercal}}^{(h)}}_{\{[i,j]\}}=:\hat{d}_{i,j}^{0, (h)}-\tr \del[2]{\sbr[1]{(\bar{y}-\mu)(\bar{y}-\mu)^{\intercal}}^{(h)}}_{\{[i,j]\}}
\end{align*}
where 
\begin{align*}
M_T^0&:=\frac{1}{T}\sum_{t=1}^{T}(y_t-\mu)(y_t-\mu)^{\intercal}\\
M_T^{0,(h)} &:= K_{n_{h}\times \cdots \times n_{v},n_{1}\times \cdots
\times n_{h-1}}M_T^0 K_{n_{1}\times \cdots \times n_{h-1},n_{h}\times
\cdots \times n_{v}}\\
\sbr[1]{(\bar{y}-\mu)(\bar{y}-\mu)^{\intercal}}^{(h)}&:=K_{n_{h}\times \cdots \times n_{v},n_{1}\times \cdots
\times n_{h-1}}(\bar{y}-\mu)(\bar{y}-\mu)^{\intercal} K_{n_{1}\times \cdots \times n_{h-1},n_{h}\times
\cdots \times n_{v}},
\end{align*}
$M^{0,(h)}_{T,\{[i,j]\}}$ is the $[i,j]$th block of $M_T^{0,(h)}$ (each block is $n_{-h}\times n_{-h}$ dimensional), and $\del[1]{\sbr[1]{(\bar{y}-\mu)(\bar{y}-\mu)^{\intercal}}^{(h)}}_{\{[i,j]\}}$ is the $[i,j]$th block of $\sbr[1]{(\bar{y}-\mu)(\bar{y}-\mu)^{\intercal}}^{(h)}$ (each block is $n_{-h}\times n_{-h}$ dimensional). Thus we have
\begin{align}
&\max_{1\leq h\leq v}\max_{1\leq i,j\leq n_{h}}\frac{1}{n_{-h}}\envert[1]{\hat{d}_{i,j}^{(h)}-d_{i,j}^{(h)}}\notag\\
&\leq \max_{1\leq h\leq v}\max_{1\leq i,j\leq n_{h}}\frac{1}{n_{-h}}\envert[1]{\hat{d}_{i,j}^{0,(h)}-d_{i,j}^{(h)}}+\max_{1\leq h\leq v}\max_{1\leq i,j\leq n_{h}}\frac{1}{n_{-h}}\envert[3]{\tr \del[2]{\sbr[1]{(\bar{y}-\mu)(\bar{y}-\mu)^{\intercal}}^{(h)}}_{\{[i,j]\}}}.\label{align random28}
\end{align}

We consider the first term of (\ref{align random28}) first. Note that $\mathbb{E}[\hat{d}_{i,j}^{0,(h)}]=d_{i,j}^{(h)}$. Write for some $M>0$
\begin{align*}
&\mathbb{P}\del[3]{\max_{1\leq h\leq v}\max_{1\leq i,j\leq n_h}\sqrt{\frac{n^{2-\beta_1}T}{\log n}}\frac{1}{n_{-h}}|\hat{d}_{i,j}^{0,(h)}-d_{i,j}^{(h)}|>M}=\mathbb{P}\del[3]{\bigcup_{1\leq h\leq v}\bigcup_{1\leq i,j\leq n_h}\cbr[3]{\sqrt{\frac{n^{2-\beta_1}T}{\log n}}\frac{1}{n_{-h}}|\hat{d}_{i,j}^{0,(h)}-d_{i,j}^{(h)}|>M}}\\
&\leq \sum_{h=1}^{v}\sum_{i=1}^{n_h}\sum_{j=1}^{n_h}\mathbb{P}\del[3]{\sqrt{\frac{n^{2-\beta_1}T}{\log n}}\frac{1}{n_{-h}}|\hat{d}_{i,j}^{0,(h)}-d_{i,j}^{(h)}|>M}\leq \frac{n^{2-\beta_1}T\sum_{h=1}^{v}\sum_{i=1}^{n_h}\sum_{j=1}^{n_h}\var (\hat{d}_{i,j}^{0,(h)}/n_{-h})}{\log n \cdot M^2}\\
&\leq \frac{v\max_{1\leq h\leq v}n_h^2n^{2-\beta_1}T\max_{1\leq h\leq v}\max_{1\leq i,j\leq n_h}\var (\hat{d}_{i,j}^{0,(h)}/n_{-h})}{\log n\cdot M^2}
\end{align*}
where the second inequality is due to Chebyshev's inequality. We now show that 
\[\max_{1\leq h\leq v}\max_{1\leq i,j\leq n_h}\var (\hat{d}_{i,j}^{0,(h)}/n_{-h})=O\del[3]{\frac{1}{n^{2-\beta_1}T}}.\]
For arbitrary $i,j=1,\ldots, n_h$,
\begin{align}
&\var (\hat{d}_{i,j}^{0,(h)}/n_{-h})=\frac{1}{n_{-h}^2}\var \del[3]{\sum_{k=1}^{n_{-h}}\sbr[1]{M^{0,(h)}_{T,\{[i,j]\}}}_{kk}}=\frac{1}{n_{-h}^2}\var \del[3]{\frac{1}{T}\sum_{t=1}^{T}\sum_{k=1}^{n_{-h}}\dot{y}_{t,(i-1)n_{-h}+k}^{(h)} \dot{y}_{t,(j-1)n_{-h}+k}^{(h)}}\notag\\
&=\frac{1}{n_{-h}^2T}\sum_{k=1}^{n_{-h}}\sum_{\ell=1}^{n_{-h}}\cov \del[1]{\dot{y}_{t,(i-1)n_{-h}+k}^{(h)} \dot{y}_{t,(j-1)n_{-h}+k}^{(h)},\dot{y}_{t,(i-1)n_{-h}+\ell}^{(h)} \dot{y}_{t,(j-1)n_{-h}+\ell}^{(h)}}\notag\\
&\leq \frac{C}{n_{-h}^2T}\sum_{k=1}^{n_{-h}}\sum_{\ell=1}^{n_{-h}}\cov \del[1]{\dot{z}_{t,(i-1)n_{-h}+k}^{(h)} \dot{z}_{t,(j-1)n_{-h}+k}^{(h)},\dot{z}_{t,(i-1)n_{-h}+\ell}^{(h)} \dot{z}_{t,(j-1)n_{-h}+\ell}^{(h)}},\label{align school bus}
\end{align}
where $\dot{y}_t^{(h)}:=K_{n_h\times \cdots\times n_v, n_1\times \cdots \times n_{h-1}}(y_t-\mu)$ such that $\mathbb{E}[\dot{y}_t^{(h)}\dot{y}_t^{(h)\intercal}]=\Sigma^{(h)}$ and $\dot{z}_t^{(h)}$ is to be interpreted similarly, the third equality is due to independence over $t$ of $y_t$ in Assumption \ref{assu normality}(i), and the first inequality is due to Assumption \ref{assu normality}(iii).  Using Lemma 9 of \cite{magnusneudecker1986}, we have
\[\var \del[1]{\ve (\dot{z}_t^{(h)}\dot{z}_t^{(h)\intercal})}=\var \del[1]{\dot{z}_t^{(h)}\otimes \dot{z}_t^{(h)}}=2D_nD_n^+(\Sigma^{(h)}\otimes \Sigma^{(h)})=\del[1]{I_{n^2}+K_{n,n}}(\Sigma^{(h)}\otimes \Sigma^{(h)}),\]
where the last equality is due to (33) of \cite{magnusneudecker1986}. Thus we recognise that the summand on the right side of (\ref{align school bus}) is some element of $\del[1]{I_{n^2}+K_{n,n}}(\Sigma^{(h)}\otimes \Sigma^{(h)})$. We need to determine the exact position of the summand on the right side of (\ref{align school bus}) in $\del[1]{I_{n^2}+K_{n,n}}(\Sigma^{(h)}\otimes \Sigma^{(h)})$. We consider $\Sigma^{(h)}\otimes \Sigma^{(h)}$ and $K_{n,n}(\Sigma^{(h)}\otimes \Sigma^{(h)})$ separately.

Consider $\Sigma^{(h)}\otimes \Sigma^{(h)}$ first. We now introduce a new way to locate an element in a matrix. Divide the $n^2\times n^2$ matrix $\Sigma^{(h)}\otimes \Sigma^{(h)}$ into $n\times n$ blocks of matrices, each of which is $n\times n$ dimensional. Then $(\Sigma^{(h)}\otimes \Sigma^{(h)})_{\{[x,w],[p,q]\}}$ refers the $[p,q]$th element of the $[x,w]$th block matrix of $\Sigma^{(h)}\otimes \Sigma^{(h)}$, where $x,w,p,q=1,\ldots, n$. It is not difficult to see that
\[\cov \del[1]{\dot{z}_{t,(i-1)n_{-h}+k}^{(h)} \dot{z}_{t,(j-1)n_{-h}+k}^{(h)},\dot{z}_{t,(i-1)n_{-h}+\ell}^{(h)} \dot{z}_{t,(j-1)n_{-h}+\ell}^{(h)}}\]
corresponds to
\begin{equation}
\label{eqn random1}
(\Sigma^{(h)}\otimes \Sigma^{(h)})_{\cbr[1]{\sbr[1]{(i-1)n_{-h}+k,(i-1)n_{-h}+\ell},\sbr[1]{(j-1)n_{-h}+k,(j-1)n_{-h}+\ell}}}.
\end{equation}

We now consider $K_{n,n}(\Sigma^{(h)}\otimes \Sigma^{(h)})$. It is important to recognise that $K_{n,n}$ is a permutation matrix. Left multiplication of $\Sigma^{(h)}\otimes \Sigma^{(h)}$ by $K_{n,n}$ permutes the rows of $\Sigma^{(h)}\otimes \Sigma^{(h)}$. Since $K_{n,n}$ is $n\times n$, we can also divide $K_{n,n}$ into $n\times n$ blocks of matrices, each of which is $n\times n$ dimensional. Since $K_{n,n}$ is also a permutation matrix, its elements can only be either 0 or 1. It is not difficult to see that the $[q,p]$th element of the $[p,q]$th block matrix of $K_{n,n}$ is 1 for $p,q=1,\ldots, n$; all other elements of  $K_{n,n}$ are 0. Switch back to the traditional way to locate an element in a matrix. For $p,q=1,\ldots, n$, $[K_{n,n}]_{(p-1)n+q, (q-1)n+p}=1$. This implies that the $((p-1)n+q)$th row of $K_{n,n}(\Sigma^{(h)}\otimes \Sigma^{(h)})$ is actually the $((q-1)n+p)$th row of $\Sigma^{(h)}\otimes \Sigma^{(h)}$. Switch back to the new way to locate an element in a matrix. This says that, for arbitrary $x,w=1,\ldots,n$, the $[q,x]$th element of the $[p,w]$th block matrix of $K_{n,n}(\Sigma^{(h)}\otimes \Sigma^{(h)})$ is the $[p,x]$th element of the $[q,w]$th block matrix of $\Sigma^{(h)}\otimes \Sigma^{(h)}$. Thus
\[\cov \del[1]{\dot{z}_{t,(i-1)n_{-h}+k}^{(h)} \dot{z}_{t,(j-1)n_{-h}+k}^{(h)},\dot{z}_{t,(i-1)n_{-h}+\ell}^{(h)} \dot{z}_{t,(j-1)n_{-h}+\ell}^{(h)}}\]
corresponds to
\begin{align}
&[K_{n,n}(\Sigma^{(h)}\otimes \Sigma^{(h)})]_{\cbr[1]{\sbr[1]{(i-1)n_{-h}+k,(i-1)n_{-h}+\ell},\sbr[1]{(j-1)n_{-h}+k,(j-1)n_{-h}+\ell}}}\notag\\
&=(\Sigma^{(h)}\otimes \Sigma^{(h)})_{\cbr[1]{\sbr[1]{(j-1)n_{-h}+k,(i-1)n_{-h}+\ell},\sbr[1]{(i-1)n_{-h}+k,(j-1)n_{-h}+\ell}}}.\label{align random5}
\end{align}
Using (\ref{eqn random1}) and (\ref{align random5}), we have
\begin{align*}
&\max_{1\leq h\leq v}\max_{1\leq i,j\leq n_h}\var (\hat{d}^{0,(h)}_{i,j}/n_{-h})\\
&=\max_{1\leq h\leq v}\max_{1\leq i,j\leq n_h}\frac{1}{n_{-h}^2T}\sum_{k=1}^{n_{-h}}\sum_{\ell=1}^{n_{-h}}\cov \del[1]{\dot{y}_{t,(i-1)n_{-h}+k}^{(h)} \dot{y}_{t,(j-1)n_{-h}+k}^{(h)},\dot{y}_{t,(i-1)n_{-h}+\ell}^{(h)} \dot{y}_{t,(j-1)n_{-h}+\ell}^{(h)}}\\
&\leq \max_{1\leq h\leq v}\max_{1\leq i,j\leq n_h}\frac{C}{n_{-h}^2T}\sum_{k=1}^{n_{-h}}\sum_{\ell=1}^{n_{-h}}\cov \del[1]{\dot{z}_{t,(i-1)n_{-h}+k}^{(h)} \dot{z}_{t,(j-1)n_{-h}+k}^{(h)},\dot{z}_{t,(i-1)n_{-h}+\ell}^{(h)} \dot{z}_{t,(j-1)n_{-h}+\ell}^{(h)}}\\
&=\max_{1\leq h\leq v}\max_{1\leq i,j\leq n_h}\frac{C}{n_{-h}^2T}\sum_{k=1}^{n_{-h}}\sum_{\ell=1}^{n_{-h}}(\Sigma^{(h)}\otimes \Sigma^{(h)})_{\cbr[1]{\sbr[1]{(i-1)n_{-h}+k,(i-1)n_{-h}+\ell},\sbr[1]{(j-1)n_{-h}+k,(j-1)n_{-h}+\ell}}}\\
&\qquad +\max_{1\leq h\leq v}\max_{1\leq i,j\leq n_h}\frac{C}{n_{-h}^2T}\sum_{k=1}^{n_{-h}}\sum_{\ell=1}^{n_{-h}}(\Sigma^{(h)}\otimes \Sigma^{(h)})_{\cbr[1]{\sbr[1]{(j-1)n_{-h}+k,(i-1)n_{-h}+\ell},\sbr[1]{(i-1)n_{-h}+k,(j-1)n_{-h}+\ell}}}\\
&=\max_{1\leq h\leq v}\max_{1\leq i,j\leq n_h}\frac{C}{n_{-h}^2T}\sum_{k=1}^{n_{-h}}\sum_{\ell=1}^{n_{-h}}\del[1]{\Sigma^{(h)}_{(i-1)n_{-h}+k,(i-1)n_{-h}+\ell}\cdot\Sigma^{(h)}_{(j-1)n_{-h}+k,(j-1)n_{-h}+\ell}}\\
&\qquad+\max_{1\leq h\leq v} \max_{1\leq i,j\leq n_h}\frac{C}{n_{-h}^2T}\sum_{k=1}^{n_{-h}}\sum_{\ell=1}^{n_{-h}}\del[1]{\Sigma^{(h)}_{(j-1)n_{-h}+k,(i-1)n_{-h}+\ell}\cdot \Sigma^{(h)}_{(i-1)n_{-h}+k,(j-1)n_{-h}+\ell}}\\
&=\max_{1\leq h\leq v}\max_{1\leq i,j\leq n_h}\frac{C\sigma^4}{n_{-h}^2T}\sum_{k=1}^{n_{-h}}\sum_{\ell=1}^{n_{-h}}\sbr[2]{[\Sigma_h]_{i,i}\cdot [\Sigma_{-h}]_{k,\ell}\cdot [\Sigma_h]_{j,j}\cdot [\Sigma_{-h}]_{k,\ell}+[\Sigma_h]_{j,i}\cdot [\Sigma_{-h}]_{k,\ell}\cdot [\Sigma_h]_{i,j}\cdot [\Sigma_{-h}]_{k,\ell}}\\
&=\max_{1\leq h\leq v}\max_{1\leq i,j\leq n_h}\del[1]{[\Sigma_h]_{i,i}\cdot [\Sigma_h]_{j,j}+[\Sigma_h]_{i,j}\cdot [\Sigma_h]_{j,i}}\frac{C\sigma^4}{n_{-h}^2T}\sum_{k=1}^{n_{-h}}\sum_{\ell=1}^{n_{-h}}[\Sigma_{-h}]_{k,\ell}^2\\
&=\max_{1\leq h\leq v}\max_{1\leq i,j\leq n_h}\del[1]{[\Sigma_h]_{i,i}\cdot [\Sigma_h]_{j,j}+[\Sigma_h]_{i,j}\cdot [\Sigma_h]_{j,i}}\frac{C\sigma^4}{n_{-h}^2T}\|\Sigma_{-h}\|_F^2\leq  \max_{1\leq h\leq v}\max_{1\leq i,j\leq n_h}\frac{2C[\Sigma_h]_{i,i} [\Sigma_h]_{j,j}\sigma^4}{n_{-h}^2T}\|\Sigma_{-h}\|_F^2\\
&= \max_{1\leq h\leq v}\frac{O(1)\sigma^4}{n_{-h}^2T}\|\Sigma_{-h}\|_F^2= \max_{1\leq h\leq v}\frac{O(1)n_h^2\sigma^4}{n^2T\|\Sigma_{h}\|_F^2}\del[3]{\|\Sigma_{-h}\|_F^2\|\Sigma_{h}\|_F^2}= O\del[3]{\frac{1}{n^{2-\beta_1}T}} \max_{1\leq h\leq v}\del[3]{\frac{\sigma^4}{n^{\beta_1}}\|\Sigma_{-h}\|_F^2\|\Sigma_{h}\|_F^2}\\
&=O\del[3]{\frac{1}{n^{2-\beta_1}T}},
\end{align*}
where the second inequality is due to Cauchy-Schwarz inequality $[\Sigma_h]_{i,j}\leq \sqrt{[\Sigma_h]_{i,i}}\sqrt{[\Sigma_h]_{j,j}}$ using the fact that $\Sigma_h$ is a covariance matrix, the fourth last equality uses the fact that $\max_{1\leq h\leq v}\max_{1\leq i\leq n_{h}}[\Sigma_h]_{i,i}\leq \max_{1\leq h\leq v}\lambda_{\max}(\Sigma_h)<\infty$, the second last equality is due to Lemma \ref{lemma F norm submatrices}, and the last equality is due to Assumption \ref{assu summability}. We hence have
\begin{equation}
\label{eqn random3}
\max_{1\leq h\leq v}\max_{1\leq i,j\leq n_{h}}\frac{1}{n_{-h}}\envert[1]{\hat{d}_{i,j}^{0,(h)}-d_{i,j}^{(h)}}=O_{p}\del[3]{\sqrt{\frac{\log
n}{n^{2-\beta_1}T}}}.
\end{equation}

We now consider the second term of (\ref{align random28}).
\begin{align}
&\max_{1\leq h\leq v}\max_{1\leq i,j\leq n_{h}}\frac{1}{n_{-h}}\envert[3]{\tr \del[2]{\sbr[1]{(\bar{y}-\mu)(\bar{y}-\mu)^{\intercal}}^{(h)}}_{\{[i,j]\}}}=\max_{1\leq h\leq v}\max_{1\leq i,j\leq n_{h}}\frac{1}{n_{-h}}\envert[3]{\sum_{k=1}^{n_{-h}}\sbr[3]{ \del[2]{\sbr[1]{(\bar{y}-\mu)(\bar{y}-\mu)^{\intercal}}^{(h)}}_{\{[i,j]\}}}_{kk}}\notag\\
&\leq \max_{1\leq h\leq v}\max_{1\leq i,j\leq n_{h}}\max_{1\leq k\leq n_{-h}}\envert[3]{\sbr[3]{ \del[2]{\sbr[1]{(\bar{y}-\mu)(\bar{y}-\mu)^{\intercal}}^{(h)}}_{\{[i,j]\}}}_{kk}}\leq \max_{1\leq h\leq v}\enVert[2]{\sbr[1]{(\bar{y}-\mu)(\bar{y}-\mu)^{\intercal}}^{(h)}}_{\infty}\notag\\
&=\enVert[2]{(\bar{y}-\mu)(\bar{y}-\mu)^{\intercal}}_{\infty}=\sbr[3]{\max_{1\leq i\leq n}\envert[3]{\frac{1}{T}\sum_{t=1}^{T}(y_{t,i}-\mathbb{E}y_{t,i})}}^2=O_p\del[3]{\frac{\log n}{T}}\label{align random29}
\end{align}
where the last equality is due to Lemma \ref{lemma rate for mu hat} in Section \ref{sec AppendixB}. Inserting (\ref{eqn random3}) and (\ref{align random29}) into (\ref{align random28}) delivers part (i).

\bigskip

For part (ii), note that for $h=1,\ldots,v$
\begin{align*}
&\tr(d^{(h)})/(n_hn_{-h})=\frac{1}{n_{-h}}\sigma^2\tr(\Sigma_{-h})=\sigma^2 \frac{1}{n_{-h}}\tr(\Sigma_{h+1})\times \cdots\times \tr(\Sigma_v)\times \tr(\Sigma_1)\times \cdots\times \tr(\Sigma_{h-1})\\
&=\sigma^2>0.
\end{align*}
Now write
\begin{align*}
&\max_{1\leq h\leq v}\frac{1}{n_hn_{-h}}\envert[1]{\tr(\hat{d}^{(h)})-\tr(d^{(h)})}=\max_{1\leq h\leq v}\frac{1}{n_hn_{-h}}\envert[3]{\sum_{i=1}^{n_h}(\hat{d}^{(h)}_{i,i}-d^{(h)}_{i,i})}\leq\max_{1\leq h\leq v} \frac{1}{n_hn_{-h}}\sum_{i=1}^{n_h}\envert[1]{\hat{d}^{(h)}_{i,i}-d^{(h)}_{i,i}}\\
&\leq\max_{1\leq h\leq v} \max_{1\leq i\leq n_h}\frac{1}{n_{-h}}\envert[1]{\hat{d}^{(h)}_{i,i}-d^{(h)}_{i,i}}=O_p\del[3]{\sqrt{\frac{\log n}{n^{2-\beta_1}T}}}+O_p\del[3]{\frac{\log n}{T}},
\end{align*}
where the last equality is due to part (i). The last part of part (ii) also follows.

\bigskip

For part (iii), write
\begin{align}
\max_{1\leq h\leq v}\max_{1\leq i,j\leq n_h}\envert[1]{[\tilde{\Sigma}_h]_{i,j}-[\Sigma_h]_{i,j}}&\leq \max_{1\leq h\leq v}\max_{1\leq i,j\leq n_h}\envert[3]{\frac{\hat{d}^{(h)}_{i,j}}{\tr(\hat{d}^{(h)})/n_h}-\frac{d^{(h)}_{i,j}}{\tr(\hat{d}^{(h)})/n_h}}\notag\\
&\qquad +\max_{1\leq h\leq v}\max_{1\leq i,j\leq n_h}\envert[3]{\frac{d_{i,j}^{(h)}}{\tr(\hat{d}^{(h)})/n_h}-\frac{d_{i,j}^{(h)}}{\tr(d^{(h)})/n_h}}.\label{align random6}
\end{align}
Consider the first term on the right side of (\ref{align random6}).
\begin{align*}
&\max_{1\leq h\leq v}\max_{1\leq i,j\leq n_h}\envert[3]{\frac{\hat{d}^{(h)}_{i,j}}{\tr(\hat{d}^{(h)})/n_h}-\frac{d^{(h)}_{i,j}}{\tr(\hat{d}^{(h)})/n_h}}=\max_{1\leq h\leq v}\max_{1\leq i,j\leq n_h}\frac{1}{\tr(\hat{d}^{(h)})/(n_hn_{-h})}\frac{1}{n_{-h}}\envert[1]{\hat{d}^{(h)}_{i,j}-d^{(h)}_{i,j}}\\
&=O_p(1)\max_{1\leq h\leq v}\max_{1\leq i,j\leq n_h}\frac{1}{n_{-h}}\envert[1]{\hat{d}^{(h)}_{i,j}-d^{(h)}_{i,j}}=O_p\del[3]{\sqrt{\frac{\log n}{n^{2-\beta_1}T}}}+O_p\del[3]{\frac{\log n}{T}},
\end{align*}
where the second equality is due to part (ii) and the last equality is due to part (i). Consider the second term on the right side of (\ref{align random6}).
\begin{align*}
&\max_{1\leq h\leq v}\max_{1\leq i,j\leq n_h}\envert[3]{\frac{d_{i,j}^{(h)}}{\tr(\hat{d}^{(h)})/n_h}-\frac{d^{(h)}_{i,j}}{\tr(d^{(h)})/n_h}}=\max_{1\leq h\leq v}\max_{1\leq i,j\leq n_h}\frac{\envert[2]{\frac{\tr(d^{(h)})}{n_hn_{-h}}-\frac{\tr(\hat{d}^{(h)})}{n_hn_{-h}}}}{\envert[2]{\frac{\tr(\hat{d}^{(h)})}{n_hn_{-h}}\cdot\frac{\tr(d^{(h)})}{n_hn_{-h}}}}\frac{1}{n_{-h}}\envert[1]{d^{(h)}_{i,j}}\\
&=\sbr[3]{O_p\del[3]{\sqrt{\frac{\log n}{n^{2-\beta_1}T}}}+O_p\del[3]{\frac{\log n}{T}}}\max_{1\leq h\leq v}\max_{1\leq i,j\leq n_h}\frac{|d^{(h)}_{i,j}|}{n_{-h}}\\
&=\sbr[3]{O_p\del[3]{\sqrt{\frac{\log n}{n^{2-\beta_1}T}}}+O_p\del[3]{\frac{\log n}{T}}}\max_{1\leq h\leq v}\max_{1\leq i,j\leq n_h}\frac{\sigma^2|[\Sigma_h]_{i,j}|\tr(\Sigma_{-h})}{n_{-h}}\\
&=\sbr[3]{O_p\del[3]{\sqrt{\frac{\log n}{n^{2-\beta_1}T}}}+O_p\del[3]{\frac{\log n}{T}}}\max_{1\leq h\leq v}\max_{1\leq i,j\leq n_h}\sigma^2|[\Sigma_h]_{i,j}|\\
&=\sbr[3]{O_p\del[3]{\sqrt{\frac{\log n}{n^{2-\beta_1}T}}}+O_p\del[3]{\frac{\log n}{T}}}\sigma^2\max_{1\leq h\leq v}\lambda_{\max}(\Sigma_h)\\
&=O_p\del[3]{\sqrt{\frac{\log n}{n^{2-\beta_1}T}}}+O_p\del[3]{\frac{\log n}{T}},
\end{align*}
where the second equality is due to part (ii), and the last equality is due to Lemma \ref{lemmaOmega_j}(ii). Part (iii) hence follows.

\bigskip

For part (iv), write
\begin{align}
&\envert[1]{\hat{\sigma}^2-\sigma^2}=\envert[3]{\frac{1}{n}\tr (M_T-\Sigma)}=\envert[3]{\frac{1}{n}\tr \del[2]{M_T^0-(\bar{y}-\mu)(\bar{y}-\mu)^{\intercal}-\Sigma}}\notag\\
&=\envert[3]{\frac{1}{n}\sum_{i=1}^{n}\del[2]{ M_{T,i,i}^0-\sbr[1]{(\bar{y}-\mu)(\bar{y}-\mu)^{\intercal}}_{i,i}-\Sigma_{i,i}}}\notag\\
&\leq \envert[3]{\frac{1}{nT}\sum_{i=1}^{n}\sum_{t=1}^{T}\sbr[1]{(y_{t,i}-\mu_i)^2-\mathbb{E}(y_{t,i}-\mu_i)^2}}+\sbr[3]{\max_{1\leq i\leq n}\envert[3]{\frac{1}{T}\sum_{t=1}^{T}(y_{t,i}-\mu_{i})}}^2\notag\\
&=\envert[3]{\frac{1}{nT}\sum_{i=1}^{n}\sum_{t=1}^{T}\sbr[1]{\dot{y}_{t,i}^2-\mathbb{E}\dot{y}_{t,i}^2}}+O_p\del[3]{\frac{\log n}{T}}\label{align random30}
\end{align}
where $\dot{y}_{t,i}:=y_{t,i}-\mu_i$, and the last equality is due to Lemma \ref{lemma rate for mu hat} in Section \ref{sec AppendixB}. We now establish a rate for the first term in (\ref{align random30}). For some $M>0$,
\begin{align*}
\mathbb{P}\del[3]{\envert[3]{\frac{\sqrt{n^{2-\beta_1}T}}{nT}\sum_{i=1}^{n}\sum_{t=1}^{T}(\dot{y}_{t,i}^2-\mathbb{E}\dot{y}_{t,i}^2)}>M}\leq \frac{n^{2-\beta_1}T\var \del[1]{\frac{1}{nT}\sum_{i=1}^{n}\sum_{t=1}^{T}\dot{y}_{t,i}^2}}{M^2}.
\end{align*}
We now show $\var \del[1]{\frac{1}{nT}\sum_{i=1}^{n}\sum_{t=1}^{T}y_{t,i}^2}=O(1/(n^{2-\beta_1}T))$. 
\begin{align*}
&\var\del[3]{\frac{1}{nT}\sum_{i=1}^{n}\sum_{t=1}^{T}\dot{y}_{t,i}^2}=\frac{1}{T}\var\del[3]{\frac{1}{n}\sum_{i=1}^{n}\dot{y}_{t,i}^2}=\frac{1}{Tn^2}\sum_{i=1}^{n}\sum_{j=1}^{n}\cov\del[1]{\dot{y}_{t,i}\dot{y}_{t,i}, \dot{y}_{t,j}\dot{y}_{t,j}}\\
&\leq \frac{C}{Tn^2}\sum_{i=1}^{n}\sum_{j=1}^{n}\cov\del[1]{\dot{z}_{t,i}\dot{z}_{t,i}, \dot{z}_{t,j}\dot{z}_{t,j}}= \frac{C}{Tn^2}\sum_{i=1}^{n}\sum_{j=1}^{n}\del[2]{(\Sigma\otimes \Sigma)_{\{[i,j],[i,j]\}}+\del[1]{K_{n,n}(\Sigma\otimes \Sigma)}_{\{[i,j],[i,j]\}}}\\
&= \frac{2C}{Tn^2}\sum_{i=1}^{n}\sum_{j=1}^{n}(\Sigma\otimes \Sigma)_{\{[i,j],[i,j]\}}= \frac{2C}{Tn^2}\sum_{i=1}^{n}\sum_{j=1}^{n}\Sigma_{i,j}\cdot \Sigma_{i,j}=\frac{2C}{Tn^{2-\beta_1}}\frac{1}{n^{\beta_1}}\|\Sigma\|_F^2=O\del[3]{\frac{1}{Tn^{2-\beta_1}}}
\end{align*}
where the first equality is due to independence over $t$ of Assumption \ref{assu normality}(i), the first inequality is due to Assumption \ref{assu normality}(iii), the third and fourth equalities are due to the similar arguments which we used to prove part (i), and the last equality is due to Assumption \ref{assu summability}. Thus we have
\[\envert[3]{\frac{1}{nT}\sum_{i=1}^{n}\sum_{t=1}^{T}(\dot{y}_{t,i}^2-\mathbb{E}\dot{y}_{t,i}^2)}=O_p\del[3]{\sqrt{\frac{1}{n^{2-\beta_1}T}}}.\]
Substituting this into (\ref{align random30}) delivers part (iv).

\bigskip

For part (v), we have
\begin{align*}
&\max_{1\leq h\leq v}\|\tilde{\Sigma}_{h}-\Sigma_{h}\|_{F}\leq \max_{1\leq h\leq v}n_h\|\tilde{\Sigma}_{h}-\Sigma_{h}\|_{\infty}=O_{p}\del[3]{\sqrt{\frac{\log n}{n^{2-\beta_1}T}}}+O_p\del[3]{\frac{\log n}{T}}
\end{align*}
where the last equality is due to part (iii). For part (vi), invoke Lemma \ref{lemma saikkonen lemma} and use that $\max_{1\leq h\leq v}\|\Sigma_{h}^{-1}\|_F=O(1)$ in Lemma \ref{lemma F norm submatrices}.

\bigskip

For part (vii), we have
\[\max_{1\leq h\leq v}\|\tilde{\Sigma}_h-\Sigma_h\|_1\leq \max_{1\leq h\leq v}n_h \|\tilde{\Sigma}_h-\Sigma_h\|_F=O_{p}\del[3]{\sqrt{\frac{\log n}{n^{2-\beta_1}T}}}+O_p\del[3]{\frac{\log n}{T}}.\]
For part (viii), we have
\[\max_{1\leq h\leq v}\|\tilde{\Sigma}_h^{-1}-\Sigma_h^{-1}\|_1\leq \max_{1\leq h\leq v}n_h \|\tilde{\Sigma}_h^{-1}-\Sigma_h^{-1}\|_F=O_{p}\del[3]{\sqrt{\frac{\log n}{n^{2-\beta_1}T}}}+O_p\del[3]{\frac{\log n}{T}}.\]

\bigskip

For part (ix), we have
\[\max_{1\leq h\leq v}\|\tilde{\Sigma}_h-\Sigma_h\|_{\ell_2}\leq \max_{1\leq h\leq v} \|\tilde{\Sigma}_h-\Sigma_h\|_F=O_{p}\del[3]{\sqrt{\frac{\log n}{n^{2-\beta_1}T}}}+O_p\del[3]{\frac{\log n}{T}}.\]
For part (x), we have
\[\max_{1\leq h\leq v}\|\tilde{\Sigma}_h^{-1}-\Sigma_h^{-1}\|_{\ell_2}\leq \max_{1\leq h\leq v} \|\tilde{\Sigma}_h^{-1}-\Sigma_h^{-1}\|_F=O_{p}\del[3]{\sqrt{\frac{\log n}{n^{2-\beta_1}T}}}+O_p\del[3]{\frac{\log n}{T}}.\]
\end{proof}

\subsubsection{Proof of Theorem \protect\ref{thm quadratic form final}}

\begin{proof}[Proof of Theorem \ref{thm quadratic form final}]
For part (i),
\begin{align}
&\enVert[1]{\tilde{\Sigma}-\Sigma}_F/\|\Sigma\|_F=\enVert[1]{\hat{\sigma}^2\times \tilde{\Sigma}_1\otimes \cdots\otimes \tilde{\Sigma}_v-\sigma^2 \times \Sigma_1\otimes \cdots\otimes \Sigma_v}_F/\|\Sigma\|_F=\notag\\
&\enVert[1]{\hat{\sigma}^2\times \tilde{\Sigma}_1\otimes \cdots\otimes \tilde{\Sigma}_v-\hat{\sigma}^2 \times \Sigma_1\otimes \cdots\otimes \Sigma_v+\hat{\sigma}^2 \times \Sigma_1\otimes \cdots\otimes \Sigma_v-\sigma^2 \times \Sigma_1\otimes \cdots\otimes \Sigma_v}_F/\|\Sigma\|_F\notag\\
&\leq \hat{\sigma}^2\enVert[1]{ \tilde{\Sigma}_1\otimes \cdots\otimes \tilde{\Sigma}_v-  \Sigma_1\otimes \cdots\otimes\Sigma_v}_F/\|\Sigma\|_F+|\hat{\sigma}^2-\sigma^2|\enVert[1]{  \Sigma_1\otimes \cdots\otimes \Sigma_v}_F/\|\Sigma\|_F. \label{align random2}
\end{align}
We consider the first term in (\ref{align random2}). By inserting terms like $\Sigma_1\otimes \tilde{\Sigma}_2\otimes \cdots\otimes \tilde{\Sigma}_v$ and the triangular inequality, we have
\begin{align}
&\enVert[1]{ \tilde{\Sigma}_1\otimes \cdots\otimes \tilde{\Sigma}_v-  \Sigma_1\otimes \cdots\otimes\Sigma_v}_F\leq \notag\\
& \enVert[1]{\tilde{\Sigma}_1-\Sigma_1}_F\prod_{\ell=2}^{v}\|\tilde{\Sigma}_{\ell}\|_F+\sum_{j=2}^{v-1}\del[3]{\sbr[3]{\prod_{k=1}^{j-1}\|\Sigma_{k}\|_F}\enVert[1]{\tilde{\Sigma}_j-\Sigma_j}_F\sbr[3]{\prod_{\ell=j+1}^{v}\|\tilde{\Sigma}_{\ell}\|_F}}+\sbr[3]{\prod_{k=1}^{v-1}\|\Sigma_{k}\|_F}\enVert[1]{\tilde{\Sigma}_v-\Sigma_v}_F.\label{align three terms}
\end{align}
We first divide the first term of (\ref{align three terms}) by $\prod_{\ell=1}^{v}\|\Sigma_{\ell}\|_F$. We have
\begin{align}
&\frac{\enVert[1]{\tilde{\Sigma}_1-\Sigma_1}_F\prod_{\ell=2}^{v}\|\tilde{\Sigma}_{\ell}\|_F}{\prod_{\ell=1}^{v}\|\Sigma_{\ell}\|_F}=\frac{\enVert[1]{\tilde{\Sigma}_1-\Sigma_1}_F}{\|\Sigma_{1}\|_F}\prod_{\ell=2}^{v}\frac{\|\tilde{\Sigma}_{\ell}\|_F}{\|\Sigma_{\ell}\|_F}\leq \frac{\enVert[1]{\tilde{\Sigma}_1-\Sigma_1}_F}{\|\Sigma_{1}\|_F}\prod_{\ell=2}^{v}\sbr[3]{1+\frac{\|\tilde{\Sigma}_{\ell}-\Sigma_{\ell}\|_F}{\|\Sigma_{\ell}\|_F}}\notag \\
&\leq \frac{\enVert[1]{\tilde{\Sigma}_1-\Sigma_1}_F}{\|\Sigma_{1}\|_F}\sbr[3]{1+\frac{\max_{1\leq k\leq v}\|\tilde{\Sigma}_{k}-\Sigma_{k}\|_F}{\min_{1\leq k\leq v}\|\Sigma_{k}\|_F}}^{v-1}.\label{align three terms 1}
\end{align}
We next divide the summand of the second term of (\ref{align three terms}) by $\prod_{\ell=1}^{v}\|\Sigma_{\ell}\|_F$. We have for $j=2,\ldots,v-1$
\begin{align}
&\frac{\sbr[1]{\prod_{k=1}^{j-1}\|\Sigma_{k}\|_F}\enVert[1]{\tilde{\Sigma}_j-\Sigma_j}_F\sbr[1]{\prod_{\ell=j+1}^{v}\|\tilde{\Sigma}_{\ell}\|_F}}{\prod_{\ell=1}^{v}\|\Sigma_{\ell}\|_F}=\frac{\enVert[1]{\tilde{\Sigma}_j-\Sigma_j}_F}{\|\Sigma_{j}\|_F}\prod_{\ell=j+1}^{v}\frac{\|\tilde{\Sigma}_{\ell}\|_F}{\|\Sigma_{\ell}\|_F}\notag\\
&\leq \frac{\enVert[1]{\tilde{\Sigma}_j-\Sigma_j}_F}{\|\Sigma_{j}\|_F}\prod_{\ell=j+1}^{v}\sbr[3]{1+\frac{\|\tilde{\Sigma}_{\ell}-\Sigma_{\ell}\|_F}{\|\Sigma_{\ell}\|_F}}\leq \frac{\enVert[1]{\tilde{\Sigma}_j-\Sigma_j}_F}{\|\Sigma_{j}\|_F}\sbr[3]{1+\frac{\max_{1\leq k\leq v}\|\tilde{\Sigma}_{k}-\Sigma_{k}\|_F}{\min_{1\leq k\leq v}\|\Sigma_{k}\|_F}}^{v-j}.\label{align three terms 2}
\end{align}
We finally divide the third term of (\ref{align three terms}) by $\prod_{\ell=1}^{v}\|\Sigma_{\ell}\|_F$. We have
\begin{align}
\frac{\sbr[1]{\prod_{k=1}^{v-1}\|\Sigma_{k}\|_F}\enVert[1]{\tilde{\Sigma}_v-\Sigma_v}_F}{\prod_{\ell=1}^{v}\|\Sigma_{\ell}\|_F}=\frac{\enVert[1]{\tilde{\Sigma}_v-\Sigma_v}_F}{\|\Sigma_{v}\|_F}.\label{align three terms 3}
\end{align}
Thus we have
\begin{align*}
&\hat{\sigma}^2\enVert[1]{ \tilde{\Sigma}_1\otimes \cdots\otimes \tilde{\Sigma}_v-  \Sigma_1\otimes \cdots\otimes\Sigma_v}_F/\|\Sigma\|_F\leq  \frac{\hat{\sigma}^2}{\sigma^2}\sum_{j=1}^{v}\frac{\enVert[1]{\tilde{\Sigma}_j-\Sigma_j}_F}{\|\Sigma_j\|_F}\del[3]{1+\frac{\max_{1\leq k\leq v}\|\tilde{\Sigma}_k-\Sigma_k\|_F}{\min_{1\leq k\leq v}\|\Sigma_k\|_F}}^{v-j} \\
&\leq \frac{\hat{\sigma}^2}{\sigma^2}\del[3]{1+\frac{\max_{1\leq k\leq v}\|\tilde{\Sigma}_k-\Sigma_k\|_F}{\min_{1\leq k\leq v}\|\Sigma_k\|_F}}^{v-1}\sum_{j=1}^{v}\frac{\enVert[1]{\tilde{\Sigma}_j-\Sigma_j}_F}{\|\Sigma_j\|_F}=\frac{\hat{\sigma}^2}{\sigma^2}O_p(1)\sum_{j=1}^{v}\frac{\enVert[1]{\tilde{\Sigma}_j-\Sigma_j}_F}{\|\Sigma_j\|_F}\\
&=O_p(1) \sum_{j=1}^{v}\enVert[1]{\tilde{\Sigma}_j-\Sigma_j}_F=vO_p\del[3]{\sqrt{\frac{\log n}{n^{2-\beta_1}T}}\vee \frac{\log n}{T}}=O_p\del[3]{\sqrt{\frac{\log^3 n}{n^{2-\beta_1}T}}}+O_p\del[3]{\frac{\log^2 n}{T}}
\end{align*}
where the first inequality is due to that $\|\Sigma\|_F=\sigma^2\prod_{j=1}^{v}\|\Sigma_j\|_F$ via Lemma \ref{lemma F norm Kronecker}, (\ref{align three terms 1}), (\ref{align three terms 2}) and (\ref{align three terms 3}),  the first equality is due to Lemma \ref{lemma F norm submatrices} and Theorem \ref{thm quadratic form}(v)\footnote{To see this:
\begin{align*}
\del[3]{1+\frac{\max_{1\leq k\leq v}\|\tilde{\Sigma}_k-\Sigma_k\|_F}{\min_{1\leq k\leq v}\|\Sigma_k\|_F}}^{v-1}=\del[3]{1+O_p\del[3]{\sqrt{\frac{\log n}{n^{2-\beta_1}T}}\vee \frac{\log n}{T}}}^{O(\log n)}=O_p(1)
\end{align*}
where the last equality could be deduced from the fact that $\lim_{x\to\infty}(1+1/x)^x=e$ and $\log^3n/T\to 0$.
}, the second equality is due to Lemma \ref{lemma F norm submatrices} and Theorem \ref{thm quadratic form}(iv), and the third equality is due to Theorem \ref{thm quadratic form}(v).

We now consider the second term in (\ref{align random2}).
\begin{align*}
|\hat{\sigma}^2-\sigma^2|\enVert[1]{  \Sigma_1\otimes \cdots\otimes \Sigma_v}_F/\|\Sigma\|_F=\frac{|\hat{\sigma}^2-\sigma^2|}{\sigma^2}=O_p\del[3]{\sqrt{\frac{1}{n^{2-\beta_1}T}}}+O_p\del[3]{\frac{\log n}{T}}
\end{align*}
where the last equality is due to Theorem \ref{thm quadratic form}(iv). Part (ii)-(vi) of the theorem could be established in a similar manner, so we omit the details.
\end{proof}

\subsection{Proof of Theorem \protect\ref{thm Wald statistics}}

We first give an auxiliary theorem leading to the proof of Theorem \ref{thm
Wald statistics}.

\subsubsection{Theorem \protect\ref{thm Keleijian}}

The following theorem is adapted from Theorem 1 of \cite{kelejianprucha2001}.

\begin{thm}
\label{thm Keleijian} Consider $\{\varepsilon_{T,i}: 1\leq i\leq n, n\geq
1,T\geq 1\}$, an array of real numbers $\{b_{T,i}: 1\leq i\leq n, n\geq
1,T\geq 1\}$ and $Q_{n,T}:=
\sum_{i=1}^{n}\varepsilon_{T,i}^2+\sum_{i=1}^{n}b_{T,i}\varepsilon_{T,i}$.
Suppose that

\begin{enumerate}[(i)]

\item $\mathbb{E}[\varepsilon_{T,i}]=0$ for $1\leq i\leq n, n\geq 1,T\geq 1 $%
. Furthermore, for each $n\geq 1,T\geq 1$, $\varepsilon_{T,1},\ldots,
\varepsilon_{T,n}$ are (mutually) independent.

\item 
\begin{align*}
\limsup_{T\to \infty} \sup_{n\geq 1}\sup_{1\leq i\leq n}\mathbb{E}%
\envert[1]{\varepsilon_{T,i}}^{4+2\delta}&<\infty \\
\limsup_{n,T\to \infty}\frac{1}{n}\sum_{i=1}^{n}|b_{T,i}|^{2+\delta}&<\infty
\end{align*}
for some $\delta >0$.

\item 
\begin{equation*}
\liminf_{n,T\to\infty}\frac{1}{n}\var (Q_{n,T})\geq C>0
\end{equation*}
for some absolute positive constant $C$.
\end{enumerate}

Then as $n,T\to \infty$, 
\begin{equation*}
\frac{Q_{n,T}-\mathbb{E}[Q_{n,T}]}{\sqrt{\var(Q_{n,T})}}\xrightarrow{d}%
N(0,1).
\end{equation*}
\end{thm}

\begin{proof}
We can calculate that 
\begin{align*}
\mathbb{E}[Q_{n,T}]&=\mathbb{E}\sbr[3]{\sum_{i=1}^{n}\varepsilon_{T,i}^2+\sum_{i=1}^{n}b_{T,i}\varepsilon_{T,i}}=\sum_{i=1}^{n}\mathbb{E}\sbr[1]{\varepsilon_{T,i}^2}=:\sum_{i=1}^{n}\sigma_{T,i}^2\\
Q_{n,T}-\mathbb{E}[Q_{n,T}]&=\sum_{i=1}^{n}\del[1]{\varepsilon_{T,i}^2-\sigma_{T,i}^2+b_{T,i}\varepsilon_{T,i}}=:\sum_{i=1}^{n}Y_{T,i}\\
\mathbb{E}[Y_{T,i}^2]&=\mathbb{E}[\varepsilon_{T,i}^4]-\sigma_{T,i}^4+b_{T,i}^2\sigma_{T,i}^2+2b_{T,i}\mathbb{E}\varepsilon_{T,i}^3\\
\var(Q_{n,T})&=\mathbb{E}\sbr[3]{\sum_{i=1}^{n}Y_{T,i}}^2=\mathbb{E}\sbr[3]{\sum_{i=1}^{n}\sum_{j=1}^{n}Y_{T,i}Y_{T,j}}=\sum_{i=1}^{n}\mathbb{E}[Y_{T,i}^2],
\end{align*}
where the last equality is due to independence of $\varepsilon_{T,i}^2$ across $i$. We now show that
\begin{align*}
\frac{Q_{n,T}-\mathbb{E}[Q_{n,T}]}{\sqrt{\var(Q_{n,T})}}=\sum_{i=1}^{n}\frac{Y_{T,i}}{\sqrt{\var(Q_{n,T})}}\xrightarrow{d}N(0,1)
\end{align*}
as $n,T\to\infty$. This reduces to verifying the Lyapounov's condition in Theorem \ref{thmdoubleindexCLT} part (b); that is, for some $\delta>0$,
\[\lim_{n,T\to\infty}\sum_{i=1}^{n}\frac{1}{[\var(Q_{n,T})]^{1+\delta/2}}\mathbb{E}\envert[1]{Y_{T,i}}^{2+\delta}=0.\]
We first find an upper bound for $\mathbb{E}\envert[1]{Y_{T,i}}^{2+\delta}$. 
\begin{align*}
&\mathbb{E}\envert[1]{Y_{T,i}}^{2+\delta}=\mathbb{E}\envert[1]{\varepsilon_{T,i}^2-\sigma_{T,i}^2+b_{T,i}\varepsilon_{T,i}}^{2+\delta}\leq 3^{1+\delta}\del[1]{\mathbb{E}|\varepsilon_{T,i}^2|^{2+\delta}+\mathbb{E}|\sigma_{T,i}^2|^{2+\delta}+|b_{T,i}|^{2+\delta}\mathbb{E}|\varepsilon_{T,i}|^{2+\delta}}\\
&=3^{1+\delta}\del[1]{\mathbb{E}|\varepsilon_{T,i}|^{4+2\delta}+\sigma_{T,i}^{4+2\delta}+|b_{T,i}|^{2+\delta}\mathbb{E}|\varepsilon_{T,i}|^{2+\delta}}\leq K_1+K_2|b_{T,i}|^{2+\delta},
\end{align*}
for absolute positive constants $K_1$ and $K_2$ for sufficiently large $T$, where the first inequality is due to Loeve's $c_r$ inequality, and the last inequality is due to assumption (ii) of the theorem. Then we have
\begin{align*}
\sum_{i=1}^{n}\frac{\mathbb{E}\envert[1]{Y_{T,i}}^{2+\delta}}{[\var(Q_{n,T})]^{1+\delta/2}}\leq \frac{nK_1+nK_2\del[1]{\frac{1}{n}\sum_{i=1}^{n}|b_{T,i}|^{2+\delta}}}{[n^{-1}\var(Q_{n,T})]^{1+\delta/2}n^{1+\delta/2}}=\frac{K_1+K_2\del[1]{\frac{1}{n}\sum_{i=1}^{n}|b_{T,i}|^{2+\delta}}}{[n^{-1}\var(Q_{n,T})]^{1+\delta/2}n^{\delta/2}}\to 0
\end{align*}
as $n,T\to \infty$, where the convergence to 0 relies on assumptions (ii) and (iii) of the theorem.
\end{proof}

\subsubsection{Proof of Theorem \protect\ref{thm Wald statistics}}

\begin{proof}[Proof of Theorem \ref{thm Wald statistics}]
Write
\begin{align*}
\frac{LM_{n,T}-n}{\sqrt{2n}}&=\frac{T(\bar{y}-\mu_0)^{\intercal}\tilde{\Sigma}_{\mu_0}^{-1}(\bar{y}-\mu_0)-n}{\sqrt{2n}}\\
&=\frac{T(\bar{y}-\mu_0)^{\intercal}\Sigma^{-1}(\bar{y}-\mu_0)-n}{\sqrt{2n}}+\frac{T(\bar{y}-\mu_0)^{\intercal}(\tilde{\Sigma}^{-1}_{\mu_0}-\Sigma^{-1})(\bar{y}-\mu_0)}{\sqrt{2n}}.
\end{align*}
We first show that under $H_0$ as $n,T\to \infty$,
\[
\frac{T(\bar{y}-\mu_0)^{\intercal}\Sigma^{-1}(\bar{y}-\mu_0)-n}{\sqrt{2n}}\xrightarrow{d}N(0,1).\]
Write
\begin{align*}
&\frac{T(\bar{y}-\mu_0)^{\intercal}\Sigma^{-1}(\bar{y}-\mu_0)-n}{\sqrt{2n}}=\frac{\sbr[1]{\frac{1}{\sqrt{T}}\sum_{t=1}^{T}(y_t-\mu_0)}^{\intercal}(L^{-1})^{\intercal}L^{-1}\sbr[1]{\frac{1}{\sqrt{T}}\sum_{t=1}^{T}(y_t-\mu_0)}-n}{\sqrt{2n}}\\
&=:\frac{\del[1]{\frac{1}{\sqrt{T}}\sum_{t=1}^{T}x_t}^{\intercal}\del[1]{\frac{1}{\sqrt{T}}\sum_{t=1}^{T}x_t}-n}{\sqrt{2n}}=:\frac{z_{T}^{\intercal}z_{T}-n}{\sqrt{2n}}=\frac{\sum_{i=1}^{n}z_{T,i}^2-n}{\sqrt{2n}}=:\frac{Q_{n,T}-n}{\sqrt{2n}}.
\end{align*}
Note that for each $n\geq 1, T\geq 1$, $z_{T,1},\ldots, z_{T,n}$ are (mutually) independent under assumption (b) of the theorem and Assumption \ref{assu normality}(i). Under $H_0$, 
\begin{align*}
\mathbb{E}[z_{T,i}]&=\mathbb{E}\sbr[3]{\frac{1}{\sqrt{T}}\sum_{t=1}^{T}x_{t,i}}=0\\
\var(z_{T})&=\var\del[3]{\frac{1}{\sqrt{T}}\sum_{t=1}^{T}x_{t}}=I_n\\
\mathbb{E}[Q_{n,T}]&=\mathbb{E}\sbr[3]{\sum_{i=1}^{n}z_{T,i}^2}=\sum_{i=1}^{n}\mathbb{E}\sbr[1]{z_{T,i}^2}=n\\
\var(Q_{n,T})&=\var\del[3]{\sum_{i=1}^{n}z_{T,i}^2}=\sum_{i=1}^{n}\var\del[1]{z_{T,i}^2}=\sum_{i=1}^{n}\sbr[2]{\mathbb{E}[z_{T,i}^4]-\del[1]{\mathbb{E}[z_{T,i}^2]}^2}\\
&=\sum_{i=1}^{n}\del[2]{\mathbb{E}[z_{T,i}^4]-1}=:\sum_{i=1}^{n}\del[1]{\gamma_{z,i}+2}
\end{align*}
where $\gamma_{z,i}$ is the excess kurtosis of $z_{T,i}$:
\[\gamma_{z,i}:=\frac{\mathbb{E}[z_{T,i}^4]}{[\var(z_{T,i})]^2}-3=\mathbb{E}[z_{T,i}^4]-3.\]

We next calculate $\mathbb{E}[z_{T,i}^4]$ in terms of moments of $x_{t,i}$.
\begin{align}
\mathbb{E}[z_{T,i}^4]=\mathbb{E}\sbr[3]{\del[3]{\frac{1}{\sqrt{T}}\sum_{t=1}^{T}x_{t,i}}^4}=\frac{1}{T^2}\sum_{t=1}^{T}\sum_{s=1}^{T}\sum_{k=1}^{T}\sum_{\ell=1}^{T}\mathbb{E}\sbr[1]{x_{t,i}x_{s,i}x_{k,i}x_{\ell,i}}.\label{align random13}
\end{align}
Note that the summand in (\ref{align random13}) is non-zero only if $t=s=k=\ell$, $t=s\neq k=\ell$, $t=k\neq s=\ell$, $t=\ell\neq k=s$. First, consider the case $t=s=k=\ell$. Collecting all the summands in (\ref{align random13}) satisfying this, we have
\begin{align}
\frac{1}{T^2}\sum_{t=1}^{T}\mathbb{E}\sbr[1]{x_{t,i}^4}=\frac{1}{T^2}\sum_{t=1}^{T}(\gamma_{x,t,i}+3)=\frac{1}{T^2}\sum_{t=1}^{T}\gamma_{x,t,i}+\frac{3}{T}\label{align random14}
\end{align}
where $\gamma_{x,t,i}$ is the excess kurtosis of $x_{t,i}$:
\[\gamma_{x,t,i}:=\frac{\mathbb{E}\sbr[1]{x_{t,i}^4}}{[\var(x_{t,i})]^2}-3=\mathbb{E}\sbr[1]{x_{t,i}^4}-3.\]
Second, consider the case $t=s\neq k=\ell$. Collecting all the summands in (\ref{align random13}) satisfying this, we have
\begin{align}
\frac{1}{T^2}\sum_{t=1}^{T}\sum_{\substack{k=1\\\neq t}}^{T}\mathbb{E}\sbr[1]{x_{t,i}^2x_{k,i}^2}=\frac{1}{T^2}\sum_{t=1}^{T}\sum_{\substack{k=1\\\neq t}}^{T}\mathbb{E}\sbr[1]{x_{t,i}^2}\mathbb{E}\sbr[1]{x_{k,i}^2}=\frac{T(T-1)}{T^2}=1-\frac{1}{T}.\label{align random15}
\end{align}
Likewise, for cases $t=k\neq s=\ell$ and $t=\ell\neq k=s$, both sums are $1-1/T$. Substituting (\ref{align random14}) and (\ref{align random15}) into (\ref{align random13}), we have
\begin{align*}
\mathbb{E}[z_{T,i}^4]=\frac{1}{T^2}\sum_{t=1}^{T}\gamma_{x,t,i}+\frac{3}{T}+3\del[3]{1-\frac{1}{T}}=\frac{1}{T^2}\sum_{t=1}^{T}\gamma_{x,t,i}+3
\end{align*}
whence we have $\gamma_{z,i}=\mathbb{E}[z_{T,i}^4]-3=\frac{1}{T^2}\sum_{t=1}^{T}\gamma_{x,t,i}$ and
\begin{equation}
\label{eqn random4}
\var(Q_{n,T})=\sum_{i=1}^{n}\del[1]{\gamma_{z,i}+2}=\sum_{i=1}^{n}\del[3]{\frac{1}{T^2}\sum_{t=1}^{T}\gamma_{x,t,i}+2}=2n\del[3]{1+\frac{1}{2T}\frac{1}{nT}\sum_{i=1}^{n}\sum_{t=1}^{T}\gamma_{x,t,i}}.
\end{equation}

It remains to verify condition (ii)-(iii) of Theorem \ref{thm Keleijian}. We have
\begin{align*}
\frac{1}{n}\var(Q_{n,T})=2+\frac{1}{T}\del[3]{\frac{1}{nT}\sum_{i=1}^{n}\sum_{t=1}^{T}\gamma_{x,t,i}}>0
\end{align*}
for large enough $T$ because $\gamma_{x,t,i}>-3$ for all $t$ and $i$ by definition of the excess kurtosis. Hence (iii) of Theorem \ref{thm Keleijian} is satisfied. Condition (ii) of Theorem \ref{thm Keleijian} is also satisfied: for some $\delta>0$
\[\limsup_{T\to \infty} \sup_{n\geq 1}\sup_{1\leq i\leq n}\mathbb{E}\envert[3]{\frac{1}{\sqrt{T}}\sum_{t=1}^{T}x_{t,i}}^{4+2\delta}<\infty\]
by Theorem \ref{thmBrillinger} in Section \ref{sec AppendixB} under assumption (b) of the theorem. Thus we have
\begin{align*}
\frac{T(\bar{y}-\mu_0)^{\intercal}\Sigma^{-1}(\bar{y}-\mu_0)-n}{\sqrt{2n}}&=\frac{Q_{n,T}-n}{\sqrt{2n}}=\frac{Q_{n,T}-n}{\sqrt{2n\del[1]{1+\frac{1}{2T}\frac{1}{nT}\sum_{i=1}^{n}\sum_{t=1}^{T}\gamma_{x,t,i}}}}\del[1]{1+o(1)}\xrightarrow{d}N(0,1),
\end{align*}
under $H_0$ as $n,T\to \infty$, where the second equality is due to 
\[\limsup_{n,T\to \infty}\frac{1}{nT}\sum_{i=1}^{n}\sum_{t=1}^{T}\mathbb{E}[x_{t,i}^4]\leq \limsup_{n,T\to \infty}\max_{1\leq i\leq n}\max_{1\leq t\leq T}\mathbb{E}[x_{t,i}^4]<\infty\]
under assumption (b) of the theorem, and the weak convergence is due to Theorem \ref{thm Keleijian}.

The theorem would follow if we show that
\[\frac{T(\bar{y}-\mu_0)^{\intercal}(\tilde{\Sigma}^{-1}_{\mu_0}-\Sigma^{-1})(\bar{y}-\mu_0)}{\sqrt{2n}}=o_p(1).\]
We now show this.
\begin{align*}
&\frac{T|(\bar{y}-\mu_0)^{\intercal}(\tilde{\Sigma}^{-1}_{\mu_0}-\Sigma^{-1})(y_t-\mu_0)|}{\sqrt{2n}}=\frac{\envert[1]{\del[1]{\frac{1}{\sqrt{T}}\sum_{t=1}^{T}(y_t-\mu_0)}^{\intercal}(\tilde{\Sigma}^{-1}_{\mu_0}-\Sigma^{-1})\del[1]{\frac{1}{\sqrt{T}}\sum_{t=1}^{T}(y_t-\mu_0)}}}{\sqrt{2n}}\\
&= \frac{1}{\sqrt{2n}}\envert[3]{\sum_{i=1}^{n}\sum_{j=1}^{n}\del[3]{\frac{1}{\sqrt{T}}\sum_{t=1}^{T}(y_{t,i}-\mu_{0,i})}\del[3]{\frac{1}{\sqrt{T}}\sum_{t=1}^{T}(y_{t,j}-\mu_{0,j})}(\tilde{\Sigma}_{\mu_0,i,j}^{-1}-\Sigma_{i,j}^{-1})}\\
&\leq  \frac{1}{\sqrt{2n}}\del[3]{\max_{1\leq i\leq n}\envert[3]{\frac{1}{\sqrt{T}}\sum_{t=1}^{T}(y_{t,i}-\mu_{0,i})}}^2\sum_{i=1}^{n}\sum_{j=1}^{n}|\tilde{\Sigma}_{\mu_0,i,j}^{-1}-\Sigma_{i,j}^{-1}|\\
&=\frac{1}{\sqrt{2n}}\del[3]{\max_{1\leq i\leq n}\envert[3]{\frac{1}{\sqrt{T}}\sum_{t=1}^{T}(y_{t,i}-\mu_{0,i})}}^2\|\tilde{\Sigma}^{-1}_{\mu_0}-\Sigma^{-1}\|_1\\
&=O_p\del[3]{\frac{\log n}{\sqrt{n}}}\|\Sigma^{-1}\|_1O_{p}\del[3]{\sqrt{\frac{\log^3 n}{n^{2-\beta_1}T}}}=\frac{1}{n^{\beta_2}}\|\Sigma^{-1}\|_1O_{p}\del[3]{\sqrt{\frac{n^{2\beta_2-1}\log^5 n}{n^{2-\beta_1}T}}}\\
&=O_{p}\del[3]{\sqrt{\frac{n^{2\beta_2+\beta_1-3}\log^5 n}{T}}}=o_p(1)
\end{align*}
where the fourth equality is due to Lemma \ref{lemma rate for mu hat}, the sixth equality is due to Assumption \ref{assu summability 2}, and the last equality is due to assumption (a) of the theorem.

\bigskip

For the Wald statistic, write
\begin{align*}
\frac{W_{n,T}-n}{\sqrt{2n}}&=\frac{T(\bar{y}-\mu_0)^{\intercal}\tilde{\Sigma}^{-1}(\bar{y}-\mu_0)-n}{\sqrt{2n}}\\
&=\frac{T(\bar{y}-\mu_0)^{\intercal}\Sigma^{-1}(\bar{y}-\mu_0)-n}{\sqrt{2n}}+\frac{T(\bar{y}-\mu_0)^{\intercal}(\tilde{\Sigma}^{-1}-\Sigma^{-1})(\bar{y}-\mu_0)}{\sqrt{2n}}.
\end{align*}
We have already shown in the proof of the LM test that under assumptions (a)-(b) of the theorem and under $H_0$ as $n,T\to \infty$,
\begin{equation*}
\frac{T(\bar{y}-\mu_0)^{\intercal}\Sigma^{-1}(\bar{y}-\mu_0)-n}{\sqrt{2n}}\xrightarrow{d}N(0,1).
\end{equation*}

Display (\ref{eqn wald weak convergence to standard normal}) would follow if we show that
\[\frac{T(\bar{y}-\mu_0)^{\intercal}(\tilde{\Sigma}^{-1}-\Sigma^{-1})(\bar{y}-\mu_0)}{\sqrt{2n}}=o_p(1).\]
We now show this.
\begin{align*}
&\frac{T|(\bar{y}-\mu_0)^{\intercal}(\tilde{\Sigma}^{-1}-\Sigma^{-1})(y_t-\mu_0)|}{\sqrt{2n}}=\frac{\envert[1]{\del[1]{\frac{1}{\sqrt{T}}\sum_{t=1}^{T}(y_t-\mu_0)}^{\intercal}(\tilde{\Sigma}^{-1}-\Sigma^{-1})\del[1]{\frac{1}{\sqrt{T}}\sum_{t=1}^{T}(y_t-\mu_0)}}}{\sqrt{2n}}\\
&= \frac{1}{\sqrt{2n}}\envert[3]{\sum_{i=1}^{n}\sum_{j=1}^{n}\del[3]{\frac{1}{\sqrt{T}}\sum_{t=1}^{T}(y_{t,i}-\mu_{0,i})}\del[3]{\frac{1}{\sqrt{T}}\sum_{t=1}^{T}(y_{t,j}-\mu_{0,j})}(\tilde{\Sigma}_{i,j}^{-1}-\Sigma_{i,j}^{-1})}\\
&\leq  \frac{1}{\sqrt{2n}}\del[3]{\max_{1\leq i\leq n}\envert[3]{\frac{1}{\sqrt{T}}\sum_{t=1}^{T}(y_{t,i}-\mu_{0,i})}}^2\sum_{i=1}^{n}\sum_{j=1}^{n}|\tilde{\Sigma}_{i,j}^{-1}-\Sigma_{i,j}^{-1}|\\
&=\frac{1}{\sqrt{2n}}\del[3]{\max_{1\leq i\leq n}\envert[3]{\frac{1}{\sqrt{T}}\sum_{t=1}^{T}(y_{t,i}-\mu_{0,i})}}^2\|\tilde{\Sigma}^{-1}-\Sigma^{-1}\|_1\\
&=O_p\del[3]{\frac{\log n}{\sqrt{n}}}\|\Sigma^{-1}\|_1\sbr[3]{O_{p}\del[3]{\sqrt{\frac{\log^3 n}{n^{2-\beta_1}T}}}+O_p\del[3]{\frac{\log^2n}{T}}}\\
&=\frac{1}{n^{\beta_2}}\|\Sigma^{-1}\|_1\sbr[3]{O_{p}\del[3]{\sqrt{\frac{n^{2\beta_2-1}\log^5 n}{n^{2-\beta_1}T}}}+O_p\del[3]{\frac{n^{\beta_2-\frac{1}{2}}\log^3n}{T}}}\\
&=O_{p}\del[3]{\sqrt{\frac{n^{2\beta_2+\beta_1-3}\log^5 n}{T}}}+O_p\del[3]{\frac{n^{\beta_2-\frac{1}{2}}\log^3n}{T}}=o_p(1)
\end{align*}
where the fourth equality is due to Lemma \ref{lemma rate for mu hat} and Theorem \ref{thm quadratic form final}(iv), and the sixth equality is due to Assumption \ref{assu summability 2}.
\end{proof}

\subsection{Proof of Theorem \protect\ref{thm Wald local alternative}}

\begin{proof}
Write
\begin{align*}
&W_{n,T}=T(\bar{y}-\mu_0)^{\intercal}\tilde{\Sigma}^{-1}(\bar{y}-\mu_0)=T\sbr[1]{\bar{y}-\mu_T+\mu_T-\mu_0}^{\intercal}\tilde{\Sigma}^{-1}\sbr[1]{\bar{y}-\mu_T+\mu_T-\mu_0}\notag\\
&=T(\bar{y}-\mu_T)^{\intercal}\tilde{\Sigma}^{-1}(\bar{y}-\mu_T)+2T(\mu_T-\mu_0)^{\intercal}\tilde{\Sigma}^{-1}(\bar{y}-\mu_T)+T(\mu_T-\mu_0)^{\intercal}\tilde{\Sigma}^{-1}(\mu_T-\mu_0)\notag\\
&=:W_{n,T,1}+\theta^{\intercal}\tilde{\Sigma}^{-1}\theta
\end{align*}
whence we have
\begin{align}
\label{align random31}
\frac{W_{n,T}-n}{\sqrt{2n\del[1]{1+\frac{2}{n}\theta^{\intercal}\Sigma^{-1}\theta}}}-\frac{\theta^{\intercal}\Sigma^{-1}\theta}{\sqrt{2n\del[1]{1+\frac{2}{n}\theta^{\intercal}\Sigma^{-1}\theta}}}=\frac{W_{n,T,1}-n}{\sqrt{2n\del[1]{1+\frac{2}{n}\theta^{\intercal}\Sigma^{-1}\theta}}}+\frac{\theta^{\intercal}(\tilde{\Sigma}^{-1}-\Sigma_{-1})\theta}{\sqrt{2n\del[1]{1+\frac{2}{n}\theta^{\intercal}\Sigma^{-1}\theta}}}.
\end{align}
We first consider the first term on the right side of (\ref{align random31}).
\begin{align}
\frac{W_{n,T,1}-n}{\sqrt{2n\del[1]{1+\frac{2}{n}\theta^{\intercal}\Sigma^{-1}\theta}}}&=\frac{T(\bar{y}-\mu_T)^{\intercal}\Sigma^{-1}(\bar{y}-\mu_T)+2T(\mu_T-\mu_0)^{\intercal}\Sigma^{-1}(\bar{y}-\mu_T)-n}{\sqrt{2n\del[1]{1+\frac{2}{n}\theta^{\intercal}\Sigma^{-1}\theta}}}\notag\\
&\qquad+\frac{T(\bar{y}-\mu_T)^{\intercal}(\tilde{\Sigma}^{-1}-\Sigma^{-1})(\bar{y}-\mu_T)+2T(\mu_T-\mu_0)^{\intercal}(\tilde{\Sigma}^{-1}-\Sigma^{-1})(\bar{y}-\mu_T)}{\sqrt{2n\del[1]{1+\frac{2}{n}\theta^{\intercal}\Sigma^{-1}\theta}}}.\label{align random32}
\end{align}
We show that the first term on the right side of (\ref{align random32}) converges in distribution under the local alternatives.
\begin{align*}
&\frac{T(\bar{y}-\mu_T)^{\intercal}\Sigma^{-1}(\bar{y}-\mu_T)+2T(\mu_T-\mu_0)^{\intercal}\Sigma^{-1}(\bar{y}-\mu_T)-n}{\sqrt{2n\del[1]{1+\frac{2}{n}\theta^{\intercal}\Sigma^{-1}\theta}}}\\
&=\frac{\sbr[1]{\frac{1}{\sqrt{T}}\sum_{t=1}^{T}(y_t-\mu_T)}^{\intercal}(L^{-1})^{\intercal}L^{-1}\sbr[1]{\frac{1}{\sqrt{T}}\sum_{t=1}^{T}(y_t-\mu_T)}+2\theta^{\intercal}(L^{-1})^{\intercal}L^{-1}\sbr[1]{\frac{1}{\sqrt{T}}\sum_{t=1}^{T}(y_t-\mu_T)}-n}{\sqrt{2n\del[1]{1+\frac{2}{n}\theta^{\intercal}\Sigma^{-1}\theta}}}\\
&=:\frac{\del[1]{\frac{1}{\sqrt{T}}\sum_{t=1}^{T}x_t}^{\intercal}\del[1]{\frac{1}{\sqrt{T}}\sum_{t=1}^{T}x_t}+2(L^{-1}\theta)^{\intercal}\del[1]{\frac{1}{\sqrt{T}}\sum_{t=1}^{T}x_t}-n}{\sqrt{2n\del[1]{1+\frac{2}{n}\theta^{\intercal}\Sigma^{-1}\theta}}}=:\frac{z_{T}^{\intercal}z_{T}+b_T^{\intercal}z_T-n}{\sqrt{2n\del[1]{1+\frac{2}{n}\theta^{\intercal}\Sigma^{-1}\theta}}}\\
&=\frac{\sum_{i=1}^{n}z_{T,i}^2+\sum_{i=1}^{n}b_{T,i}z_{T,i}-n}{\sqrt{2n\del[1]{1+\frac{2}{n}\theta^{\intercal}\Sigma^{-1}\theta}}}=:\frac{Q_{n,T}-n}{\sqrt{2n\del[1]{1+\frac{2}{n}\theta^{\intercal}\Sigma^{-1}\theta}}}.
\end{align*}
Note that for each $n\geq 1, T\geq 1$, $z_{T,1},\ldots, z_{T,n}$ are (mutually) independent under assumption (b) of the theorem and Assumption \ref{assu normality}(i). Under $H_1$, 
\begin{align*}
\mathbb{E}[z_{T,i}]&=\mathbb{E}\sbr[3]{\frac{1}{\sqrt{T}}\sum_{t=1}^{T}x_{t,i}}=0\\
\var(z_{T})&=\var\del[3]{\frac{1}{\sqrt{T}}\sum_{t=1}^{T}x_{t}}=I_n\\
\mathbb{E}[Q_{n,T}]&=\mathbb{E}\sbr[3]{\sum_{i=1}^{n}z_{T,i}^2+\sum_{i=1}^{n}b_{T,i}z_{T,i}}=\sum_{i=1}^{n}\mathbb{E}\sbr[1]{z_{T,i}^2}=n.
\end{align*}
We next calculate $\var(Q_{n,T})$.
\begin{align*}
&\var(Q_{n,T})=\var\del[3]{\sum_{i=1}^{n}z_{T,i}^2+\sum_{i=1}^{n}b_{T,i}z_{T,i}}=\sum_{i=1}^{n}\var\del[1]{z_{T,i}^2+b_{T,i}z_{T,i}}\\
&=\sum_{i=1}^{n}\mathbb{E}\sbr[1]{z_{T,i}^2+b_{T,i}z_{T,i}-\mathbb{E}z_{T,i}^2}^2=\sum_{i=1}^{n}\mathbb{E}\sbr[1]{z_{T,i}^2+b_{T,i}z_{T,i}-1}^2\\
&=\sum_{i=1}^{n}\sbr[2]{\del[1]{\mathbb{E}[z_{T,i}^4]-1}+2b_{T,i}\mathbb{E}[z_{T,i}^3]+b_{T,i}^2}=\sum_{i=1}^{n}\del[1]{\gamma_{z,i}+2}+2\sum_{i=1}^{n}b_{T,i}\mathbb{E}[z_{T,i}^3]+\sum_{i=1}^{n}b_{T,i}^2.
\end{align*}
In (\ref{eqn random4}), we have already calculated that
\[\sum_{i=1}^{n}\del[1]{\gamma_{z,i}+2}=2n\del[3]{1+\frac{1}{2T}\frac{1}{nT}\sum_{i=1}^{n}\sum_{t=1}^{T}\gamma_{x,t,i}}.\]
We now calculate $\mathbb{E}[z_{T,i}^3]$.
\begin{align*}
&\mathbb{E}[z_{T,i}^3]=\mathbb{E}\sbr[3]{\frac{1}{\sqrt{T}}\sum_{t=1}^{T}x_{t,i}}^3=\mathbb{E}\sbr[3]{\frac{1}{T^{3/2}}\sum_{t=1}^{T}\sum_{s=1}^{T}\sum_{k=1}^{T}x_{t,i}x_{s,i}x_{k,i}}=\frac{1}{T^{3/2}}\sum_{t=1}^{T}\sum_{s=1}^{T}\sum_{k=1}^{T}\mathbb{E}\sbr[1]{x_{t,i}x_{s,i}x_{k,i}}\\
&=\frac{1}{T^{3/2}}\sum_{t=1}^{T}\mathbb{E}\sbr[1]{x_{t,i}^3}.
\end{align*}
Backing up, we have
\begin{align*}
\var(Q_{n,T})=2n\del[3]{1+\frac{1}{2T}\frac{1}{nT}\sum_{i=1}^{n}\sum_{t=1}^{T}\gamma_{x,t,i}}+\frac{2}{T^{3/2}}\sum_{i=1}^{n}b_{T,i}\sum_{t=1}^{T}\mathbb{E}\sbr[1]{x_{t,i}^3}+\sum_{i=1}^{n}b_{T,i}^2.
\end{align*}

We now verify conditions (ii) and (iii) of Theorem \ref{thm Keleijian}. For condition (ii), we have already verified in the proof of Theorem \ref{thm Wald statistics} that $\limsup_{T\to \infty} \sup_{n\geq 1}\sup_{1\leq i\leq n}\mathbb{E}\envert[1]{z_{T,i}}^{4+2\delta}<\infty$. Next,
\begin{align*}
\limsup_{n,T\to \infty}\frac{1}{n}\sum_{i=1}^{n}|b_{T,i}|^{2+\delta}=\limsup_{n\to \infty}\frac{1}{n}\sum_{i=1}^{n}\envert[1]{2(L^{-1}\theta)_i}^{2+\delta}=2^{2+\delta}\limsup_{n\to \infty}\frac{1}{n}\sum_{i=1}^{n}\envert[1]{(L^{-1}\theta)_i}^{2+\delta}<\infty
\end{align*}
via assumption (b) of the theorem. Thus condition (ii) of Theorem \ref{thm Keleijian} is met. Finally,
\begin{align*}
&\frac{1}{n}\var(Q_{n,T})=2\del[3]{1+\frac{1}{2T}\frac{1}{nT}\sum_{i=1}^{n}\sum_{t=1}^{T}\gamma_{x,t,i}}+\frac{2}{nT^{3/2}}\sum_{i=1}^{n}b_{T,i}\sum_{t=1}^{T}\mathbb{E}\sbr[1]{x_{t,i}^3}+\frac{1}{n}\sum_{i=1}^{n}b_{T,i}^2\\
&=2\del[3]{1+\frac{1}{2T}\frac{1}{nT}\sum_{i=1}^{n}\sum_{t=1}^{T}\gamma_{x,t,i}}+\frac{4}{\sqrt{T}}\frac{1}{nT}\sum_{i=1}^{n}\sum_{t=1}^{T}(L^{-1}\theta)_i\mathbb{E}\sbr[1]{x_{t,i}^3}+\frac{4}{n}\theta^{\intercal}\Sigma^{-1}\theta>0
\end{align*}
for large enough $n$ and $T$ because $\gamma_{x,t,i}>-3$ for all $t$ and $i$ by definition of the excess kurtosis. Thus condition (iii) of Theorem \ref{thm Keleijian} is met.

Thus we have
\begin{align*}
&\frac{T(\bar{y}-\mu_T)^{\intercal}\Sigma^{-1}(\bar{y}-\mu_T)+2T(\mu_T-\mu_0)^{\intercal}\Sigma^{-1}(\bar{y}-\mu_T)-n}{\sqrt{2n\del[1]{1+\frac{2}{n}\theta^{\intercal}\Sigma^{-1}\theta}}}=\frac{Q_{n,T}-n}{\sqrt{2n\del[1]{1+\frac{2}{n}\theta^{\intercal}\Sigma^{-1}\theta}}}\\
&= \frac{Q_{n,T}-n}{\sqrt{\var(Q_{n,T})}}\frac{\sqrt{\var(Q_{n,T})}}{\sqrt{2n\del[1]{1+\frac{2}{n}\theta^{\intercal}\Sigma^{-1}\theta}}}= \frac{Q_{n,T}-n}{\sqrt{\var(Q_{n,T})}}(1+o(1))\xrightarrow{d}N\del[1]{0,1}
\end{align*}
as $n,T\to\infty$.

We next show that the second term on the right side of (\ref{align random32}) is $o_p(1)$ under $H_1$
\begin{align}
&\frac{T\envert[1]{(\bar{y}-\mu_T)^{\intercal}(\tilde{\Sigma}^{-1}-\Sigma^{-1})(\bar{y}-\mu_T)}}{\sqrt{2n\del[1]{1+\frac{2}{n}\theta^{\intercal}\Sigma^{-1}\theta}}}+\frac{2T\envert[1]{(\mu_T-\mu_0)^{\intercal}(\tilde{\Sigma}^{-1}-\Sigma^{-1})(\bar{y}-\mu_T)}}{\sqrt{2n\del[1]{1+\frac{2}{n}\theta^{\intercal}\Sigma^{-1}\theta}}}\notag\\
&=\frac{\envert[1]{\del[1]{\frac{1}{\sqrt{T}}\sum_{t=1}^{T}(y_t-\mu_T)}^{\intercal}(\tilde{\Sigma}^{-1}-\Sigma^{-1})\del[1]{\frac{1}{\sqrt{T}}\sum_{t=1}^{T}(y_t-\mu_T)}}}{\sqrt{2n\del[1]{1+\frac{2}{n}\theta^{\intercal}\Sigma^{-1}\theta}}}+\frac{2\envert[1]{\theta^{\intercal}(\tilde{\Sigma}^{-1}-\Sigma^{-1})\del[1]{\frac{1}{\sqrt{T}}\sum_{t=1}^{T}(y_t-\mu_T)}}}{\sqrt{2n\del[1]{1+\frac{2}{n}\theta^{\intercal}\Sigma^{-1}\theta}}}\notag\\
&\leq \frac{\envert[1]{\del[1]{\frac{1}{\sqrt{T}}\sum_{t=1}^{T}(y_t-\mu_T)}^{\intercal}(\tilde{\Sigma}^{-1}-\Sigma^{-1})\del[1]{\frac{1}{\sqrt{T}}\sum_{t=1}^{T}(y_t-\mu_T)}}}{\sqrt{2n}}+\frac{2\envert[1]{\theta^{\intercal}(\tilde{\Sigma}^{-1}-\Sigma^{-1})\del[1]{\frac{1}{\sqrt{T}}\sum_{t=1}^{T}(y_t-\mu_T)}}}{\sqrt{2n}}\notag\\
&\leq \frac{1}{\sqrt{2n}}\del[3]{\max_{1\leq i\leq n}\envert[3]{\frac{1}{\sqrt{T}}\sum_{t=1}^{T}(y_{t,i}-\mu_{T,i})}}^2\|\tilde{\Sigma}^{-1}-\Sigma^{-1}\|_1\notag\\
&\qquad+\sqrt{\frac{2}{n}}\max_{1\leq i\leq n}|\theta_i| \del[3]{\max_{1\leq i\leq n}\envert[3]{\frac{1}{\sqrt{T}}\sum_{t=1}^{T}(y_{t,i}-\mu_{T,i})}}\|\tilde{\Sigma}^{-1}-\Sigma^{-1}\|_1\notag\\
&=O_p\del[3]{\frac{\log n}{\sqrt{n}}}\|\Sigma^{-1}\|_1\sbr[3]{O_{p}\del[3]{\sqrt{\frac{\log^3 n}{n^{2-\beta_1}T}}}+O_p\del[3]{\frac{\log^2n}{T}}}\notag\\
&=\frac{1}{n^{\beta_2}}\|\Sigma^{-1}\|_1\sbr[3]{O_{p}\del[3]{\sqrt{\frac{n^{2\beta_2-1}\log^5 n}{n^{2-\beta_1}T}}}+O_p\del[3]{\frac{n^{\beta_2-\frac{1}{2}}\log^3n}{T}}}\notag\\
&=O_{p}\del[3]{\sqrt{\frac{n^{2\beta_2+\beta_1-3}\log^5 n}{T}}}+O_p\del[3]{\frac{n^{\beta_2-\frac{1}{2}}\log^3n}{T}}=o_p(1)\label{align random33}
\end{align}
where the second equality is due to Lemma \ref{lemma rate for mu hat} and Theorem \ref{thm quadratic form final}(iv), and the fourth equality is due to Assumption \ref{assu summability 2}.

Backing up, in (\ref{align random32}), we hence have under $H_1$ as $n,T\to \infty$
\begin{align*}
\frac{W_{n,T,1}-n}{\sqrt{2n\del[1]{1+\frac{2}{n}\theta^{\intercal}\Sigma^{-1}\theta}}}\xrightarrow{d}N\del[1]{0,1}.
\end{align*}
We finally consider the second term on the right side of (\ref{align random31}), noting that
\begin{align*}
\frac{\theta^{\intercal}(\tilde{\Sigma}^{-1}-\Sigma^{-1})\theta}{\sqrt{2n\del[1]{1+\frac{2}{n}\theta^{\intercal}\Sigma^{-1}\theta}}}=o_p(1)
\end{align*}
using a similar device as that in (\ref{align random33}). Backing up, in (\ref{align random31}), we hence have under $H_1$ as $n,T\to \infty$
\begin{align*}
\frac{W_{n,T}-n}{\sqrt{2n\del[1]{1+\frac{2}{n}\theta^{\intercal}\Sigma^{-1}\theta}}}-\frac{\theta^{\intercal}\Sigma^{-1}\theta}{\sqrt{2n\del[1]{1+\frac{2}{n}\theta^{\intercal}\Sigma^{-1}\theta}}}\xrightarrow{d}N\del[1]{0,1}.
\end{align*}
\end{proof}

\subsection{Proof of Theorem \protect\ref{thm linear restrictions}}

\begin{proof}[Proof of Theorem \ref{thm linear restrictions}]
\begin{align}
&W_{n,T}^*:=T(R\bar{y}-r)^{\intercal}(R\tilde{\Sigma} R^{\intercal})^{-1}(R\bar{y}-r)\notag\\
&=T(R\bar{y}-r)^{\intercal}(R\Sigma R^{\intercal})^{-1}(R\bar{y}-r)-T(R\bar{y}-r)^{\intercal}\sbr[1]{(R\tilde{\Sigma} R^{\intercal})^{-1}-(R\Sigma R^{\intercal})^{-1}}(R\bar{y}-r)\label{align random35}
\end{align}
We now show that the first term of (\ref{align random35}) is asymptotically chi square distributed under $H_0$. Since $R$ has full row rank $q$ and $\lambda_{\min}(\Sigma)$ is bounded away from zero by an absolute positive constant, $R\Sigma R^{\intercal}$ has full rank $q$. Consider the Cholesky decomposition of $R\Sigma R^{\intercal}=L_RL_R^{\intercal}$, where $L_R$ is a $q\times q$ nonsingular lower triangular matrix with positive diagonal elements. Write
\begin{align*}
&T(R\bar{y}-r)^{\intercal}(R\Sigma R^{\intercal})^{-1}(R\bar{y}-r)=T(R\bar{y}-r)^{\intercal}(L_R^{-1})^{\intercal}L_R^{-1}(R\bar{y}-r)\\
&=\sbr[3]{\frac{1}{\sqrt{T}}\sum_{t=1}^{T}L_R^{-1}R(y_t-\mu)}^{\intercal}\sbr[3]{\frac{1}{\sqrt{T}}\sum_{t=1}^{T}L_R^{-1}R(y_t-\mu)}.
\end{align*}
Note that $L_R^{-1}R(y_1-\mu), L_R^{-1}R(y_2-\mu),\ldots, L_R^{-1}R(y_T-\mu)$ are independent random vectors in $\mathbb{R}^q$ with mean zero and variance matrix $I_q$. Then we can invoke a version of the multivariate central limit theorem to show $T^{-1/2}\sum_{t=1}^{T}L_R^{-1}R(y_t-\mu)\xrightarrow{d}N(0,I_q)$ as $n,T\to \infty$, whence we have $T(R\bar{y}-r)^{\intercal}(R\Sigma R^{\intercal})^{-1}(R\bar{y}-r)\xrightarrow{d}\chi^2_q$ as $n,T\to \infty$.

We now show that the second term of (\ref{align random35}) is $o_p(1)$ under $H_0$.
\begin{align*}
&\envert[2]{T(R\bar{y}-r)^{\intercal}\sbr[1]{(R\tilde{\Sigma} R^{\intercal})^{-1}-(R\Sigma R^{\intercal})^{-1}}(R\bar{y}-r)}\\
&=\envert[3]{\sum_{i=1}^{q}\sum_{j=1}^{q}\del[3]{\frac{1}{\sqrt{T}}\sum_{t=1}^{T}[R(y_{t}-\mu)]_i}\del[3]{\frac{1}{\sqrt{T}}\sum_{t=1}^{T}[R(y_{t}-\mu)]_j}\sbr[1]{(R\tilde{\Sigma} R^{\intercal})_{i,j}^{-1}-(R\Sigma R^{\intercal})_{i,j}^{-1}}}\\
&\leq \del[3]{\max_{1\leq i\leq q}\envert[3]{\frac{1}{\sqrt{T}}\sum_{t=1}^{T}[R(y_{t}-\mu)]_i}}^2 \enVert[1]{(R\tilde{\Sigma} R^{\intercal})^{-1}-(R\Sigma R^{\intercal})^{-1}}_1\\
&=O_p(1)\enVert[1]{(R\tilde{\Sigma} R^{\intercal})^{-1}-(R\Sigma R^{\intercal})^{-1}}_1.
\end{align*}
We need to find a rate for $\enVert[1]{(R\tilde{\Sigma} R^{\intercal})^{-1}-(R\Sigma R^{\intercal})^{-1}}_1$. First, note that
\begin{align*}
&\enVert[1]{R\tilde{\Sigma} R^{\intercal}-R\Sigma R^{\intercal}}_1\leq q^{3/2}\enVert[1]{R(\tilde{\Sigma} -\Sigma) R^{\intercal}}_{\ell_2}\leq q^{3/2}\|R\|_{\ell_2}\|\tilde{\Sigma} -\Sigma\|_{\ell_2} \|R^{\intercal}\|_{\ell_2}\\
&=q^{3/2}\|R^{\intercal}\|_{\ell_2}^2\|\tilde{\Sigma} -\Sigma\|_{\ell_2}=q^{3/2}\lambda_{\max}(RR^{\intercal})\|\tilde{\Sigma} -\Sigma\|_{\ell_2}\\
&=q^{3/2}\lambda_{\max}(RR^{\intercal})\|\Sigma\|_{\ell_2}\sbr[3]{O_{p}\del[3]{\sqrt{\frac{\log^3 n}{n^{2-\beta_1}T}}}+O_p\del[3]{\frac{\log^2 n}{T}}}=o_p(1),
\end{align*}
where the second to last equality is due to Theorem \ref{thm quadratic form final}(v), and the last equality is due to (\ref{eqn random2}). Second, 
\begin{align*}
&\enVert[1]{(R\Sigma R^{\intercal})^{-1}}_1\leq q^{3/2}\enVert[1]{(R\Sigma R^{\intercal})^{-1}}_{\ell_2}=q^{3/2}\lambda_{\max}\sbr[1]{(R\Sigma R^{\intercal})^{-1}}=\frac{q^{3/2}}{\lambda_{\min}\sbr[1]{R\Sigma R^{\intercal}}}\leq \frac{q^{3/2}}{\lambda_{\min}(R R^{\intercal})\lambda_{\min}(\Sigma )}\\
&=O(1)
\end{align*}
where the last inequality is due to Lemma \ref{lemmasandwich eigenvalue} in Section \ref{sec AppendixB} and the last equality is due to the assumption of the theorem. Then via Lemma \ref{lemma saikkonen lemma} in Section \ref{sec AppendixB} we have
\begin{align*}
\enVert[1]{(R\tilde{\Sigma} R^{\intercal})^{-1}-(R\Sigma R^{\intercal})^{-1}}_1=q^{3/2}\lambda_{\max}(RR^{\intercal})\|\Sigma\|_{\ell_2}\sbr[3]{O_{p}\del[3]{\sqrt{\frac{\log^3 n}{n^{2-\beta_1}T}}}+O_p\del[3]{\frac{\log^2 n}{T}}}=o_p(1).
\end{align*}
Backing up, we have proved that the second term of (\ref{align random35}) is $o_p(1)$ under $H_0$.
\end{proof}

\subsection{Proof of Lemma \protect\ref{lemma simu confi}}

\begin{proof}[Proof of Lemma \ref{lemma simu confi}]
The assumptions of the lemma allow us to invoke Theorem \ref{thm Wald statistics}. Thus under $H_{0}:\mu =\mu_0$, as $n,T\rightarrow \infty $, 
 \begin{equation*}
 \frac{W_{n,T}-n}{\sqrt{2n}}=\frac{T(\bar{y}-\mu_0)^{\intercal }\tilde{\Sigma}^{-1}%
 (\bar{y}-\mu_0)-n}{\sqrt{2n}}\xrightarrow{d}N(0,1).
 \end{equation*}
This implies that for any unknown $\mu$
\begin{align*}
\mathbb{P}_{\mu}\del[3]{\frac{T(\bar{y}-\mu)^{\intercal }\tilde{\Sigma}^{-1}%
 (\bar{y}-\mu)-n}{\sqrt{2n}}<z_{\alpha}}\to 1-\alpha
\end{align*}
as $n,T\to \infty$, where $z_{\alpha}$ is the upper $\alpha$ percentile of $N(0,1)$.

Invoking Lemma \ref{lemma general cauchy} in Section \ref{sec AppendixB} with $x=\bar{y}-\mu$ and $S=\tilde{\Sigma}$ yields: For any $\phi \in \mathbb{R}^n$
\begin{align*}
\sbr[1]{\phi^{\intercal}(\bar{y}-\mu)}^2\leq \phi^{\intercal}\tilde{\Sigma}\phi \cdot (\bar{y}-\mu)^{\intercal}\tilde{\Sigma}^{-1}(\bar{y}-\mu)
\end{align*}
whence we have
\begin{align*}
\frac{\sbr[1]{\phi^{\intercal}(\bar{y}-\mu)}^2}{\phi^{\intercal}\tilde{\Sigma}\phi}\leq (\bar{y}-\mu)^{\intercal}\tilde{\Sigma}^{-1}(\bar{y}-\mu).
\end{align*}
Multiply both sides by $T$, minus $n$, and divide by $\sqrt{2n}$:
\begin{align*}
\frac{T\sbr[1]{\phi^{\intercal}(\bar{y}-\mu)}^2/\phi^{\intercal}\tilde{\Sigma}\phi-n}{\sqrt{2n}}\leq \frac{T(\bar{y}-\mu)^{\intercal}\tilde{\Sigma}^{-1}(\bar{y}-\mu)-n}{\sqrt{2n}}.
\end{align*}
Thus we assert with confidence $1-\alpha$ that the unknown $\mu$ satisfies \textit{simultaneously for all} $\phi$ the inequalities:
\[\frac{T\sbr[1]{\phi^{\intercal}(\bar{y}-\mu)}^2/\phi^{\intercal}\tilde{\Sigma}\phi-n}{\sqrt{2n}}<z_{\alpha},\]
as $n,T\to \infty$.
\end{proof}

\section{Auxiliary Lemmas}

\label{sec AppendixB}

\begin{lemma}
\label{lemma rate for mu hat} Suppose Assumption \ref{assu normality}%
(i)-(ii) hold. Then we have 
\begin{equation*}
\max_{1\leq i\leq n}\envert[3]{\frac{1}{\sqrt{T}}\sum_{t=1}^{T}%
\del[1]{y_{t,i}-\mathbb{E}y_{t,i}}}=O_{p}(\sqrt{\log n}).
\end{equation*}
\end{lemma}

\begin{proof}
Under Assumption \ref{assu normality}(ii), we have, for $i=1,\ldots, n$, $m=2,3,\ldots$, 
\begin{align*}
&\frac{1}{T}\sum_{t=1}^{T}\mathbb{E}\envert[1]{y_{t,i}-\mathbb{E}y_{t,i}}^m\leq \frac{1}{T}\sum_{t=1}^{T}2^{m-1}\del[1]{\mathbb{E}|y_{t,i}|^m+\mathbb{E}|\mathbb{E}y_{t,i}|^m}\leq \frac{1}{T}\sum_{t=1}^{T}2^{m-1}\del[1]{\mathbb{E}|y_{t,i}|^m+\mathbb{E}|y_{t,i}|^m}\\
&=2^m\frac{1}{T}\sum_{t=1}^{T}\mathbb{E}|y_{t,i}|^m\leq 2^mA^m \leq 2m!A^m=\frac{m!}{2}A^{m-2}A^24
\end{align*}
for some absolute positive constant $A$. Now invoke the Bernstein's inequality in Section \ref{sec AppendixB} with $\sigma_0^2=4A^2$: 
For all $\epsilon>0$
\[\mathbb{P}\del[3]{ \envert[3]{ \frac{1}{T}\sum_{t=1}^{T}(y_{t,i}-\mathbb{E}y_{t,i})}\geq \sigma_0^2\left[ A\epsilon+\sqrt{2\epsilon}\right] }\leq 2e^{-T\sigma_0^2\epsilon}. \]
Invoking Corollary \ref{coro Bernstein follow up rate trick} in Section \ref{sec AppendixB}, we have
\[\max_{1\leq i\leq n}\envert[3]{ \frac{1}{T}\sum_{t=1}^{T}(y_{t,i}-\mathbb{E}y_{t,i})}=O_p\del[3]{\frac{\log n}{T}\vee \sqrt{\frac{\log n}{T}}}=O_p\del[3]{ \sqrt{\frac{\log n}{T}}}.\]
The lemma follows.
\end{proof}

\bigskip

We next give two central limit theorems for double-index ($n,T$) processes. 

\begin{thm}
\label{thmdoubleindexCLT}

\item 
\begin{enumerate}[(a)]

\item Suppose $Y_{n, t}$ is a random variable independent across $1\leq
t\leq T$ for $n\geq 1$ and $T\geq 1$. Assume that 
\begin{equation*}
\mathbb{E}[Y_{n,t}]=0\qquad\mathbb{E}[Y_{n,T,t}^{2}]=\sigma_{n,t}^{2}.
\end{equation*}
Define 
\begin{equation*}
s_{n,T}^{2}:=\sum_{t=1}^{T}\sigma_{n,t}^{2}\qquad\xi_{n,T,t}:=\frac{Y_{n,t}}{%
s_{n,T}}.
\end{equation*}
Assume that $s_{n,T}^2>0$ for large enough $n$ and $T$. Suppose the
following Lyapounov's condition holds: For some $\delta>0$, 
\begin{equation*}
\lim_{n,T\to\infty}\sum_{t=1}^{T}\frac{1}{s_{n,T}^{2+\delta}}\mathbb{E} %
\envert[1]{Y_{n,t}}^{2+\delta}=0.
\end{equation*}
Then as $n,T\to\infty$ 
\begin{equation*}
\sum_{t=1}^{T}\xi_{n,T,t}\xrightarrow{d}N(0,1).
\end{equation*}

\item Suppose $Y_{T, i}$ is a random variable independent across $1\leq
i\leq n$ for $n\geq 1$ and $T\geq 1$. Assume that 
\begin{equation*}
\mathbb{E}[Y_{T,i}]=0\qquad\mathbb{E}[Y_{T,i}^{2}]=\sigma_{T,i}^{2}.
\end{equation*}
Define 
\begin{equation*}
s_{n,T}^{2}:=\sum_{i=1}^{n}\sigma_{T,i}^{2}\qquad\xi_{n,T,i}:=\frac{Y_{T,i}}{%
s_{n,T}}.
\end{equation*}
Assume that $s_{n,T}^2>0$ for large enough $n$ and $T$. Suppose the
following Lyapounov's condition holds: For some $\delta>0$, 
\begin{equation*}
\lim_{n,T\to\infty}\sum_{i=1}^{n}\frac{1}{s_{n,T}^{2+\delta}}\mathbb{E}%
\envert[1]{Y_{T,i}}^{2+\delta}=0.
\end{equation*}
Then as $n,T\to\infty$ 
\begin{equation*}
\sum_{i=1}^{n}\xi_{n,T,i}\xrightarrow{d}N(0,1).
\end{equation*}
\end{enumerate}
\end{thm}

\begin{proof}
The proofs can be easily adapted from Lyapounov's condition for triangular arrays (cf. page 362 \cite{billingsley1995}).
\end{proof}

\bigskip

\begin{thm}[Bernstein's inequality]
We let $Z_{1},\ldots, Z_{T}$ be independent random variables, satisfying for
absolute positive constants $A$ and $\sigma_{0}^{2}$ 
\begin{equation*}
\mathbb{E}Z_{t}=0 \quad\forall t, \quad\frac{1}{T}\sum_{t=1}^{T}\mathbb{E}%
|Z_{t}|^{m}\leq\frac{m!}{2}A^{m-2}\sigma_{0}^{2}, \quad m=2,3,\ldots.
\end{equation*}
Let $\epsilon>0$ be arbitrary. Then 
\begin{equation*}
\mathbb{P}\del[3]{\envert[3]{ \frac{1}{T}\sum_{t=1}^{T}Z_t } \geq
\sigma_0^2\left[ A\epsilon+\sqrt{2\epsilon}\right]}\leq
2e^{-T\sigma_{0}^{2}\epsilon}.
\end{equation*}
\end{thm}

\begin{proof}
Slightly adapted from \cite{buhlmannvandegeer2011} p487.
\end{proof}

\bigskip

We can use Bernstein's inequality to establish a rate for the maximum.

\begin{corollary}
\label{coro Bernstein follow up rate trick} Suppose via Bernstein's
inequality that we have for $1\leq i\leq n$, 
\begin{equation*}
\mathbb{P}\del[4]{\envert[3]{ \frac{1}{T}\sum_{t=1}^{T}Z_{t,i} } \geq
\sigma^2_0\sbr[1]{ K\epsilon+\sqrt{2\epsilon}}}\leq
2e^{-T\sigma^{2}_{0}\epsilon}.
\end{equation*}
for some absolute positive constants $K$ and $\sigma_{0}^{2}$. Then 
\begin{equation*}
\max_{1\leq i\leq n}\envert[3]{ \frac{1}{T}\sum_{t=1}^{T}Z_{t,i} }=O_{p}%
\del[3]{\frac{\log n}{T}\vee \sqrt{\frac{\log n}{T}}}.
\end{equation*}
\end{corollary}

\begin{proof}
We need to use joint asymptotics $n,T\to \infty$. We shall use the preceding inequality with $\epsilon=(2\log n)/(T\sigma_0^2)$. Fix $\varepsilon>0$. These exist $N_{\varepsilon}:=2/\varepsilon$, $T_{\varepsilon}$ and $M_{\varepsilon}:=\max (4K, 4\sigma_0)$ such that for all $n>N_{\varepsilon}$ and $T>T_{\varepsilon}$ we have
\begin{align*}
&\mathbb{P}\del[4]{\max_{1\leq i\leq n}\envert[3]{ \frac{1}{T}\sum_{t=1}^{T}Z_{t,i} } \geq M_{\varepsilon}\del[3]{\frac{\log n}{T}\vee \sqrt{\frac{\log n}{T}}}}\\
&\leq \sum_{i=1}^{n}\mathbb{P}\del[4]{\envert[3]{ \frac{1}{T}\sum_{t=1}^{T}Z_{t,i} } \geq \sigma^2_0\sbr[1]{ K\epsilon+\sqrt{2\epsilon}}}\leq 2e^{\log n-2\log n}=\frac{2}{n}< \varepsilon.
\end{align*}
\end{proof}

\bigskip

\begin{lemma}
\label{lemmasandwich eigenvalue} Suppose matrix $A$ is real symmetric. Then
for any comparable real matrix $B$ 
\begin{equation*}
\lambda_{\min}\del[1]{A}\lambda_{\min}\del[1]{BB^{\intercal}}\leq\lambda
_{\min}\del[1]{BAB^{\intercal}}\leq\lambda_{\max}\del[1]{BAB^{\intercal}}%
\leq \lambda_{\max}\del[1]{A}\lambda_{\max}\del[1]{BB^{\intercal}}.
\end{equation*}
\end{lemma}

\begin{proof}
First, note that $BAB^{\intercal}$ is Hermitian. By Rayleigh-Ritz theorem, we have
\begin{align*}
& \lambda_{\max}\del[1]{BAB^{\intercal}}=\max_{\|c\|_2=1}c^{\intercal}BAB^{\intercal}c\leq \max_{\|c\|_2=1} \lambda_{\max}(A)\|B^{\intercal}c\|^2=\lambda_{\max}(A)\max_{\|c\|_2=1} c^{\intercal}BB^{\intercal}c\\
& = \lambda_{\max}\del[1]{A}\lambda_{\max}\del[1]{BB^{\intercal}}.
\end{align*}
On the other hand,
\begin{align*}
& \lambda_{\min}\del[1]{BAB^{\intercal}}=\min_{\|c\|_2=1}c^{\intercal}BAB^{\intercal}c\geq \min_{\|c\|_2=1} \lambda_{\min}(A)\|B^{\intercal}c\|^2=\lambda_{\min}(A)\min_{\|c\|_2=1} c^{\intercal}BB^{\intercal}c\\
& = \lambda_{\min}\del[1]{A}\lambda_{\min}\del[1]{BB^{\intercal}}.
\end{align*}
\end{proof}







\begin{lemma}
\label{lemma F norm Kronecker} For any real matrices $A$ and $B$,

\begin{enumerate}[(i)]

\item 
\begin{equation*}
\|A\otimes B\|_{F}=\|A\|_{F}\times\|B\|_{F}.
\end{equation*}

\item 
\begin{equation*}
\|A\otimes B\|_{\ell_{2}}=\|A\|_{\ell_{2}}\times\| B\|_{\ell_{2}}.
\end{equation*}

\item 
\begin{equation*}
\|A\otimes B\|_{1}=\|A\|_{1}\times\|B\|_{1}.
\end{equation*}
\end{enumerate}
\end{lemma}

\begin{proof}
For part (i),
\begin{align*}
\|A\otimes B\|_F^2=\tr \sbr[1]{(A^{\intercal}\otimes B^{\intercal})(A\otimes B)}=\tr \sbr[1]{A^{\intercal}A\otimes B^{\intercal}B}=\tr (A^{\intercal}A)\tr ( B^{\intercal}B)=\|A\|_F^2\|B\|_F^2.
\end{align*}
For part (ii),
\begin{align*}
&\|A\otimes B\|_{\ell_2}=\sqrt{\text{maxeval}[(A\otimes B)^{\intercal}(A\otimes B)]}=\sqrt{\text{maxeval}[(A^{\intercal}\otimes B^{\intercal})(A\otimes B)]}\\
&=\sqrt{\text{maxeval}[A^{\intercal}A\otimes B^{\intercal}B]}=\sqrt{\text{maxeval}[A^{\intercal}A]\text{maxeval}[B^{\intercal}B]}=\|A\|_{\ell_2}\| B\|_{\ell_2},
\end{align*}
where the fourth equality is due to the fact that both $A^{\intercal}A$ and $B^{\intercal}B$ are symmetric and positive semidefinite.
For part (iii), suppose that $A$ is $m\times n$ and $B$ is $p\times q$.
\begin{align*}
\|A\otimes B\|_1 &=\sum_{i=1}^{m}\sum_{j=1}^{n}\del[1]{|a_{i,j}|\|B\|_1}=\sum_{i=1}^{m}\sum_{j=1}^{n}\del[3]{|a_{i,j}|\sum_{k=1}^{p}\sum_{\ell=1}^{q}|b_{k,\ell}|}=\del[3]{\sum_{i=1}^{m}\sum_{j=1}^{n}|a_{i,j}|}\del[3]{\sum_{k=1}^{p}\sum_{\ell=1}^{q}|b_{k,\ell}|}\\
&=\|A\|_1\|B\|_1.
\end{align*}
\end{proof}

\bigskip

\begin{lemma}
\label{lemma saikkonen lemma} Let $\hat{\Omega}_{n,j}$ and $\Omega_{n,j}$ be
invertible (both possibly stochastic) $n\times n$ square matrices for $%
j=1,\ldots,m$, where both $n$ and $m$ could be growing. Let $T$ be the
sample size. For any matrix norm $\|\cdot\|$, suppose that $\max_{1\leq
j\leq m}\|\Omega_{n,j}^{-1}\|=O_{p}(1)$ and $\max_{1\leq j\leq m}\|\hat{%
\Omega }_{n,j}-\Omega_{n,j}\|=O_{p}(a_{m,n,T})$ for some sequence $a_{m,
n,T} $ with $a_{m,n,T}\to0$ as $m,n, T\to\infty$ simultaneously. Then $%
\max_{1\leq j\leq m}\|\hat{\Omega}_{n,j}^{-1}-\Omega_{n,j}^{-1}%
\|=O_{p}(a_{m,n,T})$.
\end{lemma}

\begin{proof}
The original proof could be found in \cite{saikkonenlutkepohl1996} Lemma A.2.
\begin{align*}
\|\hat{\Omega}^{-1}_{n,j}-\Omega^{-1}_{n,j}\|\leq \|\hat{\Omega}_{n,j}^{-1}\|\|\Omega_{n,j}-\hat{\Omega}_{n,j}\|\|\Omega^{-1}_{n,j}\|\leq \del[1]{\|\Omega_{n,j}^{-1}\|+\|\hat{\Omega}_{n,j}^{-1}-\Omega_{n,j}^{-1}\|}\|\Omega_{n,j}-\hat{\Omega}_{n,j}\|\|\Omega^{-1}_{n,j}\|.
\end{align*}
Let $v_{j,n,T}$, $z_{j,n,T}$ and $x_{j,n,T}$ denote $\|\Omega_{j,n}^{-1}\|$, $\|\hat{\Omega}_{j,n}^{-1}-\Omega_{j,n}^{-1}\|$ and $\|\Omega_{j,n}-\hat{\Omega}_{j,n}\|$, respectively. From the preceding equation, we have
\[w_{j,n,T}:=\frac{z_{j,n,T}}{(v_{j,n,T}+z_{j,n,T})v_{j,n,T}}\leq x_{j,n,T},\]
whence we have $\max_{1\leq j\leq m}w_{j,n,T}\leq \max_{1\leq j\leq m}x_{j,n,T}=O_p(a_{m,n,T})=o_p(1)$. We now solve for $z_{j,n,T}$:
\[z_{j, n,T}=\frac{v_{j,n,T}^2w_{j,n,T}}{1-v_{j,n,T}w_{j,n,T}}.\]
Then we have
\begin{align*}
&\max_{1\leq j\leq m}z_{j, n,T}=\max_{1\leq j\leq m}\frac{v_{j,n,T}^2w_{j,n,T}}{1-v_{j,n,T}w_{j,n,T}}=\frac{\max_{1\leq j\leq m}v_{j,n,T}^2\max_{1\leq j\leq m}w_{j,n,T}}{1-\max_{1\leq j\leq m}v_{j,n,T}\max_{1\leq j\leq m}w_{j,n,T}}=O_p(a_{m,n,T})
\end{align*}
where the second equality is due to the fact that $0\leq v_{j,n,T}w_{j,n,T}\leq 1$ for any $j$.
\end{proof}

\bigskip

\bigskip

\begin{thm}
\label{thmBrillinger} Let $\{x_{t,i}\}$ be a double-index process having
zero mean and being independent across $1\leq t\leq T$ for $n\geq 1$ and $%
T\geq 1$. If there exists $k$, $k\geq 2$, such that 
\begin{equation*}
\max_{n\geq 1}\max_{1\leq i\leq n}\max_{T\geq 1}\max_{1\leq t\leq T}\mathbb{E%
}|x_{t,i}|^{k}<\infty,
\end{equation*}
then we have 
\begin{equation*}
\max_{n\geq 1}\max_{1\leq i\leq n}\max_{T\geq 1}\mathbb{E}%
\envert[3]{\frac{1}{\sqrt{T}}\sum_{t=1}^{T}x_{t,i}}^{k}\leq K
\end{equation*}
for some absolute positive constant $K$.
\end{thm}

\begin{proof}
Slightly adapted from \cite{brilllinger1962}.
\end{proof}

\bigskip

\begin{lemma}[Generalised Cauchy-Schwarz Inequality]
\label{lemma general cauchy} For a positive definite matrix $S$ and any
vectors $\phi$ and $x$ 
\begin{equation*}
(\phi^{\intercal}x)^2\leq \phi^{\intercal}S\phi \cdot x^{\intercal}S^{-1}x.
\end{equation*}
\end{lemma}

\begin{proof}
See Lemma 5.3.2 (p178) of \cite{anderson1984}.
\end{proof}

\bibliographystyle{ecta}
\bibliography{HLpartialm}





\end{document}